\documentclass[hidelinks,11pt, letterpaper]{article}
\author{Peter Keevash \thanks{Mathematical Institute, University of Oxford, UK. Supported partially by ERC Consolidator Grant 647678.}
\qquad Noam Lifshitz \thanks{Einstein Institute of Mathematics, Hebrew University, Jerusalem, Israel.}
\qquad Eoin Long \thanks{School of Mathematics, University of Birmingham, Birmingham, UK.}
\qquad Dor Minzer \thanks{Department of Mathematics, Massachusetts Institute of Technology, Cambridge, USA.}}
\date{\vspace{-5ex}}
\usepackage{fullpage}
\usepackage{amsthm}
\usepackage{amsmath,amssymb,amsfonts,nicefrac}
\usepackage{xspace}
\usepackage{color}
\usepackage{url}
\usepackage{hyperref}
\usepackage{bm}
\usepackage{bbm}
\usepackage{times}
\usepackage[T1]{fontenc}

\usepackage{enumitem}

\newtheorem{thm}{Theorem}[section]

\newtheorem{lemma}[thm]{Lemma}
\newtheorem{corollary}[thm]{Corollary}
\newtheorem{claim}[thm]{Claim}
\newtheorem{proposition}[thm]{Proposition}

\theoremstyle{definition}
\newtheorem{definition}[thm]{Definition}
\newtheorem{remark}[thm]{Remark}
\newtheorem{example}[thm]{Example}

\newtheorem{fact}[thm]{Fact}

\newcommand\E{\mathop{\mathbb{E}}}
\newcommand\Var{\mathop{\text{Var}}}
\newcommand\card[1]{\left| {#1} \right|}
\newcommand\sett[2]{\left\{ \left. #1 \;\right\vert #2 \right\}}

\newcommand\set[1]{{\left\{ #1 \right\}}}
\newcommand\Prob[2]{{\Pr_{#1}\left[ {#2} \right]}}

\newcommand\Expect[2]{{\mathop{\mathbb{E}}_{#1}\left[ {#2} \right]}}

\newcommand\cExpect[3]{{\mathop{\mathbb{E}}_{#1}\left[ \left. #3 \;\right\vert #2 \right]}}
\newcommand\norm[1]{\left\| #1 \right\|}
\newcommand\power[1]{\set{0,1}^{#1}}
\newcommand\ceil[1]{\lceil{#1}\rceil}
\newcommand\round[1]{\left\lfloor{#1}\right\rfloor}
\newcommand\half{{1\over2}}

\newcommand\skipi{{\vskip 10pt}}

\newcommand\inner[2]{\langle{#1},{#2}\rangle}
\newcommand\eps{\varepsilon}

\renewcommand\geq{\geqslant}
\renewcommand\leq{\leqslant}

\newcommand{\T}{\mathrm{T}}
\renewcommand{\L}{\mathrm{L}}

\makeatletter
\newtheorem*{rep@theorem}{\rep@title}
\newcommand{\newreptheorem}[2]{%
\newenvironment{rep#1}[1]{%
\def\rep@title{\bf #2 \ref*{##1} \text{(Restated)} }%
\begin{rep@theorem} }%
{\end{rep@theorem} } }

\newtheorem*{rep@claim}{\rep@title}
\newcommand{\newrepclaim}[2]{%
\newenvironment{rep#1}[1]{%
\def\rep@title{\bf #2 \ref*{##1} \text{(Restated)} }%
\begin{rep@claim} }%
{\end{rep@claim} } }

\newtheorem*{rep@lemma}{\rep@title}
\newcommand{\newreplemma}[2]{%
\newenvironment{rep#1}[1]{%
\def\rep@title{\bf #2 \ref*{##1} \text{(Restated)} }%
\begin{rep@lemma} }%
{\end{rep@lemma} } }

\makeatother
\newreptheorem{theorem}{Theorem}
\newrepclaim{claim}{Claim}
\newreplemma{lemma}{Lemma}

\newcommand{\mc}[1]{\mathcal{#1}}
\newcommand{\mb}[1]{\mathbb{#1}}

\newcommand{\sub}{\subset}
\newcommand{\subn}{\subsetneq}

\newcommand{\ua}{\uparrow}
\newcommand{\ra}{\rightarrow}

\newcommand{\sm}{\setminus}
\newcommand{\es}{\emptyset}
\newcommand{\pl}{\partial}
\newcommand{\ov}{\overline}

\newcommand{\wt}{\widetilde}

\newcommand{\aA}{\alpha}
\newcommand{\bB}{\beta}
\newcommand{\gG}{\gamma}
\newcommand{\dD}{\delta}

\newcommand{\lL}{\lambda}

\newcommand{\sS}{\sigma}

\newcommand{\OO}{\Omega}

\title{Forbidden intersections for codes}
\begin{document}
\maketitle
\begin{abstract}
Determining the maximum size of a $t$-intersecting
code in $[m]^n$ was a longstanding open problem
of Frankl and F\"uredi, solved independently
by Ahlswede and Khachatrian
and by Frankl and Tokushige.
We extend their result to the setting
of forbidden intersections, by showing that for any $m>2$
and $n$ large compared with $t$ (but not necessarily $m$)
that the same bound holds for codes with the weaker
property of being $(t-1)$-avoiding, i.e.\ having
no two vectors that agree on exactly $t-1$ coordinates.
Our proof proceeds via a junta approximation result
of independent interest, which we prove via a development
of our recent theory of global hypercontractivity:
we show that any $(t-1)$-avoiding code
is approximately contained in a $t$-intersecting junta
(a code where membership is determined
by a constant number of coordinates).
In particular, when $t=1$ this gives an alternative proof
of a recent result of Eberhard, Kahn, Narayanan and Spirkl
that symmetric intersecting codes in $[m]^n$
have size $o(m^n)$.
\end{abstract}

\section{Introduction}

Many intersection problems for finite sets
(see the survey \cite{frankl2016survey})
have natural generalisations to a setting
variously described as codes, vectors or integer sequences.
For example, any intersecting family
of subsets of $[n]$ has size at most $2^{n-1}$,
and more generally any intersecting code in $[m]^n$
has size at most $m^{n-1}$, where we say a code
${\cal F} \sub [m]^n$ is intersecting if
for any $x,y$ in ${\cal F}$ there is some $i$
with $x_i=y_i$. However, these settings
are quite different, in that there are
many maximum intersecting families of sets,
including very symmetric examples
such as the family of all sets of size $>n/2$,
whereas in $[m]^n$ for $m>2$ the only example
is obtained by fixing one coordinate to have a fixed value.
A more substantial difference was recently demonstrated
by Eberhard, Kahn, Narayanan and Spirkl \cite{EKNS},
who showed that adding a symmetry assumption
reduces the maximum size to $o(m^n)$.

A longstanding open problem of Frankl and F\"uredi \cite{FF}
posed the corresponding question for codes
${\cal F} \sub [m]^n$ that are \emph{$t$-intersecting},
in that any $x,y$ in ${\cal F}$ have \emph{agreement}
${\sf agr}(x,y)=|\{i:x_i=y_i\}| \ge t$.
From the perspective of coding theory,
one may think of such ${\cal F}$ as an `anti-code',
in that we are imposing an upper bound
on the Hamming distance between any two of its vectors.
From a combinatorial perspective, the natural analogy
is with $t$-intersecting $k$-graphs ($k$-uniform hypergraphs),
for which the extremal question was also a longstanding
open problem, posed by Erd{\H o}s, Ko and Rado \cite{EKR}
and finally resolved by the Complete Intersection Theorem
of Ahlswede and Khachatrian \cite{AKcomplete}.
The analogous result for codes, resolving the problem
of Frankl and F\"uredi, was also obtained by
Ahlswede and Khachatrian \cite{AKcodes},
and independently by Frankl and Tokushige \cite{FTcodes}.
They showed that the maximum size
of a $t$-intersecting code in $[m]^n$
is achieved by one of the following natural examples,
which can be thought of as Hamming balls
on a subset of the coordinates,
and which we will simply call `balls'
(following \cite{PachTardos}): let
\[ {\cal S}_{t,r}[m]^n = \{ x\in[m]^n
: |\{ j\in [1,t+2r] : x_j = 1 \}| \geq t+r \}.\]

We show for any $m>2$
and $n$ large compared with $t$ (but not necessarily $m$)
that the same conclusion holds under the weaker assumption
that ${\cal F}$ is \emph{$(t-1)$-avoiding}, i.e.\
no $x,y$ in ${\cal F}$ have agreement $t-1$.

\begin{thm}\label{thm:main}
For all $t\in\mathbb{N}$ there is $n_0 \in\mathbb{N}$ such that
if $\mathcal{F} \sub [m]^n$ is a $(t-1)$-avoiding code
with $m\geq 3$ and $n\geq n_0$ then
$|\mathcal{F}| \leq \max_{r \ge 0} |{\cal S}_{t,r}[m]^n|$
with equality only when ${\cal F}$ is isomorphic to a ball.
\end{thm}

Theorem \ref{thm:main}
can be viewed as an analogue for codes of the
classical forbidden intersection problem for set systems,
which has a substantial literature, particularly stemming
from the many applications of the celebrated Frankl-R\"odl
theorem \cite{frankl1987forbidden}
(see also \cite{frankl1985forbidding,KeevashLongFranklRodl}).
Our proof (discussed in the next subsection)
proceeds via a junta approximation result
of independent interest,
showing that any $(t-1)$-avoiding code
is approximately contained in a $t$-intersecting junta
(a code where membership is determined
by a constant number of coordinates).
In particular, when $t=1$ this gives an alternative proof
of the result of \cite{EKNS}, as a family that essentially
depends on few coordinates is very far from being symmetric.

\subsection{Overview of the proof}

The proof of Theorem \ref{thm:main} has three steps,
each of which has elements of independent interest.
\begin{enumerate}
\item [(1)] Junta approximation:
any $(t-1)$-avoiding code is approximately
contained in a $t$-intersecting junta.
\item [(2)] Anticode Stability:
a stability version of the
Ahlswede-Khachatrian theorem on anticodes
determines the structure of the junta from (1)
-- it must be a certain ball $\mc{F}$.
\item [(3)] Bootstrapping:
given that the code of maximum size
is close to $\mc{F}$, it must
in fact be equal to $\mc{F}$.
\end{enumerate}

The methods required to implement these three steps
depend considerably on the size of $m$,
and we need a variety of ideas in
Combinatorics and Analysis, some of which are new.
The most significant new idea in this paper
is a random gluing operation, which may be thought of as a natural, more versatile,
analog of the sharp threshold phenomenon from the biased hypercube, as we explain next.

\paragraph{Random gluings.} Often times, when working over
the $p$-biased Boolean hypercube, i.e. $\{0,1\}^n$ along with
the measure $\mu_p(x) = p^{\card{\sett{i\in[n]}{x_i = 1}}}(1-p)^{\card{\sett{i\in[n]}{x_i = 0}}}$,
one is interested in studying the structure of a monotone family $\mathcal{F} \subseteq \{0,1\}^n$
(i.e. a family such that if $x\in\mathcal{F}$ and $x_i\leq y_i$ for all $i$, then $y\in\mathcal{F}$).
One particularly useful idea is to see how much the measure of the family changes when increasing
$p$, i.e. study the behaviour of $\mu_p(\mathcal{F}) = \Prob{x\sim\mu_p}{x\in \mathcal{F}}$ as a function of $p$. It is easy to see
that this is an increasing function of $p$, and the main point of this idea
is that the rate of increase tells us a lot about the structure
$\mathcal{F}$ has. In a nutshell, unless the family $\mathcal{F}$
has some local, junta-like, structure,
\footnote{When $p$ is bounded away from $0$ and $1$ this structure is simply a junta, but when $p = o(1)$ or
$p = 1-o(1)$, this structure may be more complicated and is not fully understood.
The notion of ``local structure'' in this case considered herein,
corresponds to having restrictions of the family $\mathcal{F}$ with significant measure.}
this increase must be sharp. This idea plays significant role is various problems in analysis and extremal combinatorics, but seems
to be specific to the cube: one heavily relies on an ordering of $\{0,1\}^n$ which makes sense with respect to intersection problems,
and such orderings do not exist on many other domains, such as $[m]^n$.

Our random gluing operator may be viewed as a natural extension of the above operator to $[m]^n$, which is also potentially more versatile and may
be relevant in other domains. Given a $k<m$ and a family $\mathcal{F}\subseteq[m]^n$, we think of shrinking the alphabet (in each
coordinate independently) from $m$ to $k$, by identifying each symbol $\sigma\in[m]$ with a symbol from $[k]$. I.e.,
given such identifications $\pi_i\colon [m]\to[k]$ for each $i$, one may consider the family
$\mathcal{F}_{\pi} = \sett{(\pi_1(x_1),\ldots,\pi_n(x_n))}{x\in \mathcal{F}}$. It is clear that such operation is ``friendly''
with respect to intersection problems (e.g., if $\mathcal{F}$ is $t$-intersecting, then so is $\mathcal{F}_{\pi}$). We show
that this operation, when sampling $\pi_1,\ldots,\pi_n$ appropriately and considering an appropriate product measure
on $[k]^n$, also enjoys the second effect of the ``increasing $p$'' idea from above. Namely, we show that unless $\mathcal{F}$
has local structure (i.e, if $\mathcal{F}$ is global as per Definition~\ref{def:global}), one can find a gluing operation that increases the measure
of $\mathcal{F}$ significantly.

The analysis of this gluing operation proceeds via noise stability and a new hypercontractive
inequality in general product spaces, which further extends our recent theory
of global hypercontractivity introduced in \cite{KLLM}.
This part of the argument can also be viewed
as a development of the Junta Method (see \cite{dinur2009intersecting,KellerLifshitz,KLLM}.)

\skipi
The following is a precise statement
of our junta approximation theorem,
which is a stability theorem of independent interest,
describing the approximate structure of any
$(t-1)$-avoiding code with size that is within
a constant factor of the maximum possible.

\begin{thm}\label{thm:junta_approx}
For every $t\in\mathbb{N}$ and $\eta>0$
there are $n_0$ and $J$ in $\mathbb{N}$
such that if $\mathcal{F}\subset[m]^{n}$
is a $(t-1)$-avoiding code with $m\geq 3$ and $n\geq n_0$
then there is a $t$-intersecting $J$-junta
$\mathcal{J}\subset[m]^n$ such that
$|\mathcal{F}\sm\mathcal{J}|\leq \eta |\mathcal{J}|$.
\end{thm}

As mentioned above, Theorem \ref{thm:junta_approx}
implies the result of \cite{EKNS}, as a junta
is far from being symmetric.
The assumption $m \ge 3$ is necessary,
as when $m=2$ we have symmetric examples
as mentioned above. When $m>m_0(t)$ is large
we will in fact obtain a more precise statement:
${\cal J}$ will be a subcube of co-dimension $t$
and we will give effective estimates for
the approximation parameter $\eta$
(see Theorems \ref{thm:approx_moderate_m}
and \ref{thm:approx_huge_m}).

Our first ingredient in the proof of
Theorem \ref{thm:junta_approx}
is a regularity lemma,
showing that any code can be approximately
decomposed into a constant number of pieces,
each of which is pseudorandom,
in a certain sense that depends on the size of $m$.
When $m<m_0(t)$ is fixed and $n>n_0(t,m)$ is large,
each piece is such that constant size restrictions
cannot significantly affect the measure.
This is a strong pseudorandomness condition,
from which the proof can be completed
fairly easily using a result of Mossel
on Markov chains hitting pseudorandom sets~\cite{Mossel}.
The idea is that, if two restrictions defining
the regularity decomposition agree in fewer than $t$
coordinates, then we can impose a further restriction
to make them agree in exactly $t-1$ coordinates,
with no significant loss in measure by pseudorandomness.
If our code is $(t-1)$-avoiding these restrictions
must be cross intersecting, but Mossel's result
implies that this is impossible for pseudorandom
codes of non-negligible measure.

When $m$ is large, one cannot obtain
such a strong pseudorandomness condition
in a regularity lemma, so we settle for
the weaker property of uncapturability.
A family $\mathcal{F}\subseteq[m]^n$ is said to be uncapturable
if it is not approximately contained in a union of constantly many
``dictatorships'', i.e. families of the form $D_{i\rightarrow j} = \sett{x\in [m]^n}{x_i = j}$ for $i\in[n]$ and $j\in [m]$.
We stress here that $m$ is not thought of as constant, so one cannot fix $i$ and take $D_{i\rightarrow j}$ for all $j\in[m]$.
Our regularity lemma in this case shows that any given family $\mathcal{F}$ may be decomposed into
pieces, such that each piece is uncapturable.
This weaker regularity lemma makes it significantly harder
to establish the $t$-intersection property as outlined above
in the case that $m$ is fixed; the main issue is that uncapturability
may not be preserved by further restrictions.

Furthermore, if $m$ is `huge' (by which we will
mean exponential in $n$) then the cross-agreement
statement used for fixed $m$ is false.
To see this, consider the codes ${\cal E}$
having all vectors with all coordinates even, and
${\cal O}$ having all vectors with all coordinates odd.
There is no non-zero agreement
between ${\cal E}$ and ${\cal O}$, yet they are
both highly uncapturable, and have measure $2^{-n}$
(which is non-negligible when $m$ is huge).

The above example naturally suggests a further case:
we say $m$ is `moderate' if it is large but not huge.
In this case, the high-level proof strategy is the same
as for fixed $m$, although the required cross-agreement
statement for uncapturable codes is difficult to prove,
and this is where we need the most
significant new ideas of the paper
(gluing and global hypercontractivity).
On the other hand, when $m$ is huge, the above example
shows that we need a different proof strategy.
Here we draw inspiration from more combinatorial arguments
of Keller and Lifshitz \cite{KellerLifshitz}
which we adapt to the setting of codes
by thinking of ${\cal F} \sub [m]^n$
as an $n$-partite $n$-graph ($n$-uniform hypergraph)
with parts of size $m$. While the high-level strategy
is similar to that in \cite{KellerLifshitz},
the implementation is quite different;
for example, the key to bootstrapping in this case
turns out to be a subtle application of Shearer's
entropy inequality.

We write $\mathcal{S}_{n,m,t}$ for a largest family
among $\{ {\cal S}_{t,r}[m]^n: r \ge 0\}$.
From Theorem~\ref{thm:junta_approx}, we see that
if a $(t-1)$-avoiding code $\mathcal{F} \sub [m]^n$
is at least as large as $\mathcal{S}_{n,m,t}$
then it is close to a $t$-intersecting junta.
This raises the stability question
for $t$-intersecting codes,
which is the second ingredient
in our proof of Theorem~\ref{thm:main}:\
must this junta be close to an extremal result?
When $m$ is large compared with $t$,
it is not hard to show that such a junta
must be close to a subcube of co-dimension $t$,
i.e.\ the ball $\mathcal{S}_{t,0}[m]^n$.
For fixed $m$, the picture is more complex, and the full range
of balls can occur; nevertheless, we are able
to establish the required stability version of
the Ahlswede-Khachatrian anticode theorem.

\begin{thm}\label{thm:stability}
For every $t\in\mathbb{N}$ and $\eps>0$
there is $\delta>0$ such that
if $\mathcal{F} \sub [m]^n$ is $t$-intersecting
with $m \ge 3$ and
$|\mathcal{F}| \ge (1-\delta)|{\cal S}_{n,m,t}|$
then $|\mathcal{F}\sm\mathcal{S}|\leq \eps |\mathcal{S}|$
for some family $\mathcal{S}$ which is isomorphic to ${\cal S}_{n,m,t} = {\cal S}_{t,r}[m]^n$,
where $0 \le r \le t$, and $r=0$ if $m>t+1$.
\end{thm}

The proof of Theorem~\ref{thm:stability}
uses a local stability analysis of the compression operator
of Ahlswede and Khachatrian \cite{AKcodes},
and also the corresponding stability result
for $t$-intersecting families in
the $p$-biased hypercube obtained by
Ellis, Keller and Lifshitz \cite{EKLstability}.

\paragraph{Notation.}
Throughout the paper, we write $[m] = \set{1,\ldots,m}$.
For any $x,y\in[m]^n$ we write
${\sf agr}(x,y) = |\{i \in [n]: x_i=y_i\}|$.
We often identify a code $\mathcal{F}\subset[m]^n$
with its characteristic function $[m]^n \mapsto \{0,1\}$.

Given $x \in [m]^n$ and $R \sub [n]$ we define
$x_R \in [m]^R$ by $(x_R)_i=x_i$.
Given disjoint $R,R' \sub [n]$
and $a \in [m]^R$, $a' \in [m]^{R'}$,
we sometimes denote their concatenation in $[m]^{R \cup R'}$
by $(x_R=a,x_{R'}=a')$.

Given $\alpha\in[m]^R$ for some $R\sub [n]$ we write
${\cal F}[\aA] = \{ x \in {\cal F}: x_R = \aA \}$ and
${\cal F}(\aA) = \{ x \in [m]^{[n]\sm R}: (x,\aA) \in {\cal F} \}$.
We also often denote ${\cal F}(\aA)$ by ${\cal F}_{R \to \aA}$.

For a coordinate $i\in [n]$ and symbol $a\in[m]$,
we write $D_{i\ra a}$ for
the subcube having all $x\in [m]^n$ for which $x_i = a$;
we will also refer to this as a `dictator'.
More generally, for $R \sub [n]$ and $a \in [m]^R$
we write $D_{R \to a} = \cap_{i \in R} D_{i \to a_i}
= \{x \in [m]^n: x_R=a\}$.

Given ${\cal F} \subset [m]^n$ and $J \subset [n]$ we say that ${\cal F}$ is a $J$-junta if there is ${\cal A} \subset [m]^J$ such that ${\cal F} = \{x \in [m]^n: x_J \in {\cal A}\}$. When we do not wish to emphasize the set $J$ itself, we instead refer to such families as $|J|$-juntas.

We will deal with various product domains
$\Omega = \Omega_1\times\ldots\times\Omega_n$,
mostly (but not only) with $\Omega = [m]^n$;
we reserve $\mu$ to denote the uniform distribution
over the domain under discussion
(which will be clear from context).
For any probability measure $\nu$ on $\Omega$
and ${\cal F} \sub \Omega$ we write
$\nu(\mathcal{F}) = \sum_{x \in {\cal F}} \nu(x)$;
similarly for $f:\Omega \to \mb{R}$ we write
$\nu(f) = \mb{E}_{x \sim \nu} f(x)
= \sum_{x \in {\cal F}} \nu(x) f(x)$.

We write $a \ll b$ to mean that there is some $a_0(b)>0$
such that the following statement holds for $0<a<a_0(b)$.

\newpage

\part{Small alphabets}

This paper has two parts. We will consider small alphabets
in this part and large alphabets in the second part.
Here we will prove our main result Theorem \ref{thm:main}
when the alphabet size $m$ is small, i.e.\
$t$ and $m$ are fixed and $n>n_0(t,m)$ is large.
This part of the paper will consist of three sections.
In the next section we prove Theorem \ref{thm:main}
for fixed $m$, assuming three key steps of the proof
(those described in the introduction). These steps
are then proven as separate theorems
in sections 3 and 4.

We start with the junta approximation,
for which the two key ingredients are
(i) a regularity lemma,
which approximately decomposes any code
into pieces which are pseudorandom
(in a sense to be made precise below),
and (ii) a theorem of Mossel \cite{Mossel}
on Markov chains hitting pseudorandom sets
which implies that we can find a pair of vectors
with any fixed agreement between any two pseudorandom
families (of non-negligible measure).

In proving the stability version
of the Ahlswede-Khachatrian anticode theorem,
the first key observation is that for codes
that are compressed (in a sense to be defined below),
there is a natural transformation of the problem
to the $p$-biased hypercube, where the stability
theorem has already been proved by
Ellis, Keller and Lifshitz \cite{EKLstability}.
This may at first not seem helpful
for a general stability result,
as compresssion destroys structure,
but in fact we can make a local stability argument,
that keeps control of the structure under gradual
decompression, and thus deduce the
general stability result.

For the bootstrapping step, the main ingredient is a
`cross disagreement' theorem, where given
two families $\mc{F}$ and $\mc{G}$ we need to find
$x \in  \mc{F}$ and $y \in \mc{G}$
with ${\sf agr}(x,y)=0$.
We need this result in the unbalanced setting
with $\mu(\mc{F}) = 1-\aA$ and $\mu(\mc{G})=\bB$,
where $\aA$ and $\bB$ are small,
but $\aA$ is large compared with $\bB$.
The idea for overcoming this obstacle
is to transform the problem via compressions
to the setting of cross-intersecting families
$\mc{F}'$ and $\mc{G}'$ in the
$p$-biased hypercube, where $p=1/m \le 1/3$.
We then move to the uniform ($1/2$-biased) measure,
where by an isoperimetric lemma of
Ellis, Keller and Lifshitz \cite{EKL}
the measure of the family corresponding to $\mc{G}$
becomes much larger, so that a trivial bound
implies that $\mc{F}'$ and $\mc{G}'$
cannot be cross-intersecting.

\section{Proof summary}

In this section we prove Theorem \ref{thm:main}
for fixed $m$ assuming the three theorems
(junta approximation, anticode stability,
bootstrapping) mentioned in the overview above,
which we now state formally.
The first theorem (junta approximation)
proves Theorem \ref{thm:junta_approx}
when $J$ and $n_0$ can depend on $m$
and replaces the conclusion
$|\mathcal{F}\sm\mathcal{J}|\leq \eta |\mathcal{J}|$
by $\mu({\cal F} \sm {\cal J}) \le \eta$,
which is an equivalent form when $m$ is fixed;
it will then remain to prove Theorem \ref{thm:junta_approx}
for $m>m_0(t,\eta)$ sufficiently large
(which we will do in Part II).

\begin{thm}\label{thm:junta_approx_small_m}
For every $\eta>0$ and $t,m\in\mathbb{N}$ with $m \ge 3$
there are $J$ and $n_0$ in $\mathbb{N}$ such that
if $\mathcal{F} \sub [m]^n$ is a $(t-1)$-avoiding code
with $n \ge n_0$ then there is a $t$-intersecting $J$-junta
$\mathcal{J}\subset[m]^n$ such that
$\mu({\cal F} \sm {\cal J}) \le \eta$.
\end{thm}

The second theorem (anticode stability)
is equivalent to Theorem \ref{thm:stability}
for fixed $m$ and $t$, as we can bound
$\mu(\mathcal{S}_{n,m,t})$ below by a constant.

\begin{thm}\label{thm:stability'}
For every $t\in\mathbb{N}$, $m \ge 3$ and $\eps>0$
there is $\delta>0$ such that
if $\mathcal{F} \sub [m]^n$ is $t$-intersecting
with $\mu(\mathcal{F})\geq \mu(\mathcal{S}_{n,m,t})-\delta$
then $\mu(\mathcal{F}\sm\mathcal{S}) \leq \eps $
for some $\mathcal{S}$ which is isomorphic to
some ${\cal S}_{n,m,t} = {\cal S}_{t,r}[m]^n$,
where $0 \le r \le t$, and $r=0$ if $m>t+1$.
\end{thm}

The third theorem (bootstrapping) is an unbalanced
cross disagreement theorem: it considers codes
${\cal G}, {\cal H} \sub [m]^n$ where ${\cal H}$
is small and ${\cal G}$ is almost complete,
and finds $x \in {\cal F}$ and $y \in {\cal G}$
with ${\sf agr}(x,y)=0$.
We state it in a form that will also be useful
later in the case that $m$ is moderately large.

\begin{thm}\label{thm:boot}
For every $t \in\mathbb{N}$ and $C>0$
there is $\eps_0>0$ such that if $0<\eps<\eps_0$
and ${\cal G}, {\cal H} \sub [m]^n$
with $\mu({\cal H}) = m^{-t} \eps$
and $\mu({\cal G}) > 1- C\eps$
then ${\sf agr}(x,y)=0$ for some
$x \in {\cal G}$ and $y \in {\cal H}$.
\end{thm}

Assuming these theorems, we now prove our main theorem
for fixed $m$:\ the following is obtained from Theorem
\ref{thm:main} by allowing $n_0$ to depend on $m$.

\begin{thm}\label{thm:exact_small_m}
For all $t\in\mathbb{N}$ and $m \ge 3$
there is $n_0 \in\mathbb{N}$ such that
if $\mathcal{F} \sub [m]^n$ is a $(t-1)$-avoiding code
with $n\geq n_0$ then
$|\mathcal{F}| \leq |\mathcal{S}_{n,m,t}|$,
with equality only when ${\cal F}$ is isomorphic to a ball.
\end{thm}

\begin{proof}
Let $0 < n_0^{-1} \ll
J^{-1}\ll \delta
 \ll \eps \ll t^{-1}, m^{-1}$.
Suppose $\mathcal{F} \sub [m]^n$ is $(t-1)$-avoiding
with $|\mathcal{F}| \ge |\mathcal{S}_{n,m,t}|$.
By Theorem~\ref{thm:junta_approx_small_m}
there is a $t$-intersecting $J$-junta
$\mathcal{J}\subset[m]^n$ such that
$|\mathcal{F}\sm\mathcal{J}|\leq \delta |\mathcal{J}|$.
We have $\mu(\mathcal{J})\geq
\mu(\mathcal{F}) - \mu(\mathcal{F}\sm \mathcal{J})
\ge \mu(\mathcal{S}_{n,m,t}) - \delta$.
By Theorem \ref{thm:stability'}
applied to ${\cal J}$, there is a copy
${\cal S}$ of $\mu(\mathcal{S}_{n,m,t})$
with $\mu(\mathcal{J}\sm\mathcal{S}) \leq \eps$.
Note that
\[ 0 \le \xi := \mu(\mathcal{F}\sm \mathcal{S})
\le \mu(\mathcal{F}\sm \mathcal{J})
+  \mu(\mathcal{J}\sm \mathcal{S})
\le \dD + \eps < 2\eps. \]
Suppose for contradiction that $\xi>0$.
Without loss of generality,
for some $r \le t$ we can write
 \[  \mathcal{S}=
 \{ x\in[m]^n : |\{ i \in [t+2r]: x_i=1 \}| \geq t+r \}. \]
By averaging, there is $\aA \in[m]^{[t+2r]}$
with $|\{i: \aA_i = 1\}| < t+r$ such that
${\cal H}:=\mathcal{F}_{[t+2r]\ra \alpha}$ has
$\mu({\cal H}) \geq \mu(\mathcal{F}\sm \mathcal{S}) = \xi$.
We can fix $\bB\in[m]^{t+2r}$
with $|\{i: \bB_i = 1\}| \ge t+r$ such that
${\sf agr}(\alpha,\beta) = t-1$.
We have
$\mu((\mathcal{F}\sm\mathcal{S})_{[t+2r]\ra \beta})
\leq m^{t+2r} \mu(\mathcal{F}\sm \mathcal{S}) \leq m^{3t}\xi$,
so $\mathcal{G} := \mathcal{F}_{[t+2r]\ra \beta}$
has $\mu(\mathcal{G})\geq 1 - m^{3t}\xi$.

By Theorem \ref{thm:boot}, with $C=m^{2t}$
and $m^t \xi$ in place of $\eps$,
we find $x \in {\cal G}$ and $y \in {\cal H}$
with ${\sf agr}(x,y)=0$.
However, this gives $(\aA,y)$ and $(\bB,x)$ in $\mc{F}$
with ${\sf agr}((\aA,y),(\bB,x))=t-1$,
which is a contradiction.
\end{proof}

\section{Junta approximation} \label{sec:junta-small}

In this section we prove the junta approximation theorem
for fixed $m$, i.e.\ Theorem \ref{thm:junta_approx_small_m}.
Our first ingredient is a regularity lemma,
showing that any code can be approximately
decomposed into a constant number of pieces,
each of which is pseudorandom,
in the sense that restrictions of constant size
do not significantly affect the measure.
This regularity lemma is similar in spirit
to that in \cite[Theorem 1.7]{EKLstability};
we refer the reader to section 1.2
of their paper for discussion how such results
are related to the large literature
on regularity lemmas in Combinatorics.

The second ingredient is a result of Mossel \cite{Mossel}
on Markov chains hitting pseudorandom sets,
which implies that any two pseudorandom codes
${\cal F},{\cal G} \sub [m]^n$ of non-negligible measure
cannot be cross intersecting,
i.e.\ we can find a `disagreement'
$(x,y) \in {\cal F} \times {\cal G}$
with ${\sf agr}(x,y) = 0$.
If ${\cal F}$ is $(t-1)$-avoiding this will imply
 ${\sf agr}(\aA,\bB) \ge t$ for any pieces
${\cal F}_{T \to \aA}$, ${\cal F}_{T \to \bB}$
of the regularity decomposition of ${\cal F}$
that are pseudorandom and of non-negligible measure.
Indeed, if we had ${\sf agr}(\aA,\bB)=t-1-s$ with $s \ge 0$
then we could arbitrarily fix a further restriction
$S \to \gG$ with $|S|=s$ to obtain pseudorandom families
${\cal F}_{(T,S) \to (\aA,\gG)}$,
${\cal F}_{(T,S) \to (\bB,\gG)}$
that are cross intersecting, which is impossible.
Here we are implicitly using the (important) fact that
pseudorandomness is preserved by constant size restrictions.

\subsection{The pseudorandom code regularity lemma}

In this subsection we prove a regularity lemma
which approximately decomposes any code
into pieces that are pseudorandom,
in the sense of the following definition.

\begin{definition}
We say $\mathcal{F}\subset[m]^n$
is $(r,\eps)$-pseudorandom
if for any $R \sub [n]$ with $\card{R}\leq r$
and $a\in[m]^R$ we have
$\card{\mu(\mathcal{F}_{R\ra a})
- \mu(\mathcal{F})}\leq \eps$.
\end{definition}

\begin{lemma}\label{lem:regularity_small_m}
For any $r,m\in\mathbb{N}$ and $\eps,\delta>0$
there is $D\in\mathbb{N}$ such that
for any $\mathcal{F}\subset [m]^n$ with $n \ge D$
there is $T\subset[n]$ with $|T| \le D$ such that
$\Prob{{\bf a}\in[m]^T}{\mathcal{F}_{T\ra {\bf a}}
  \text{ is not $(r,\eps)$-pseudorandom}}  \leq \delta$.
\end{lemma}
\begin{proof}
We construct $T$ iteratively. Starting with $T = \emptyset$,
we consider at each step the set $A$ of $a\in[m]^T$
for which $\mathcal{F}_{T\ra a}$
is not $(r,\eps)$-pseudorandom. For any $a \in A$
we fix $b(a)\in [m]^{R_{a}}$ for some $R_{a}\subset[n]$
with $|R_a| \le r$ such that
$\card{\mu(\mathcal{F}_{(T, R_{a})\ra (a, b(a))})
- \mu(\mathcal{F}_{T\ra \alpha})} > \eps$.
If $\mu(A)\leq \delta$ we are done;
otherwise, we replace $T$ by $T_{\sf new} = T \cup R$
where $R = \bigcup_{a\in A}{R_{a}}$ and iterate.

We will argue that this process stops with
$|T|$ bounded by some function depending
on $m$, $r$, $\dD$ and $\eps$, but not on $n$.
To do so, we apply a standard `energy increment' argument
to the mean-square density \[  E(T) =
\Expect{{\bf a}\in [m]^T}{\mu(\mathcal{F}_{T\ra {\bf a}})^2}. \]
Clearly, $E(T)\leq 1$ for any $T\subset[n]$,
and  $E(T_1)\leq E(T_2)$ whenever $T_1\subset T_2$
by Cauchy-Schwarz.

We will show that $E(T)$ increases
significantly at each step of the process.
Indeed, comparing $E(T_{\sf new})$ and $E(T)$
term by term, we have \[ E(T_{\sf new}) - E(T)
  =\Expect{{\bf a} \in [m]^T}{
  \Expect{{\bf b}\in [m]^R}
  {\mu(\mathcal{F}_{(T,R)\ra ({\bf a},{\bf b})})}^2
   -\mu(\mathcal{F}_{T\ra {\bf a}})^2}
  = \Expect{{\bf a} \in [m]^T}
  {\Var Z_{{\bf a}}},  \]
where we consider $Z_a({\bf b})
= \mu({\cal F}_{(T,R)\ra (a,{\bf b})})$
as a random variable determined by
the random choice of ${\bf b}\in [m]^R$.
We have $\Var Z_a \ge 0$ for all $a$,
and for any $a\in A$ we have
$\Var Z_{a} \ge m^{-\card{R_{a}}}\eps^2\geq m^{-r} \eps^2$
in light of the restriction $R_{a}\ra b(a)$.
Therefore, $E(T_{\sf new})\geq E(T)
+ \mu(A) m^{-r}\eps^2\geq E(T) + \delta m^{-r}\eps^2$.

In other words, as long as the process does not terminate,
the energy function increases by at least $\dD m^{-r}\eps^2$.
As the energy is always at most $1$,
the process terminates after at most $m^r/\dD\eps^2$ steps.
Each restriction adds at most $r$ new variables to $T$,
so in each step $\card{T_{{\sf new}}}\leq 2^{\card{T}}\cdot r$,
and so the final size of $T$ is bounded by
some function of $m$, $r$, $\dD$ and $\eps$.
\end{proof}

\subsection{Markov chains hitting pseudorandom sets}

In this subsection we discuss a special case of a result
of Mossel \cite{Mossel}
needed for the proof of our junta approximation
theorem for small alphabets, which can be formulated
in terms of Markov chains hitting pseudorandom sets.
We start by summarising some properties of Markov chains
(see \cite{Levin} for an introduction).
We will consider finite Markov chains,
i.e.\ a sequence of random variables $(X_i)_{i \ge 0}$
taking values in a state space $S$
(some finite set)
described by a transition matrix $T$
with rows and columns indexed by $S$,
where for any event $E$ determined by
$(X_0,\dots,X_i)$ with $X_i=x$ we have
$\mb{P}(X_{i+1}=y \mid X_i=x) = T_{xy}$.
We also view $T$ as an averaging operator
on functions $f: S \to \mb{R}$,
corresponding to matrix multiplication
when we view $f$ as a vector in $\mb{R}^S$:
we have $(Tf)(x) = \mb{E}[f(X_1) \mid X_0=x]
= \sum_y T_{xy} f(y) = (Tf)_x$.

We will suppose $T$ is irreducible
(for any $x,y \in S$ there is
some $k \in \mb{N}$ with $T^k_{xy}>0$)
so there is a unique stationary distribution
(a probability distribution $\nu$ on $S$
such that $\nu T = \nu$).
The stationary chain is obtained
by letting $X_0$ have distibution $\nu$,
and then each $X_i$ has distribution $\nu$.
In the stationary chain we have
$\mb{P}(X_0 = a, X_1=b)
= P_{ab} := \nu_a T_{ab}$.
We say $T$ is reversible if $P$ is symmetric,
i.e.\ $P_{ab}=P_{ba}$ for all $a,b \in S$
(the name corresponds to the observation
that the distribution of the stationary chain
is invariant under time reversal).

When $T$ is reversible, it defines
a self-adjoint operator on $L^2(S,\nu)$,
i.e.\ functions $f: S \to \mb{R}$
with the inner product
$\langle f,g \rangle = \sum_x \nu_x f(x) g(x)$,
so $L^2(S,\nu)$ has an orthonormal basis $B$
of eigenfunctions of $T$.
We can write any $f \in L^2(S,\nu)$
in the form $f = \sum_{b \in B} c_b b$,
and then $\mb{E}f^2 = \langle f,f \rangle
= \sum_{b \in B} c_b^2$.
The largest eigenvalue is $1$,
and the corresponding eigenspace consists
of constant functions on $S$.
If $Tf=\lL f$ with $\lL \ne 1$
then $\mb{E}f := \mb{E}_{x \sim \nu} f(x)
= \sum_x \nu_x f(x) = \langle f,1 \rangle = 0$.
The absolute spectral gap $\lL_*$
is the minimum value of $1-|\lL|$
over all eigenvalues $\lL \ne 1$; equivalently,
\[ (1-\lL_*)^2 = \sup \{ \mb{E}(Tf)^2:
 \mb{E}f=0, \mb{E}f^2=1 \}.\]

Now we describe a special case of \cite[Theorem 4.4]{Mossel}, and for
that we require a basic set-up.
Let $T$ be a reversible, irreducible Markov chain acting on $S=[m]$,
and consider its tensor power $T^{\otimes n}$
acting on $\OO=[m]^n$ independently
in each coordinate, i.e.\ with transition matrix
$T^{\otimes n}_{xy} = \prod_{i=1}^n T_{x_iy_i}$.
In essence, \cite[Theorem 4.4]{Mossel} asserts that
if $T$ has a constant spectral gap, and we have
pseudorandom codes $\mc{F},\mc{G} \sub [m]^n$
of noticeable measure, then sampling
consecutive random states $x,y$ of the stationary chain for
$T^{\otimes n}$, we have that $x\in\mc{F}$, $y\in\mc{G}$ with
significant probability.

\begin{thm}\label{thm:MOSMarkov}
Let $T$ be a reversible Markov chain on $[m]$
with absolute spectral gap $\lL_*>0$.
Let $\nu$ denote the stationary measure
of $T^{\otimes n}$ and $x$ and $y$ be consecutive
random states of the stationary chain.
Then for any $\mu>0$ there are
$\eps,c>0$ and $r \in \mb{N}$ such that
if $\mathcal{F},\mathcal{G}\subset [m]^n$
are $(r,\eps)$-pseudorandom
with $\nu({\cal F}), \nu({\cal G}) > \mu$
then $\mb{P}(x \in \mc{F}, y \in \mc{G})>c$.
\end{thm}

For convenience of the reader, we outline below the (standard) derivation of Theorem~\ref{thm:MOSMarkov} from
existing results in the literature.
\paragraph{Deriving Theorem~\ref{thm:MOSMarkov} from~\cite[Theorem 4.4]{Mossel}.}
Let $B = \{b_1,\dots,b_m\}$ be an orthonormal basis
for $L^2(S,\nu)$ consisting of eigenvectors of $T$.
We take $b_1$ to be the trivial eigenvector, i.e. $b_1(s)=1$ for all $s \in S$, which
has eigenvalue $1$. We remark that by the spectral gap of $T$, it follows that the eigenvalue of
each $b_j$ for $j\neq 1$ is at most $1-\lambda_{*}$.
We will view each $b_i$ as a random variable on $(S,\nu)$,
and in this language we have that $\mb{E} b_i b_j = 1_{i = j}$.
Using the basis $B$, we may find a basis for $L^2(S^n,\nu^{\otimes n})$ by
tensorizing. Namely, for each $i\in [n]$ we take an independent copy of $B$,
say $b^i=(b^i_j: j \in [m])$, and then our basis is $\bm{b} = (\bm{b}_{j_1,\ldots,j_n})_{j_1,\ldots,j_n\in[m]}$,
where $\bm{b}_{j_1,\ldots,j_n} = \prod\limits_{i=1}^{n} b^i_{j_i}$, denoted by.
We can thus represent any function on $\OO$
as a multilinear polynomial
$P(\bm{b}) = \sum_\aA c_\aA b^\aA$
where $\aA$ ranges over $[m]^n$ and $b^\aA := \prod_i b^i_{\aA_i}$.

This above view allows us to extend the definition of $P$ to $\mathbb{R}^{mn}$. A technical point to note, however, is that even if
our original function $P$ was bounded on $[m]^n$ (in our case, it is even be Boolean valued), the extension to
$\mathbb{R}^{mn}$ may not be bounded. For this reason, one first applies a small noise on the function $P$, i.e.
considers $Q(x) = T_{1-\eta} P(x) = \Expect{x'\sim_{1-\eta} x}{P(x')}$ where for each $i\in[n]$ independently, $x'_i = x_i$ with probability
$1-\eta$ and otherwise $x'_i$ is resampled according to $\nu$ ($\eta > 0$ is to be thought of as a small constant,
much smaller than the spectral gap $\lambda_{*}$ of $T$), and then truncates it. Namely, consider the multi-linear extension of $Q$,
$Q(b)$ as defined above, and let $\tilde{P}(b) = Q(b)$ if $0\leq Q(b)\leq 1$,
$\tilde{P}(b) = 1$ if $Q(b) > 1$, and otherwise $\tilde{P}(b) = 0$.

Let $x$ and $y$ be sampled as consecutive random states of the stationary chain for $T^{\otimes n}$, and
let $f(\bm{b}(x)) = 1_{x\in\mc{F}}$, $g(\bm{b}(y)) = 1_{y\in\mc{G}}$. Our goal is thus to prove a lower bound on $\Expect{x,y}{f(\bm{b}(x))g(\bm{b}(y))}$. The invariance
principles of~\cite{MOO,MosselGeneral,Mossel} allows one to establish non-trivial lower bounds on this quantity by considering its ``analog in Gaussian space'',
provided that $f,g$ are sufficiently random-like.

To be more precise, let us first consider $\bm{b}_{j_1,\ldots,j_n}(x)$ and $\bm{b}_{j'_1,\ldots,j'_n}(y)$
where $x$ and $y$ are sampled as consecutive random states of the stationary chain for $T^{\otimes n}$.
Thus $\mb{E} b^i_{j_i}(x) b^{i'}_{j_{i'}'}(y)$ is zero unless $i=i'$ and $j_i=j_{i'}'$, and then it is equal to
the eigenvalue $\lambda_{j_i}$ such that $T b_{j_i} = \lambda_j b_{j_i}$. We now wish to define the Gaussian analog of
$\bm{b}_{j_1,\ldots,j_n}(x)$ and $\bm{b}_{j_1,\ldots,j_n}(y)$. Let $Z = \{z_1,\dots,z_m,z_1',\ldots,z_m'\}$
be Gaussian variables
with the same covariance matrix. Namely, we take $z_1= z_1' = 1$,
and $z_2,\ldots,z_m$ and $z_2',\ldots,z_m'$ are jointedly distributed standard Gaussian
random variables such that $z_2,\ldots,z_m$ are independent, $z_2',\ldots,z_m'$ are independent,
and $\Expect{}{z_j z'_{j'}} = \Expect{}{b_j(x) b_{j'}(y)} = \lambda_j 1_{j = j'}$.
We take $n$ independent copies of $Z$, $Z^{i}  = \{{z^i}_1,\dots,{z^i}_m,{z^i}_1',\ldots,{z^i}_m'\}$,
and then define $\bm{z}_{j_1,\ldots,j_n} = \prod\limits_{i=1}^{n} {\bm{z}^i}_{j_i}$ and
$\bm{z}'_{j_1,\ldots,j_n} = \prod\limits_{i=1}^{n} {\bm{z}^i}_{j_i}'$.
The random variables $\bm{z}_{j_1,\ldots,j_n}$, $\bm{z}'_{j_1,\ldots,j_n}$ are to be thought of as
the Gaussian analogs of $\bm{b}_{j_1,\ldots,j_n}(x)$ and $\bm{b}_{j_1,\ldots,j_n}(y)$.

Building on~\cite{MOO}, Mossel~\cite{MosselGeneral} showed that
for $f,g\colon [m]^n\to [0,1]$ with ``small enough influences'',\footnote{
We omit the definition of `influences' for now,
as we do not need it here, but it will reappear later
in a more general context when we discuss our theory
of global hypercontractivity.} one has
$\Expect{}{f(\bm{b}(x))\cdot g(\bm{b}(y))}$ is very close
$\Expect{}{\tilde{f}(\bm{z})\tilde{g}(\bm{z}')}$.
The arguments in \cite{Mossel}
establish the same statement with the more relaxed condition that
$f$ and $g$ are $(r,\eps)$-pseudorandom
(the term `resilient' is used therein). More precisely, Mossel showed that
for all $\delta>0$, there are $r\in\mathbb{N}$ and $\eps>0$ (also depending on $m$ and the spectral gap $\lambda_{*}$, which are thought of as constants), such that
$\card{\Expect{}{f(\bm{b}(x))\cdot g(\bm{b}(y))}-\Expect{}{\tilde{f}(\bm{z})\tilde{g}(\bm{z}')}}\leq \delta$.

For $\tilde{f}$, the fact that $f$ has averages
at least $\mu$ implies, by the invariance principle (i.e. the above with $g = 1$), that
$\tilde{f}$ has average at least $\mu/2$; similarly the average of $\tilde{g}$ is at least $\mu/2$.
Thus, $\Expect{}{\tilde{f}(\bm{z}) \tilde{g}(\bm{z}')}>c(\lambda_*, \mu)>0$ by
reverse hypercontractivity (see \cite[Theorem A.78]{Hazla} for example), and
as this is close to $\Expect{}{f(\bm{b}(x))\cdot g(\bm{b}(y))}$, we get that
$\Expect{}{f(\bm{b}(x))\cdot g(\bm{b}(y))}\geq c/2$, establishing Theorem~\ref{thm:MOSMarkov}.

\skipi

The following result is an immediate consequence
of Theorem \ref{thm:MOSMarkov}, applied
with the Markov chain $T$ on $[m]$
which at each step moves to a uniformly random
state different from the current state
(note that $\lL_*>0$ when $m \ge 3$,
but this fails for $m=2$).

\begin{thm}\label{thm:MOS}
For every $m\geq 3$ and $\mu>0$ there are
$\eps,c>0$ and $r \in \mb{N}$ such that
if $\mathcal{F},\mathcal{G}\subset [m]^n$
are $(r,\eps)$-pseudorandom
with $\mu({\cal F}), \mu({\cal G}) > \mu$
and $(x,y)$ is a uniformly random pair
in $[m]^n \times [m]^n$ with ${\sf agr}(x,y) = 0$
then $\mb{P}(x \in \mc{F}, y \in \mc{G})>c$;
in particular, ${\sf agr}(x,y) = 0$
for some $(x,y) \in {\cal F} \times {\cal G}$.
\end{thm}

\subsection{Approximation by junta}
We conclude this section by proving
Theorem \ref{thm:junta_approx_small_m}.

\begin{proof}[Proof of
Theorem \ref{thm:junta_approx_small_m}]
Let $t,m\in\mathbb{N}$ with $m \ge 3$ and $\eta>0$, fix
$0 \ll n_0^{-1} \ll D^{-1} \ll r^{-1},\eps \ll \eta, t^{-1}, m^{-1}$
and suppose $\mathcal{F} \sub [m]^n$ is $(t-1)$-avoiding.
By Lemma \ref{lem:regularity_small_m} we find
$T\subset[n]$ with $|T| \le D$ such that
\[ \Prob{{\bf a}\in[m]^T}{\mathcal{F}_{T\ra {\bf a}}
  \text{ is not $(r,\eps)$-pseudorandom}} \leq \eta/2.\]
We will show that the required conclusions
of the theorem hold for the junta
$\mathcal{J} = \sett{x\in[m]^n}{x_T\in J}$, where
\[ J = \sett{\alpha\in[m]^T}{\mathcal{F}_{T\ra \alpha}
\text{ is $(r,\eps/2)$-pseudorandom and }
\mu(\mathcal{F}_{T\ra \alpha})\geq \eta/2},\]
i.e.\ that ${\cal J}$ is $t$-intersecting
(equivalently,  $J$ is $t$-intersecting) and
${\cal F}$ is approximately contained in ${\cal J}$.

To see that $J$ is $t$-intersecting,
suppose for contradiction we have $\aA_1,\aA_2\in J$
with ${\sf agr}(\aA_1,\aA_2)=t-1-s$ with $s \ge 0$.
Fix $S\subset[n]\sm T$ of size $s$
and $x\in[m]^S$ arbitrarily, and consider the families
\[ \mathcal{G}_i = \sett{w\in [m]^{[n]\sm (T\cup S)}}
{(\aA_i,x,w)\in \mathcal{F}} \] for $i=1,2$.
By definition of $J$ both $\mu(\mathcal{G}_i)\geq
\mu(\mathcal{F}_{\aA_i}) - \eps/2\geq \eta/3$
and  $\mathcal{G}_i$ is $(r-t,\eps)$-pseudorandom.
By Theorem \ref{thm:MOS} we find
$(w_1,w_2) \in {\cal G}_1 \times {\cal G}_2$
with ${\sf agr}(w_1,w_2)=0$.
However, this gives $(\aA_i,x,w_i)$ for $i=1,2$
in ${\cal F}$
with agreement $t-1$, which is a contradiction.

It remains to bound $\mu(\mathcal{F}\sm\mathcal{J})
= \sum_{\alpha\not\in J}
m^{-\card{T}} \mu(\mathcal{F}_{T\ra \alpha})$.
We partition $[m]^T \sm J$ into $(B_1,B_2)$
where $B_1$ contains those $\aA \in [m]^T \sm J$
with $\mu(\mathcal{F}_{T\ra \alpha}) < \eta/2$, and $B_2 = [m]^T\sm (B_1\cup J)$.
Clearly the contibution to the sum from $\aA \in B_1$
is at most $\eta/2$. For $\aA \in B_2$, we note that
$\mathcal{F}_{T\ra \alpha}$
is not $(r,\eps/2)$-pseudorandom
by definition of $J$, so $\sum_{\aA \in B_2}
m^{-\card{T}} \mu(\mathcal{F}_{T\ra \alpha})
\le \mu(B_2) < \eta/2$ by choice of $T$.
Thus $\mu(\mathcal{F}\sm\mathcal{J}) < \eta$.
\end{proof}

\section{Compression, stability and bootstrapping}

In this section we prove the anticode stability theorem
for fixed $m$, i.e.\ Theorem \ref{thm:stability'},
and the bootstrapping result Theorem \ref{thm:boot}.
Both rely on a compression procedure,
introduced by Ahlswede and Khachatrian \cite{AKcodes},
which modifies any code in such a way to use some
symbol (say 1) `as much as possible',
while maintaining its size and not reducing
its minimum intersection size.

In the first subsection we will formally define
compression and prove some of its well-known properties.
In the second subsection we prove the stability result
for compressed codes, by reducing it to
the corresponding stability result
for $t$-intersecting families in
the biased hypercube obtained by
Ellis, Keller and Lifshitz \cite{EKLstability}.
We deduce Theorem \ref{thm:stability'} in the third subsection,
via a decompression argument, in which we reverse
the compressions while keeping control of
structure via a local stability argument.
In the final subsection we prove Theorem \ref{thm:boot},
by using compressions to reformulate the problem in terms
of cross-intersecting families in the biased hypercube.

\subsection{Compression}\label{sec:compression}

For any $i\in[n]$ and $j\in[m]$
we define the compression operator
$T_{i,j}\colon [m]^n \to [m]^n$
that replaces $j$ by $1$
in coordinate $i$ if possible,
i.e.\ for $x \in [m]^n$
we let $T_{i,j}(x) = y \in [m]^n$
where $y_r = x_r$ for all $r\neq i$,
and $y_i=x_i$ if $x_i\neq j$
or $y_i=1$ if $x_i = j$.
We also define a compression operator,
also denoted $T_{i,j}$, on codes,
that replaces any vector $x$ by $T_{i,j}(x)$
unless the latter is already present, i.e.\
\[ T_{i,j}(\mathcal{F})
= \sett{x}{ T_{i,j}(x) \in \mathcal{F}}
\cup \sett{T_{i,j}(x)}{ x\in\mathcal{F}}. \]
We also define
$T_{i} = T_{i,2}\circ T_{i,3}\circ\ldots\circ T_{i,m}$
for any $i\in n$,
and $T = T_1\circ T_2\circ\ldots\circ T_n$.
One can think of $T$ as trying to set as many
coordinates as possible equal to $1$.
We call $\mathcal{F}\subset[m]^n$ compressed
if $T(\mathcal{F}) = \mathcal{F}$.

We need the following well-known facts
about these compression operators.

\begin{fact}\label{fact:compression} $ $

Let $\mathcal{F},\mathcal{G}\subset[m]^n$,
$i\in[n]$ and $j\in[m]$.
\begin{enumerate}
  \item We have $\mu(T(\mathcal{F})) = \mu(T_{i,j}(\mathcal{F})) = \mu(\mathcal{F})$.

  \item If $\mathcal{F},\mathcal{G}$ are cross $t$-intersecting
then so are $T_{i,j}(\mathcal{F})$ and $T_{i,j}(\mathcal{G})$,
and so are $T(\mathcal{F})$ and $T(\mathcal{G})$.
Furthermore, any $x\in T(\mathcal{F}),y\in T(\mathcal{G})$
have at least $t$ common coordinates equal to $1$.
\end{enumerate}

\end{fact}

\begin{proof}
Assume without loss of generality that $i=1$.
To see that $\mu(T_{i,j}(\mathcal{F})) = \mu(\mathcal{F})$,
we consider any $x\in[m]^{n-1}$,
note that vectors $(a,x)$ with $a\in [m]\sm\set{1,j}$
are unaffected by $T_{i,j}$, and that
$T_{i,j}({\cal F})$ and ${\cal F}$ contain
the same number of elements of $\{(1,x),(j,x)\}$.
By iterating we deduce
$\mu(T(\mathcal{F})) = \mu(\mathcal{F})$.

Next, suppose for contradiction that
$\mathcal{F},\mathcal{G}$ are cross $t$-intersecting
but $T_{1,j}(\mathcal{F})$ and $T_{1,j}(\mathcal{G})$ are not.
Then there are $(a,x)\in T_{i,j}(\mathcal{F})$
and $(b,y) \in T_{i,j}(\mathcal{G})$
with ${\sf agr}((a,x),(b,y))<t$.
As $\mathcal{F},\mathcal{G}$ are cross $t$-intersecting
we cannot have both $(a,x) \in {\cal F}$
and $(b,y) \in {\cal G}$, so without loss of generality
$(a,x)=(1,x)$ was obtained from $(j,x) \in {\cal F}$.
We must have $(b,y) \in {\cal G}$, as otherwise
$(b,y)=(1,y)$ was obtained from $(j,y) \in {\cal G}$,
but then $(j,x) \in {\cal F}$ and $(j,y) \in {\cal G}$
with ${\sf agr}((j,x),(j,y)) ={\sf agr}((1,x),(1,y))<t$,
contradiction. As $(j,x) \in {\cal F}$
and $(b,y) \in {\cal G}$ we have
${\sf agr}((j,x),(b,y)) \ge t$, so $b=j$.
As $(b,y) \in T_{i,j}(\mathcal{G})$
we must have $(1,y) \in \mathcal{G}$.
But now ${\sf agr}((j,x),(1,y))<t$
gives a contradiction.
Thus $T_{1,j}(\mathcal{F})$ and $T_{1,j}(\mathcal{G})$
are cross $t$-intersecting.
By iterating, so are $T(\mathcal{F})$ and $T(\mathcal{G})$.

Finally, suppose for contradiction that
$x\in T(\mathcal{F}),y\in T(\mathcal{G})$
have fewer than $t$ common coordinates equal to $1$.
Let $x'$ be obtained from $x$ by setting
$x'_i=1$ if $x_i=y_i \ne 1$ or $x'_i=x_i$ otherwise.
Then $x' \in T(\mathcal{F})$ but $x'$ and $y$
only agree on coordinates $i$ with $x_i=y_i=1$,
which contradicts
$T(\mathcal{F})$ and $T(\mathcal{G})$
being  cross $t$-intersecting.
\end{proof}

Next we will define a transformation
from compressed codes in $[m]^n$ to
monotone\footnote{We call ${\cal A} \sub \{0,1\}^n$
monotone if $y\in \mathcal{G}$ whenever
$x\in\mathcal{G}$ and $x\leq y$ coordinatewise.} 
families in the cube $\{0,1\}^n$ that
preserves minimum (cross) intersection size,
and does not decrease the measure when we
adopt the $p$-biased measure on the cube
with $p=1/m$ (as we will do throughout this section).

\begin{definition}\label{def:compress_reduce}
We define $h\colon [m]\to\power{}$ by
$h(1) = 1$ and $h(a) = 0$ for all $a\neq 1$,
and $h^{\otimes n}\colon [m]^n\to\power{n}$ by
$h^{\otimes n}(x) = (h(x_1),\ldots,h(x_n))$.
For any ${\cal F} \sub [m]^n$ we let
$\wt{\mathcal{F}} = h^{\otimes n}({\cal F}) \sub \{0,1\}^n$.
\end{definition}

\begin{fact}\label{fact:properties_compressed_reduction}
Suppose ${\cal F},{\cal G} \sub [m]^n$ are compressed.
\begin{enumerate}
  \item The family $\wt{\cal F}$ is monotone and $\mu_p(\wt{\mathcal{F}})\geq \mu(\mathcal{F})$.
  \item If ${\cal F}$ is $t$-intersecting then so is $\wt{\cal F}$.
  \item If ${\cal F}, {\cal G}$ are cross $t$-intersecting
then so are $\wt{\cal F}, \wt{\cal G}$.
\end{enumerate}
\end{fact}

\begin{proof}
The intersection statements are immediate
from the final part of Fact~\ref{fact:compression}.
For monotonicity, consider any $\wt{x}\in\wt{\cal F}$
and $\wt{y}\geq \wt{x}$. Fix  $x\in\mathcal{F}$
with $h^{\otimes n}(x)=\wt{x}$, i.e.\
$x_i = 1$ if and only if $\wt{x}_i = 1$.
Define $y \in [m]^n$ by $y_i=1$
if $\wt{y}_i = 1\neq \wt{x}_i$
or $y_i=x_i$ otherwise.
Then $y \in {\cal F}$,
as $\mathcal{F}$ is compressed,
and $h^{\otimes n}(y)=\wt{y}$,
so $\wt{y}\in\wt{\cal F}$.

To show $\mu_p(\wt{\mathcal{F}})\geq \mu(\mathcal{F})$
we consider intermediate product spaces
$\power{r}\times[m]^{n-r}$ with the measure
$\nu_r=\mu_p^{r}\times \mu$, and intermediate families
$\mathcal{F}_r = (h^{\otimes r}\otimes I^{\otimes n-r})(T(\mathcal{F}))$ for any $r\geq 0$. It suffices to show
$\nu_{r+1}(\mathcal{F}_{r+1})\geq\nu_r(\mathcal{F}_r)$
for any $r\geq 0$. We can write
\[ \nu_{r+1}(\mathcal{F}_{r+1}) - \nu_{r}(\mathcal{F}_r)
= \sum\limits_{x\in\power{r},y\in[m]^{n-r-1}}{\mu_p(x) m^{-(n-r-1)}
 ( \mu_p(B_{x,y,r}) - |A_{x,y,r}|/m )}, \]
where for $0\leq r\leq n$
and $x\in\power{r}$, $y\in[m]^{n-r-1}$ we define
\[
A_{x,y,r} = \sett{a\in [m]}{(x,a,y)\in \mathcal{F}_r},
~~~~~~~~~~~~~~~~~~~~~~~~
B_{x,y,r} = \sett{a\in\power{}}{(x,a,y)\in\mathcal{F}_{r+1}}.
\]
Thus it suffices to show
$\mu_p(B_{x,y,r})\geq \frac{\card{A_{x,y,r}}}{m}$ for all $x,y$.
To see this, suppose first that $|A_{x,y,r}|=1$.
As ${\cal F}$ is compressed we have $A_{x,y,r}=\set{1}$,
so $B_{x,y,r} = h(A_{x,y,r}) = \set{1}$
and $\mu_p(B_{x,y,r}) = p = |A_{x,y,r}|/m$.
Otherwise, if $|A_{x,y,r}| \ge 2$ we have
$B_{x,y,r} = h(A_{x,y,r}) = \set{0,1}$,
so $\mu_p(B_{x,y,r})=1 \ge |A_{x,y,r}|/m$.
\end{proof}

\subsection{Stability when compressed}

In this subsection we prove Theorem~\ref{thm:stability'}
for compressed families, using the
corresponding stability result
for $t$-intersecting families in
the biased hypercube obtained by
Ellis, Keller and Lifshitz \cite{EKLstability},
which we start by stating. Given $n,p,t$, let $\mathcal{S}_{n,p,t}$ denote
a family ${\mathcal S}_{t,r}\{0,1\}^n \subset \{0,1\}^n$
with largest $p$-biased measure, where $r = 0,1, \ldots,t-1$ and
\[
{\cal S}_{t,r}\{0,1\}^n = \{ x\in\power{n} :
| \{ i\in[t+2r] : x_i = 1 \} | \geq t+r \}
\]
The following is implied by
\cite[Theorem 1.10]{EKLstability}.%
\footnote{We state it in a weaker form where we do not specify
the exact dependency between parameters,
as we do not require this.}

\begin{thm}\label{thm:p_biased_stability}
For every $t\in\mathbb{N}$, $\zeta > 0$ and $\eps>0$
there is $\delta>0$ such that
if $\mathcal{F}\subset\power{n}$ is $t$-intersecting,
$\zeta \leq p\leq \frac{1}{2} - \zeta$ and
$\mu_p(\mathcal{F})\geq (1-\dD) \mu(\mathcal{S}_{n,p,t})$ then
$\mu_p(\mathcal{F}\sm \mathcal{S})\leq \eps\mu(\mathcal{S})$
for some copy $\mathcal{S}$
of ${\cal S}_{n,p,t} = {\cal S}_{t,r}\{0,1\}^n$,
where $0 \le r \le t$ if $p \le 1/3$,
and $r=0$ if $p \leq \frac{1}{t+1}-\zeta$.
\end{thm}

Using Theorem~\ref{thm:p_biased_stability} we can prove
a weaker version of Theorem \ref{thm:stability'},
with the additional assumption that ${\cal F}$ is compressed.
This version will be used in the next subsection to prove
Theorem \ref{thm:stability'} as stated.

\begin{claim}\label{claim:stability_compressed}
For every $t\in\mathbb{N}$, $m \ge 3$ and $\eps>0$
there is $\delta>0$ such that if $\mathcal{F} \sub [m]^n$
is compressed and $t$-intersecting with
$\mu(\mathcal{F})\geq (1-\dD)\mu(\mathcal{S}_{n,m,t})$
then $\mu(\mathcal{F}\sm\mathcal{S}) \leq \eps\mu({\cal S})$
for some copy $\mathcal{S}$
of ${\cal S}_{n,m,t} = {\cal S}_{t,r}[m]^n$,
where $0 \le r \le t$, and $r=0$ if $m>t+1$.
\end{claim}

\begin{proof}
Suppose $\mathcal{F} \sub [m]^n$
is compressed and $t$-intersecting with
$\mu(\mathcal{F})\geq (1-\dD)\mu(\mathcal{S}_{n,m,t})$,
where $\dD \ll m^{-1}, t^{-1}, \eps$.
We consider $\wt{\mathcal{F}}\subset\power{n}$
given by Definition~\ref{def:compress_reduce}.
By Fact~\ref{fact:properties_compressed_reduction},
$\wt{\mathcal{F}}$ is $t$-intersecting,
and $\mu_p(\wt{\mathcal{F}}) \geq
\mu(\mathcal{F})\geq (1-\delta)\mu(\mathcal{S}_{n,m,t})
=(1-\delta)\mu_p(\mathcal{S}_{n,p,t})$, where $p=1/m$.
By Theorem~\ref{thm:p_biased_stability} we have
$\mu_p(\mathcal{F}\sm \wt{\mathcal{S}})
\leq \eps\mu(\wt{\mathcal{S}})$
for some copy $\wt{\mathcal{S}}$
of ${\cal S}_{n,p,t} = {\cal S}_{t,r}\{0,1\}^n$,
where $0 \le r \le t$ (as $p=1/m \le 1/3$)
and $r=0$ if $m > t+1$
(taking $\zeta < \tfrac{1}{t+1} - \tfrac{1}{t+2}$).

We show that the conclusion of the claim holds for
$\mathcal{S} = \{ x: h^{\otimes n}(x)\in\wt{\mathcal{S}} \}$,
where $h\colon [m]^n\to\power{n}$
is as in Definition~\ref{def:compress_reduce}.
To see this, first note that
$\mathcal{S}$ is a copy of $\mathcal{S}_{n,m,t}$.
Furthermore, if $x\in\mathcal{F}\sm\mathcal{S}$
then $h(x)\in \wt{\mathcal{F}}\sm\wt{\mathcal{S}}$,
and if $x$ is uniformly random in $[m]^n$
then $h^{\otimes n}(x)$ is distributed as $\mu_p$, so
  \[  \Prob{x\in[m]^n}{x\in\mathcal{F}\sm\mathcal{S}}
  \leq \Prob{x\in[m]^n}{h^{\otimes n}(x)
  \in\wt{\mathcal{F}}\sm\wt{\mathcal{S}}}
  =\mu_p(\wt{\mathcal{F}}\sm\wt{\mathcal{S}})
  \leq  \eps\mu_p(\wt{\mathcal{S}})
  =\eps\mu(\mathcal{S}). \qedhere  \]
\end{proof}

\subsection{Decompression and local stability}

In this subsection we prove Theorem \ref{thm:stability'}
in general, deducing it from the compressed case proved
in the previous subsection, and for that we use decompression
and local stability arguments. We start with a proof sketch,
where for simplicity we assume that $m>t+1$,
so that the extremal examples
are cubes of co-dimension $t$.

Suppose $\mathcal{F} \sub [m]^n$ is $t$-intersecting
with size close to the maximum possible.
Let $\mathcal{G} = T(\mathcal{F})$
be the compressed form of ${\cal F}$.
By the previous subsection,
${\cal G}$ is close to a subcube,
say $\mathcal{S} = \sett{x}{x_1=\ldots=x_t = 1}$.

We now decompress: we consider how the family changes
as we undo the compression operators one by one.
First we note that undoing
$T_n,T_{n-1},\dots,T_{t+1}$
does not change the distance from ${\cal S}$,
so $\mathcal{G}_t = T_1\circ \ldots\circ T_{t}(\mathcal{F})$
has the same distance to $\mathcal{S}$ as $\mathcal{G}$.

The main point of the argument is to analyze the effect
of undoing $T_i$ for $i=1,\ldots,t$. For $j \in [m]$
we let $\aA_j$ be the fraction of ${\cal G}_{t-1}$
with prefix $(1^{t-1}, j)$. If there is some $j^{\star}$
with $\alpha_{j^{\star}}$ close to $1$ then ${\cal G}_{t-1}$
is close to a subcube, and we can continue decompressing.
Otherwise, we can partition most of ${\cal G}_{t-1}$
into two non-negligible parts such that the value of $j$
in the prefix $(1^{t-1}, j)$ always differs
between the two parts. However, as ${\cal F}$
is $t$-intersecting, this implies that the two parts
must be cross-intersecting on the coordinates $[n] \sm [t]$;
this will give a contradiction by the following form
of Hoffman's bound (which we will prove
later in a more general form, see Lemma \ref{lem:hoffman}).

\begin{lemma}\label{lem:hoffman1}
Suppose $\mathcal{G}_1,\mathcal{G}_2\subset[m]^n$ are
cross-intersecting with $\mu({\cal G}_i)=\aA_i$ for $i=1,2$.
Then $\aA_1\aA_2 \le (1-\aA_1)(1-\aA_2)/(m-1)^2$.
\end{lemma}

We start with a lemma that applies Lemma \ref{lem:hoffman1}
to implement the idea discussed in the previous paragraph.
Recall that the largest intersecting codes in $[m]^n$
are the `dictators' $D_{i \to j} = \{x: x_i=j\}$.
We show that if $\mathcal{A}, \mathcal{B} \subset [m]^n$
have nearly maximum size, are cross-intersecting and
$T_i(\mathcal{A}), T_i (\mathcal{B}) \subset D_{i\ra 1}$,
then there is some dictator $D_{i \to j}$
that essentially contains ${\cal A}$ and ${\cal B}$.
Recall that $p=1/m$ throughout.

\begin{lemma}\label{lem: monotone compression control}
Let $0<\eps\leq 1/15$ and $m\geq 3$.
Suppose $\mathcal{A}, \mathcal{B} \subset [m]^n$ are cross-intersecting with
$\mu(\mathcal{A}), \mu(\mathcal{B})\geq (1-\eps) p$ and
$T_i(\mathcal{A}), T_i (\mathcal{B}) \subset D_{i\ra 1}$.
Then there is $j\in [m]$ such that
$\mu(\mathcal{A} \cap D_{i\ra j}),
\mu(\mathcal{B} \cap D_{i\ra j}) \geq (1-3\eps)p$.
\end{lemma}

\begin{proof}
Without loss of generality we can assume $i = 1$.
As $T_1({\cal A}) \subset D_{1\ra 1}$ we can write
$T_1({\cal A})$ as the disjoint union over $j \in [m]$
of ${\cal A}_j: =\{(1,z): z \in {\cal A}_{1\ra j}$\}; in particular,
${\cal A}_1,\ldots, {\cal A}_m$ are disjoint. Similarly, we may define $\mathcal{B}_1,\ldots,\mathcal{B}_m$
and have that ${\cal B}_1,\ldots, {\cal B}_m$ are disjoint.
For each $j\in [m]$ let $\alpha_j = \mu(\mathcal{A}_{1 \ra j})$
and $\beta_j = \mu(\mathcal{B}_{1 \ra j})$.
Then $\sum _{j} \alpha _j = p^{-1}\mu(\mathcal{A}) \geq 1-\eps$
and $\sum _j \beta _j = p^{-1}\mu(\mathcal{B}) \geq 1- \eps$.
We need to show that for some $j\in [m]$ we have
$\alpha _j , \beta _j \geq 1- 3\eps$.
	
To see this, suppose without loss of generality
that $\alpha = \alpha _1$ is largest among
$\{\alpha _j\}_{j\in [m]} \cup \{\beta _j\}_{j\in [m]}$.
Let ${\cal B}_{\neq 1} := \bigcup _{j\in [2,m]} {\cal B}_j$
and $\beta _{\neq 1} = \mu(\mathcal{B}_{\neq 1})$.
Then $\beta _{\neq 1} = \sum _{j\neq 1} \beta _j
\ge 1- \eps -\beta _1 \geq 1-\alpha _1 - \eps$.
As ${\cal A}_1$ and ${\cal B}_{\neq 1}$ are cross-intersecting,
by Lemma~\ref{lem:hoffman1} we have
\begin{equation*}
\alpha _1 (1- \alpha _1 - \eps) \leq
\alpha _1 \beta _{\neq 1} \leq
(1-\aA_1)(\bB_1+\eps)/(m-1)^2
\le (1-\aA_1)(\aA_1+\eps)/(m-1)^2.
\end{equation*}
Rearranging gives
$\big ((m-1)^2 - 1 \big )\alpha _1 (1-\alpha _1)
\leq \big ( (m-1)^2\alpha _1 + (1-\alpha _1) \big ) \eps
\leq (m-1)^2 \eps$. Thus
$\alpha _1(1-\alpha _1) \leq 4 \eps/3$,
so either $\alpha _1 \leq 3 \eps$
or $\alpha _1 \geq 1-3\eps$.
We will show that the second bound holds.

Suppose otherwise. Then
$\alpha _j \leq 3 \eps$ for all $j \in [m]$.
We can partition $[m]$ as $Q_1 \cup Q_2$
so that for $k=1,2$ we have
$\sum _{j\in Q_k} \alpha _j \geq (1 - \eps- 3\eps)/2
\geq 1/2 - 2 \eps$. Without loss of generality
$\sum _{j\in Q_1} \beta_j \geq 1/2 - \eps$.
Then $\cup _{j\in Q_1} {\cal B}_j$
and $\cup _{j\in Q_2} {\cal A}_j$
cross-intersect and both have densities
at least $1/2 - 2 \eps$ in $[m]^{n-1}$,
which contradicts Lemma~\ref{lem:hoffman1},
as  $\eps \leq 1/15$.
Thus $\alpha _1 \geq 1 - 3\eps $, as required.
Now we apply Lemma~\ref{lem:hoffman1} again to
$\mathcal{A}_1$ and $\mathcal{B}_{\neq 1}$,
which gives
$(1-3\eps)\beta _{\neq 1} \leq
\alpha _1 \beta _{\neq 1} \leq
(1-\alpha _1 + \eps )\beta _1/(m-1)^2
\leq \frac {1}{4} 4 \eps \cdot 1 = \eps$,
so $\beta _{\neq 1} \leq 2 \eps $.
As $\beta _{1} \geq  (1-\eps ) - \beta _{\neq 1}$
this gives $\beta _1 \geq 1- 3 \eps $, completing the proof.
\end{proof}

We conclude this subsection
with the proof of the stability theorem.

\begin{proof}[Proof of Theorem \ref{thm:stability'}]
Let $0 < \dD \ll \eps' \ll \eps, t^{-1}, m^{-1}$.
Suppose $\mathcal{F} \sub [m]^n$ is $t$-intersecting
with $\mu(\mathcal{F})\geq (1-\dD)\mu(\mathcal{S}_{n,m,t})$.
We can assume without loss of generality that for each $i\in [n]$,
the most popular value of $x_i$ for $x\in\mathcal{F}$ is $1$ (otherwise we simply relabel the alphabet in that coordinate).
We set ${\cal F}_0 = {\cal \mathcal{F}}$ and for each
$i\in [n]$ let $\mathcal{F}_{i} = T_i (\mathcal{F}_{i-1})$.
By Fact~\ref{fact:compression},
${\cal F}_n= T(\mathcal{F}_0)$
is $t$-intersecting, compressed,
and $\mu({\cal F}_n) \geq (1-\delta)\mu(\mathcal{S}_{n,m,t})$.
By Claim~\ref{claim:stability_compressed},
$\mu(\mathcal{F}\sm\mathcal{S}) \leq \eps'\mu({\cal S})$
for some copy $\mathcal{S}$
of ${\cal S}_{n,m,t} = {\cal S}_{t,r}[m]^n$,
where $0 \le r \le t$, and $r=0$ if $m>t+1$.
We write $J$ for the set of coordinates
on which it depends, so $|J|=t+2r$.

We define $\eps_i$ for all $i\in [n]$
by $\mu(\mathcal{F}_i \cap  \mathcal{S})
= (1 - \eps_i)\mu(\mathcal{S})$.
We note that
$\mu(\mathcal{F}) = \mu(\mathcal{F}_n)
\leq (1+\eps')\mu(\mathcal{S})$
and $\mu(\mathcal{F}_n \cap \mathcal{S})
\geq (1 - \eps'- \delta)\mu(\mathcal{S})
\geq (1 - 2\eps')\mu(\mathcal{S})$,
so $\eps_n \leq 2\eps'$.
We will show inductively that $\eps_i$
is suitably small for $i=n,n-1,\dots,0$.
To prove the theorem, it suffices
to show $\eps_0<\eps/2$, as
$\mu(\mathcal{F}\sm\mathcal{S})
\leq \mu(\mathcal{F}) - \mu(\mathcal{F}\cap\mathcal{S})
\leq (1+\eps')\mu(\mathcal{S}) - (1-\eps_0)\mu(\mathcal{S})
\leq 2\eps_0\mu(\mathcal{S})$.
	
Note that if $i \notin J$ then $\mathcal{S}$
is \emph{$i$-insensitive}, meaning that
for all $x\in [m]^n$ membership of $x$ in $\mathcal{S}$
does not depend on $x_i$. For such $i$
we have $|T_i(\mathcal{G}) \cap \mathcal{S}|
= |\mathcal{G}\cap \mathcal{S}|$
for any ${\cal G} \sub [m]^n$,
so $\mu(\mathcal{F}_{i-1} \cap \mathcal{S})
= \mu(\mathcal{F}_{i} \cap \mathcal{S}) $,
i.e. $\eps_i = \eps_{i-1}$.
For $i\in J$ we will show that
$\eps_{i-1} \leq 3(t+1)^{3t} \eps_i$.
This will imply $\eps_0 < (3(t+1)^{3t})^{|J|} \eps'
< \eps/2$ as $\eps'\ll \eps$, and so will
suffice to complete the proof of the theorem.

Set $J_i := J \sm \{i\}$. Given $y \in [m]^{J_i}$
and $\mathcal{D} \subset [m]^n$, we use the abbreviation
\begin{equation*}
\mathcal{D}(y) := {\cal D}_{J_i \to y} =
\sett{z \in [m]^{[n]\sm J_i}}{(y,z) \in \mathcal{D}}
\subset [m]^{[n]\sm J_i}.
\end{equation*}
We require the following claim,
showing that if two $J_i$-restrictions
$\mathcal{F}_i(y_1)$ and $\mathcal{F}_i(y_2)$
are close to the same $i$-dictator,
where $y_1,y_2$ have agreement at most $t-1$,
then this is also true of
$\mathcal{F}_{i-1}(y_1)$ and $\mathcal{F}_{i-1}(y_2)$.

\begin{claim}\label{claim:aux}
Suppose $y_1, y_2 \in [m]^{J_i}$ with
${\sf agr}(y_1,y_2) \leq t-1$ and
$\mu(\mathcal{F}_i(y_k) \cap D_{i\ra 1})\geq (1-\xi)p$
for both $k = 1,2$, where $0\leq \xi\leq 1/6$.
Then there is $j\in [m]$ such that both
$\mu(\mathcal{F}_{i-1}(y_k) \cap D_{i\ra j})
\geq (1-3\xi)p$. Moreover, for any $j'\neq j$, both
$\mu(\mathcal{F}_{i-1}(y_k) \cap D_{i\ra j'})<(1-3\xi)p$.
\end{claim}
\begin{proof}
Note that
${\cal F}_{i-1}(y_1)$ and ${\cal F}_{i-1}(y_2)$
are cross-intersecting,
as $\mathcal{F}_{i-1}$ is $t$-intersecting and
${\sf agr}(y_1,y_2) \leq t-1$.
Let $\mathcal{A}_k = \{x \in \mathcal{F}_{i-1}(y_k):
T_i(x) \in \mathcal{F}_{i}(y_k) \cap D_{i\ra 1} \}$
for $k=1,2$. Then ${\cal A}_1$, ${\cal A}_2$ are
cross-intersecting and both $\mu({\cal A}_k)>(1-\xi)p$,
so the existence of $j$ follows
from Lemma~\ref{lem: monotone compression control}.

For the `moreover' part, note that if $j'\neq j$, then
$(\mathcal{F}_{i-1})_{J\rightarrow (y_1,j')}$
and $(\mathcal{F}_{i-1})_{J\rightarrow (y_2,j)}$
are cross-intersecting, and applying Lemma~\ref{lem:hoffman1} gives us that
$\mu((\mathcal{F}_{i-1})_{J\rightarrow (y_1,j')}) < 1-3\xi$. The same argument
works interchanging the roles of $y_1$ and $y_2$, and we get that
$\mu(\mathcal{F}_{i-1}(y_k) \cap D_{i\ra j'})<(1-3\xi)p$.
\end{proof}

Using Claim~\ref{claim:aux} we now bound $\eps_{i-1}$.
We start with the case $r=0$.
Let $\bm{1} \in [m]^{J_i}$ be the all-1 vector.
We have $\mu({\cal F}_i(\bm{1}) \cap D_{i \to 1})
= p^{1-t} \mu({\cal F}_i \cap {\cal S}) = (1-\eps_i)p$,
so by Claim \ref{claim:aux} with $y_1=y_2=\bm{1}$
there is $j \in [m]$ such that
$\mu({\cal F}_{i-1}(\bm{1}) \cap D_{i \to j})
\ge (1-3\eps_i)p$. The most popular value
in $\mathcal{F}_{i-1}$ of coordinate $i$ is $1$
(since this is the case in $\mathcal{F}$), so $j=1$.
We deduce $\mu({\cal F}_{i-1} \cap {\cal S})
= p^{t-1} \mu({\cal F}_{i-1}(\bm{1}) \cap D_{i \to 1})
\ge (1-3\eps_i)p^t$, so $\eps_{i-1} \le 3\eps_i$.

It remains to consider $r \ge 1$.
We have $m\leq t+1$ by
Claim~\ref{claim:stability_compressed}.
For a vector $y\in [m]^n$ and $j\in[m]$, let $y[j]$ be
the set of coordinates of $i$ equal to $j$.
We partition ${\cal S}$
as ${\cal S} = {\cal S}_0 \cup {\cal S}_1$, where
\begin{align*}
& \mathcal{S}_{0} := \big \{ x \in [m]^n:
x|_{J_i} \in \mathcal{G}_{0} \big \},
& \mathcal{G}_{0} :=
\{y \in [m]^{J_i}: \card{y[1]} > t+r-1\},\\
& \mathcal{S}_{1} = \big \{ x \in [m]^n:
x|_{J_i} \in \mathcal{G}_{1} \big \} \cap D_{i\ra 1},
& \mathcal{G}_{1} :=
\{y \in [m]^{J_i}:\card{y[1]} = t+r-1\}.
\end{align*}
As $\mathcal{S}_0$ is $i$-insensitive,
$\mu(\mathcal{F}_{i-1} \cap \mathcal{S}_0)
= \mu(\mathcal{F}_{i} \cap \mathcal{S}_0)$.
Now we wish to show that
$\mu(\mathcal{F}_{i-1} \cap \mathcal{S}_1)$ is large,
i.e.\ that $\mu(\mathcal{F}_{i-1}(y) \cap D_{i \to 1})$
is close to $1$ for each $y \in {\cal G}_1$.
First we show this for ${\cal F}_i$.

\begin{claim} \label{claim:aux2}
$\mu(\mathcal{F}_i(y) \cap D_{i\ra 1})
\geq (1-(t+1)^{3t}\eps_i)p$ for each $y \in {\cal G}_1$.
\end{claim}
\begin{proof}
To see this, we note that
\[ \eps_i \mu({\cal S}) = \mu({\cal S} \sm {\cal F}_i)
\ge \mu({\cal S}_1 \sm {\cal F}_i)
= \sum_{y \in {\cal G}_1}
p^{|J_i|} (p - \mu(\mathcal{F}_i(y) \cap D_{i\ra 1})).\]
Each summand on the right hand side is non-negative,
and $\card{J} = t+2r \leq 3t$, so for each $y \in {\cal G}_1$
we have $p - \mu(\mathcal{F}_i(y) \cap D_{i\ra 1}))
\le p(t+1)^{3t}\eps_i$, so the claim holds.
\end{proof}

Next we prove the corresponding claim for ${\cal F}_{i-1}$,
although at first just with $D_{i \to j}$ for some $j \in [m]$;
the theorem will follow once we show that $j=1$.
We say that $y \in \mathcal{G}_1$ is $j$-good
if $\mu(\mathcal{F}_{i-1}(y) \cap D_{i\ra j})
\geq (1-3(t+1)^{3t}\eps_i)p$.

\begin{claim}
There is some $j \in [m]$ such that
every $y \in \mathcal{G}_1$ is $j$-good.
\end{claim}
\begin{proof}
Note that for any $y,y'\in\mathcal{G}_1$
with ${\sf agr}(y,y')=t-1$, by Claims~\ref{claim:aux}
and~\ref{claim:aux2}
there is some $j \in [m]$ such that
both $y$ and $y'$ are $j$-good.
The claim then follows from the observation
that the graph $G$ whose edges consist of
such pairs $\{y,y'\}$ is connected.
(We can get between
any two elements of $\mathcal{G}_1$
by a sequence of steps where in each step
we change some coordinate from $1$ to another value
and some coordinate from another value to $1$,
and each such step can be implemented
by a path of length two in $G$.)
\end{proof}

It remains to show that $j=1$.
We consider $\mathcal{T} = \mathcal{T}_0 \cup \mathcal{T}_1$,
where
$\mathcal{T}_0 = \sett{y\in[m]^n}{y_{J_i} \in \mathcal{G}_0}
= {\cal S}_0$ and
$\mathcal{T}_1 = \sett{y\in[m]^n}
{y_{J_i}\in\mathcal{G}_1, y_i = j}$.
Recalling that
$\mu(\mathcal{F}_{i-1} \cap \mathcal{S}_0)
= \mu(\mathcal{F}_{i} \cap \mathcal{S}_0)$,
by the previous claim we deduce
$\mu(\mathcal{F}_{i-1}\cap \mathcal{T})
\geq (1-3(t+1)^{3t}\eps_i) \mu(\mathcal{T})$, so
      \[
      \mu(\mathcal{F}_{i-1}\sm \mathcal{T})\leq
      \mu(\mathcal{F}_{i-1}) - \mu(\mathcal{F}_{i-1}\cap \mathcal{T})
      \leq 3(t+1)^{3t}\eps_i + \eps'
      \leq 4(t+1)^{3t}\eps_i,
      \]
where in the second inequality we used
$\mu(\mathcal{F}_{i-1}) = \mu(\mathcal{F})\leq
(1+\eps')\mu(\mathcal{S}) = (1+\eps')\mu(\mathcal{T})$.
Hence, for $\ell\neq j$ the fraction of $x\in \mathcal{F}_{i-1}$
such that $x_i = \ell$ is at most $4(t+1)^{3t}\eps_i + q(\ell)$,
where $q(\ell)$ is the fraction of $x\in\mathcal{T}$
that have $x_i = \ell$; by symmetry, this value is the same for all $\ell\neq j$,
and we denote it by $q$.
The fraction of $x\in \mathcal{F}_{i-1}$ such that $x_i = j$ is,
for the same reasons, is at least $q(j) - 4(t+1)^{3t}\eps_i$.
But $q(j) \geq q + \mu(\mathcal{G}_1)\geq q + (t+1)^{-3t}$,
so $q(j)- 4(t+1)^{3t}\eps_i \geq q + 4(t+1)^{3t}\eps_i$
(we can ensure $\eps_i < 1/8$). Thus $j$ is the most popular
value of coordinate $i$ in ${\cal F}_{i-1}$, so $j=1$, as required.
\end{proof}

\subsection{Bootstrapping}

We conclude this part by proving Theorem \ref{thm:boot},
which completes the proof of Theorem~\ref{thm:exact_small_m}.
We will use compressions to reduce to the cube,
so we start with some remarks in this setting.

We consider $\power{n}$
equipped with the uniform measure $\mu$.
Suppose $\mathcal{A},\mathcal{B}\subset\power{n}$.
We say ${\cal A}, {\cal B}$ are
cross-intersecting if for any
$x\in\mathcal{A}$, $y\in\mathcal{B}$ there is $i\in[n]$
such that $x_i = y_i = 1$.
We say ${\cal A}, {\cal B}$ are
cross-agreeing if for any
$x\in\mathcal{A}$, $y\in\mathcal{B}$
there is $i\in[n]$ such that $x_i = y_i$.
Clearly if ${\cal A}, {\cal B}$ are
cross-intersecting then they are cross-agreeing.
We have the following easy fact,
which is immediate from the observation that if
${\cal A}, {\cal B}$ are cross-agreeing
and $x+y={\bf 1}$ (the all-1 vector) then
we cannot have $x \in {\cal A}$ and $y \in {\cal B}$.

\begin{fact}\label{fact:cross_int}
  If $\mathcal{A},\mathcal{B}\subset \power{n}$
  are cross-agreeing then
  $\mu(\mathcal{A}) + \mu(\mathcal{B}) \leq 1$.
\end{fact}

We also require the following isoperimetric lemma
of Ellis, Keller and Lifshitz  \cite{EKL}.

\begin{lemma}\label{lemma:isoperimetric_bump}
  Suppose $0\leq p\leq q\leq 1$, $\alpha\geq 0$
  and $\mathcal{F}\subset\power{n}$ is monotone.
  If $\mu_p(\mathcal{F})\geq p^{\alpha}$
  then $\mu_q(\mathcal{F}) \geq q^{\alpha}$.
\end{lemma}

\begin{proof}[Proof of Theorem \ref{thm:boot}]
Let ${\cal G}, {\cal H} \sub [m]^n$
with $\mu({\cal H}) = m^{-t} \eps$
and $\mu({\cal G}) > 1- C\eps$,
where $0\ll \eps \ll t^{-1},C^{-1}$.
Suppose for contradiction that
${\cal G}$ and ${\cal H}$ are cross-agreeing.
Let ${\cal G}' = \wt{T({\cal G})}$
and ${\cal H}' = \wt{T({\cal H})}$,
where $T$ is the compression operator
and the operator $\mathcal{F} \ra \wt{\mathcal{F}}$
is from Definition~\ref{def:compress_reduce}.
By Facts~\ref{fact:compression} and
\ref{fact:properties_compressed_reduction},
${\cal G}'$ and ${\cal H}'$
are monotone and cross-intersecting
with $\mu_p({\cal G}') \ge \mu({\cal G})$
and $\mu_p({\cal H}') \ge \mu({\cal H})$, where $p=1/m$.

Now we consider ${\cal G}'$ and ${\cal H}'$
under the uniform measure $\mu=\mu_{1/2}$.
By monotonicity we have
$\mu({\cal G}') \ge \mu_p({\cal G}') > 1- C\eps$.
By Lemma \ref{lemma:isoperimetric_bump},
$\mu({\cal H}') \ge \mu_p({\cal H}')^{\log_p(1/2)}$,
so $\log_2 \mu({\cal H}')
\ge \log_2(m^{-t} \eps)/\log_2(m)$,
giving $\mu({\cal H}') \ge 2^{-t} \eps^{0.7}$,
as $m \ge 3$ and $1/\log_2(3)<0.7$.
However, $1- C\eps + 2^{-t} \eps^{0.7} > 1$
as  $\eps \ll t^{-1},C^{-1}$,
which contradicts Fact \ref{fact:cross_int}.
\end{proof}

\newpage

\part{Large alphabets}

In this part we prove our main result Theorem \ref{thm:main}
when the alphabet size $m$ is large, i.e.\ $m>m_0(t)$. We note
that in this case ${\cal S}_{n,m,t} = \{x \in [m]^n: x_T = \alpha \}$ for some $T \subset [n]$ with $|T| = t$ and $\alpha \in [m]^T$.
As mentioned in the introduction, we cannot achieve such
a strong pseudorandomness condition in our regularity lemma
as in the case of fixed $m$, so we settle for a weaker
notion of `uncapturability'. We also recall that the strategy
for fixed $m$ based on cross-agreements between pieces
of the regularity decomposition cannot work
when $m$ is `huge' (exponential in $n$),
so we give a different (more combinatorial) argument
for this case in Section \ref{sec:huge}.
The bulk of this part is concerned with the case
that $m$ is `moderate' (large but not huge),
which we analyse in the second section,
via a version of the cross-agreement strategy
implemented by a gluing argument
that exploits expansion under another
pseudorandomness condition, namely globalness.
The tools for this are developed in the first section,
in which we study our two pseudorandomness conditions
(uncapturability and globalness) and establish
the small-set expansion for global functions
via a refined version of our global
hypercontractivity inequality from \cite{KLLM}.

\section{Tools}

This section concerns various properties of
the pseudorandomness notions of
uncapturability and globalness,
particularly a regularity lemma for uncapturability
and a small set expansion property for global functions,
which is analogous to Theorem \ref{thm:MOSMarkov}.
The latter will be established via a corresponding
statement for the noise operator, which will be
proved by a refined form of our global
hypercontractivity inequality.
Along the way we record various facts needed here and
later concerning Markov chains and
the Efron-Stein theory of orthogonal decompositions.

\subsection{Uncapturability and globalness}

This subsection contains the definitions
and basic properties of the two
key pseudorandomness conditions used in this part.
We start with uncapturability, which is the condition
that will appear in the regularity lemma
in the next subsection.
Recall that for $\mathcal{F}\subset[m]^n$
and $\alpha\in[m]^R$ for some $R\sub [n]$ we write
${\cal F}[\aA] = \{ x \in {\cal F}: x_R = \aA \}$ and
${\cal F}_{R \to \aA} = {\cal F}(\aA)
= \{ x \in [m]^{[n]\sm R}: (x,\aA) \in {\cal F} \}$.
We also write $D_{R \to \aA} = \{x \in [m]^n: x_R=\aA\}$,
which is a subcube of co-dimension $|R|$,
which we refer to as a `dictator' if $|R|=1$. For a collection
of subcubes $\mathcal{D}$, we denote by $\bigcup \mathcal{D}$
the union of these subcubes, i.e. $\bigcup_{D\in\mathcal{D}} D$.

\begin{definition}
We say $\mathcal{F}\subset[m]^n$ is $(r,\eps)$-capturable
if there is a set ${\cal D}$ of at most $r$ dictators
with $\mu({\cal F}\sm \bigcup {\cal D})\leq \eps$.
Otherwise, we say $\mathcal{F}$ is $(r,\eps)$-uncapturable.
\end{definition}

Now we define the stronger
(see Claim \ref{claim:global_to_uncap})
condition of globalness.

\begin{definition}\label{def:global}
We say $f:[m]^n \to\mb{R}$ is $(r,\eps)$-global
if for any $R \sub [n]$ with $|R| \le r$ and $a \in [m]^R$
we have $\norm{f_{R\ra a}}_2^2\leq \eps$.
We say ${\cal F} \sub [m]^n$ is $(r,\eps)$-global
if its characteristic function is $(r,\eps)$-global.
\end{definition}

Most of this section will be devoted to the proof
of the following small set expansion property
for global functions,
which is analogous to Theorem~\ref{thm:MOSMarkov}.
We remark that we will later use Theorem~\ref{thm:SSE} to prove that random gluings
significantly increase the measure of global families.

\begin{thm}\label{thm:SSE}
For any $\lL>0$ there is $c>0$ such that
the following holds for Markov chains
$T_i$ on $\OO_i$ with $\lL_*(T_i) \ge \lL$
for all $i \in [n]$
and consecutive random states $x,y$ of the
stationary chain for the product chain $T$ on $\OO$.
If ${\cal F} \sub \Omega$
is $(\log(1/\mu), \mu^{1-c})$-global
with $\mu\in (0,1/16)$ then
$\mb{P}(x \in {\cal F}, y \in {\cal F})
\leq \mu^{1+c}$.
\end{thm}

We begin by giving two simple relations between uncapturability and globalness that will be useful for us.
The first property asserts that globalness implies very strong uncapturability.
\begin{claim}\label{claim:global_to_uncap}
If $\gamma \in (0,1)$ and $\es \ne \mathcal{G} \sub [m]^n$
is $(1,\mu(\mathcal{G})/\gamma)$-global then $\mathcal{G}$
is $(\gamma m /4, \mu(\mathcal{G})/2)$-uncapturable.
\end{claim}
\begin{proof}
Suppose ${\cal D}$ is a set of  dictators with
$\mu({\cal G}\sm \bigcup {\cal D}) \leq \mu(\mathcal{G})/2$.
We need to show $|{\cal D}| > \gamma m/4$. By assumption $\mu(\mathcal{G}_{i \to a})
\leq \mu({\cal G})/\gamma$ for each
$D_{i \to a} \in {\cal D}$,
so by a union bound
\[ \mu({\cal G})/2 \leq \mu(\mathcal{G}\cap \bigcup {\cal D})
\leq \sum\limits_{D_{i \to a}\in {\cal D}}{\mu({\cal G}\cap D_{i\to a})}
 =\sum\limits_{D_{i \to a} \in {\cal D}}
 {m^{-1} \cdot \mu(\mathcal{G}_{i\ra a})}
  \leq |{\cal D}|m^{-1} \frac{\mu(G)}{\gamma}.\]
 Thus $|{\cal D}| \geq \gamma m/2 > \gamma m/4$, as required.
\end{proof}

The second property shows that any family $\mathcal{G}$ with
significant measure can be made global by taking small restrictions.
\begin{lemma}\label{lem:make_global}
Let $0 <\gamma <1$ and $r,m,n\in\mathbb{N}$.
For any $\mathcal{G} \sub [m]^n$ there is
$R\subset[n]$ and $\alpha\in[m]^R$ with
$|R|\leq r\log_{\gamma^{-1}}(\mu(\mathcal{G})^{-1})$
such that ${\cal G}' = {\cal G}_{R\ra \alpha}$
is $(r, \mu(\mathcal{G}')/\gamma)$-global
with $\mu({\cal G}') \ge \mu({\cal G})$.
\end{lemma}
\begin{proof}
Starting with ${\cal G}_0={\cal G}$,
for each $i \ge 0$, if ${\cal G}_i$
is not $(r, \mu(\mathcal{G}_i)/\gamma)$-global
we let ${\cal G}_{i+1} = ({\cal G}_i)_{R_i \to \aA_i}$
for some $R_i \sub [n]$ with $|R_i| \le r$
and $\aA_i\in[m]^{R_i}$ such that
$\mu(\mathcal{G}_{i+1})\geq \mu(\mathcal{G}_i)/\gamma$;
such a restriction exists by definition.
As all measures are bounded by $1$ there can be at most
$\log_{\gG^{-1}} (\mu(\mathcal{G})^{-1})$ iterations,
at which point we terminate with
${\cal G'} = {\cal G}_{R\ra \alpha}$
with the stated properties.
\end{proof}

\subsection{The uncapturable code regularity lemma}

In this subsection we prove
the following regularity lemma
which approximately decomposes any code
into pieces corresponding
to uncapturable restrictions.

\begin{lemma}\label{lem:regularity_large_m}
Let $r, k,m\in\mathbb{N}$ and $\eps\geq 1/m$.
For any $\mathcal{F}\subset [m]^n$
there is a collection $\mathcal{D}$ of
at most $r^k$ subcubes of co-dimension at most $k$
such that ${\cal F}_{R \to \aA}$ is
$(r,\eps \mu(D)^{-1} m^{-k})$-uncapturable
for each $D = D_{R \to \aA} \in {\cal D}$ and
$\mu({\cal F} \sm \bigcup {\cal D})
\le 3 r^{k+1} \eps m^{-k}$.
\end{lemma}

\begin{proof}
We may assume $\mathcal{F}$ is $(r,\eps m^{-k})$-capturable,
otherwise the lemma holds with ${\cal D} = \{[m^n]\}$.
We apply the following iterative process for $s = 1,\ldots,k$.
\begin{itemize}
\item
We let $\mathcal{D}'_{s-1}$
be the set of  $D = D_{R \to \aA} \in {\cal D}_{s-1}$ such that
${\cal F}_{R \to \aA}$ is $(r,\eps \mu(D)^{-1} m^{-k})$-capturable,
where for $s=1$ we let
${\cal D}'_0={\cal D}_0=\{D_{\es \to \es}\}=\{[m]^n\}$.
\item
For each $D = D_{R\ra \alpha} \in {\cal D}'_{s-1}$,
by definition of capturability we can fix
a set ${\cal D}[D]$ of at most $r$ dictators such that
$\mu( \mathcal{F}_{R \to \aA} \sm \bigcup {\cal D}[D] )
 \leq \eps \mu(D)^{-1} m^{-k}$.
\item
We define ${\cal D}_s= \{ D_{(R,i)\ra (\alpha,a)}:
D = D_{R\ra \alpha} \in \mathcal{D}'_{s-1},
D_{i \to a} \in {\cal D}[D] \}$.
\end{itemize}

At the end of the process,
we let $\mathcal{D}'_k\subset\mathcal{D}_k$
be the set of $D = D_{R \to \aA} \in \mathcal{D}_k$
such that $\mathcal{F}_{R \to \aA}$ is $(r,\eps)$-capturable.
We will show that
$\mathcal{D} = (\mathcal{D}_1\sm\mathcal{D}'_1)\cup\ldots\cup(\mathcal{D}_k\sm\mathcal{D}'_k)$
satisfies the requirements of the lemma.

Clearly, for all $D\in \mathcal{D}$ we have that
${\cal F} \cap D$ is
$(r,\eps \mu(D)^{-1} m^{-k})$-uncapturable,
and $|\mathcal{D}| \le r^k$
as we explore at most this many
subcubes during the above process.
We will bound $\mu({\cal F} \sm \bigcup {\cal D})$
by $\mu({\cal F} \sm \bigcup ({\cal D} \cup {\cal D}'_k))
+ \mu({\cal F} \cap \bigcup {\cal D}'_k)$.

For the first term in the bound, we write
${\cal F} \sm \bigcup ({\cal D} \cup {\cal D}'_k)
= \cup_{s=0}^{k-1} {\cal E}_s$, where each
\[ {\cal E}_s = \bigcup \{
\mathcal{F}_{R \to \aA} \sm \bigcup {\cal D}[D]
: D = D_{R \to \aA} \in {\cal D}'_s \}. \]
By definition, $\mu({\cal E}_s)
\leq \sum_{D \in {\cal D}'_s}
\mu(D) \cdot \eps \mu(D)^{-1} m^{-k}
= |{\cal D}'_s| \eps m^{-k}$, so
\[ \mu({\cal F} \sm \bigcup ({\cal D} \cup {\cal D}'_k))
 \le \eps m^{-k} \sum_{s=0}^{k-1} |{\cal D}'_s|
 \le r^k  \eps m^{-k} .\]
For the second term in the bound, we note that
if $D_{R \to \aA} \in \mathcal{D}'_k$ then
$\mathcal{F}_{R\to\aA}$ is $(r,\eps)$-capturable,
so has measure is at most $r\frac{1}{m} + \eps\leq (r+1)\eps$.
Thus
  \[ \mu({\cal F} \cap \bigcup {\cal D}'_k)
  \leq \sum_{D \in {\cal D}'_k} \mu(D) (r+1)\eps
  \leq r^k m^{-k}  (r+1)\eps. \]
We deduce
$\mu({\cal F} \sm \bigcup {\cal D}) \leq
\mu({\cal F} \sm \bigcup ({\cal D} \cup {\cal D}'_k))
+ \mu({\cal F} \cap \bigcup {\cal D}'_k)
\leq 3 r^{k+1} \eps m^{-k}$.
\end{proof}

\subsection{Markov Chains and Orthogonal Decompositions}

This subsection contains some
further theory of Markov Chains,
Efron-Stein orthogonal decompositions
and a general form of the Hoffman bound
for cross-intersecting families
in any product space.
The results are somewhat standard,
but we include details for the convenience
of the reader.

Let $T$ be a Markov chain on $S$
with stationary distribution $\nu$.
The absolute spectral gap $\lL_*=\lL_*(T)$ is
\[ (1-\lL_*)^2 = \sup \{ \mb{E}(Tf)^2:
 \mb{E}f=0, \mb{E}f^2=1 \}.\]
Here expectations are with respect to $\nu$.
If $T$ is reversible we can also view $\lL_*$
as the minimum value of $1-|\lL|$
over all eigenvalues $\lL \ne 1$.
We start with a general lower bound for $\lL_*$.

\begin{lemma}\label{lem:gap}
Let $T$ be a Markov chain on $S$
with stationary distribution $\nu$
such that $T_{ab} \ge \aA \nu(b)$
for every $a,b \in S$.
Then $\lL_*(T) \ge \aA$.
\end{lemma}

\begin{proof}
By assumption, $S_{ab} := T_{ab}-\aA \nu(b) \ge 0$,
with $\sum_b S_{ab} = 1 - \aA$ and
$\sum_a \nu(a) S_{ab} = (1 - \aA)\nu(b)$.

If $\mb{E}f=0$ and $\mb{E}f^2=1$ then
by Cauchy-Schwarz
\begin{align*}
\mb{E}(Tf)^2 & = \sum_a \nu(a) (Tf)(a)^2
= \sum_a \nu(a) (\sum_b T_{ab} f(b))^2
= \sum_a \nu(a) \big (\sum_b S_{ab} f(b) \big )^2 \\
& \leq \sum_a \nu(a) \big (\sum_b S_{ab} \big ) \big ( \sum_b S_{ab} f(b)^2 \big )
= (1-\aA)  \sum_a \nu(a) \sum_b S_{ab} f(b)^2  \\
& = (1-\aA)  \sum_b f(b)^2 \sum_a \nu(a)  S_{ab}
=  (1-\aA)^2  \sum_b \nu(b) f(b)^2
= (1-\aA)^2. & \qedhere
\end{align*}
\end{proof}

Now we consider Markov chains $T_i$ acting on $\OO_i$
for $i \in [n]$ and their tensor product
$T = T_1 \otimes \dots \otimes T_n$ acting on
$\OO = \OO_1 \otimes \dots \otimes \OO_n$,
with transition matrix
$T_{xy} = \prod_{i=1}^n (T_i)_{x_iy_i}$.
The stationary distribution of $T$ is
$\nu = \nu_1 \otimes \dots \otimes \nu_n$,
where each $\nu_i$ is stationary for $T_i$.
We will often have $\OO=[m]^n$ and $\nu$ uniform,
but we will also require the general setting.

We use the Efron-Stein orthogonal decomposition
(see e.g.\ \cite[Section 8.3]{Odonnell}):
for any $f \in L^2(\OO,\nu)$ we can write
$f = \sum_{S \sub [n]} f^{=S}$,
where each $f^{=S}$ is characterised by the properties
that it only depends
on coordinates in $S$ and that it is orthogonal to any function
which depends only on some set of coordinates not containing $S$;
in particular, $f^{=S}$ and $f^{=S'}$ are orthogonal for $S \ne S'$.
We have similar Plancherel / Parseval relations
as for Fourier decompositions, namely
$\inner{f}{g} =  \sum_S \Expect{}{f^{=S} g^{=S}}$,
so $\mb{E}[f^2] = \sum_S \Expect{}{(f^{=S})^2}$.
Explicitly, we let
$f^{\sub J}(x) = \mb{E}_{y \sim \nu}
[f(y )\mid y_{\ov{J}} = x_{\ov{J}} ] $ and then we have
$f^{=S} = \sum_{J \sub S} (-1)^{|S \sm J|} f^{\sub J}$
(the inclusion-exclusion formula for
$f^{\sub J} = \sum_{S \sub J} f^{=S}$).
We note the following identity which is immediate
from this construction.

\begin{fact} \label{fact:restrictOD}
For $S \sub T \sub [n]$, $x \in \OO_S$
and $f \in L^2(\OO,\nu)$ we have
$(f^{=T})_{S \to x} = (f_{S \to x})^{=T \sm S}$.
\end{fact}

We require the following general form
of the well-known Hoffman bound
(the uniform case was used in Part I,
see Lemma \ref{lem:hoffman1}). We include the
proof for completeness.

\begin{lemma}\label{lem:hoffman}
Let $\nu = \prod_{i=1}^n \nu_i$ be a product
probability measure on $[m]^n$
such that $\nu_i(x) \leq \lL \leq 1/2$
for all $i \in [n]$, $x \in [m]$.
Suppose $\mathcal{G}_1,\mathcal{G}_2\subset[m]^n$
are cross-intersecting with $\nu({\cal G}_i)=\aA_i$
for $i=1,2$. Then
\[ \aA_1\aA_2\leq\left(\frac{\lambda}{1-\lambda}\right)^2
(1-\aA_1)(1-\aA_2).  \]
  \end{lemma}

The proof of Lemma \ref{lem:hoffman}
requires the following estimate.

\begin{claim}\label{claim:op_contract}
Let $U_i$ be Markov chains on $\OO_i$ for $i \in [n]$
and let $U$ be the product chain on $\OO$.
For any $f:\OO\to\mb{R}$ and $S\sub [n]$ we have
$\norm{U f^{=S}}_2\leq \norm{f^{=S}}_2
\prod\limits_{i\in S}(1-\lL_*(U_i))$.
\end{claim}
\begin{proof}
Since $f^{=S}$ does not depend on variables outside $S$,
we may assume without loss of generality that $S=[n]$.
We introduce interpolating operators
$U_{\leq j} = \bigotimes_{i=1}^{j}{U_i}
\otimes \bigotimes_{i=j+1}^n{I_i}$,
where $I_i$ is the identity,
and $g_j = U_{\leq j} f^{=S}$
for $0\leq j\leq n$. It suffices to show
$\norm{g_j}_2\leq (1-\lL_*(U_j))\norm{g_{j-1}}_2$
for $j \in [n]$.

We calculate $\norm{g_j}_2^2 = \mb{E} (U_jg_{j-1})^2$
by conditioning on $z\in\Omega_{[n]\sm\set{j}}$, i.e.\
\[ \mb{E} (U_jg_{j-1})^2 =
\Expect{{\bf z}\sim\nu_{[n]\sm\set{j}}}
{\mb{E} (U_{j} h_{\bf z})^2 },\]
where $h_z := (g_{j-1})_{[n]\sm\set{j}\ra z}
\in L^2(\OO_j,\nu_j)$. Note that for each $z$,
\[
\Expect{x\sim\nu_j}{h_z(x)}
=\Expect{x\sim\nu_j}{(U_{\leq j-1} f^{=S})(z,x)}
=U_{\leq j-1}\left(\Expect{x\sim\nu_j}{f^{=S}(z,x)}\right)
=U_{\leq j-1} 0 = 0,
\]
so $\mb{E} (U_j h_z)^2
\le (1-\lL_*(U_j))^2 \mb{E} h_z^2$.
As $\mb{E}_{\bf z} \mb{E} h_{\bf z}^2
= \mb{E} g_{j-1}^2$ we get
$\norm{g_j}_2\leq (1-\lL_*(U_j))\norm{g_{j-1}}_2$,
as required.
\end{proof}

\begin{proof}[Proof of Lemma~\ref{lem:hoffman}]
For each $i \in [n]$ we consider
the Markov chain $U_i$ on $[m]$
with transition probabilities $(U_i)_{xx}=0$
and $(U_i)_{xy} = \nu_i(y)/(1-\nu_i(x))$ for $y \ne x$.
We claim that
\begin{equation} \label{eq:hoff}
1-\lL_*(U_i) \le  \frac{\lambda}{1-\lambda}.
\end{equation}
This holds as for any $f \in L^2([m],\nu_i)$
with $\mb{E}f=0$ and $\mb{E}f^2=1$ we have
\begin{align*}
&    \norm{U_i f}_2^2
    = \sum\limits_{x}{\nu_i(x)\left(\sum\limits_{y\neq x}{\frac{\nu_i(y)}{1-\nu_i(x)} f(y)}\right)^2}
    = \sum\limits_{x}{\frac{\nu_i(x)}{(1-\nu_i(x))^2}\left(\sum\limits_{y\neq x}{\nu_i(y) f(y)}\right)^2} \\
& = \sum\limits_{x}{\frac{\nu_i(x)}{(1-\nu_i(x))^2}(\nu_i(x)f(x))^2}
    \leq \left(\frac{\lambda}{1-\lambda}\right)^2\sum\limits_{x}{\nu_i(x) f(x)^2}
    =\left(\frac{\lambda}{1-\lambda}\right)^2.
\end{align*}
Next we note that if ${\bf x} \sim \nu$
and ${\bf y} \sim U {\bf x}$ then
$\mb{P}({\bf x} \in {\cal G}_1, {\bf y} \in {\cal G}_2)=0$,
as ${\sf agr}({\bf x},{\bf y})=0$ by definition of $U$,
but $\mathcal{G}_1,\mathcal{G}_2$ are cross-intersecting
by assumption. We can also write this probability as
$\inner{g_1}{Ug_2}$, where $g_1,g_2\colon [m]^n\to\power{}$
are the indicator functions of $\mathcal{G}_1,\mathcal{G}_2$.
By orthogonality and Cauchy-Schwarz,
    \[
    0=\sum\limits_{S\subset[n]}\inner{g_1^{=S}}{U g_2^{=S}}
    = \aA_1\aA_2 + \sum\limits_{S\neq\emptyset}{\inner{g_1^{=S}}{U g_2^{=S}}}
    \geq \aA_1\aA_2  - \sum\limits_{S\neq\emptyset}{\norm{g_1^{=S}}_2\norm{U g_2^{=S}}_2}.
    \]
By Claim~\ref{claim:op_contract} and \eqref{eq:hoff}
we have $\norm{U g_2^{=S}}\leq\left(\frac{\lambda}{1-\lambda}\right)^{\card{S}}\norm{g_2^{=S}}$, so
    \[
    \aA_1\aA_2
    \leq \sum\limits_{S\neq\emptyset}{\left(\frac{\lambda}{1-\lambda}\right)^{\card{S}}\norm{g_1^{=S}}_2\norm{g_2^{=S}}_2}
    \leq \frac{\lambda}{1-\lambda}\sum\limits_{S\neq\emptyset}{\norm{g_1^{=S}}_2\norm{g_2^{=S}}_2}.
    \]
By Cauchy-Schwarz and Parseval
\[
\left( \sum\limits_{S\neq\emptyset}{\norm{g_1^{=S}}_2\norm{g_2^{=S}}_2} \right)^2
\leq \sum\limits_{S\neq\emptyset}{\norm{g_1^{=S}}_2^2}
\sum\limits_{S\neq\emptyset}{\norm{g_2^{=S}}_2^2}
= \Var(g_1) \Var(g_2)
= \aA_1(1-\aA_1)\aA_2(1-\aA_2). \]
We deduce $(\aA_1\aA_2)^2
\leq\left(\frac{\lambda}{1-\lambda}\right)^2
\aA_1\aA_2(1-\aA_1)(1-\aA_2)$, as required.
\end{proof}

\subsection{Small set expansion via noise stability}

The goal for the remainder of this section
is to prove Theorem \ref{thm:SSE}
concerning global small set expansion.
We start by reducing it to the case
of a particular Markov chain,
namely that given by the noise operator,
which we will now define.
Let $\nu = \prod_{i=1}^n \nu_i$ be a product
probability measure on $\OO = \prod_{i=1}^n \OO_i$.
Fix $\rho\in[0,1]$. We let $\T_i$ be the Markov chain
on $\OO_i$ with transition probabilities
$(T_i)_{xy} = \rho 1_{y=x} + (1-\rho)\nu_i(y)$,
i.e.\ from any state $x$ we stay at $x$
with probability $\rho$ or otherwise
move to a random state according to $\nu_i$.
We let $\T$ be the product chain on $\OO$.
We also write $\T=\T_\rho$.
We call $\T_\rho$ the {\em noise operator}
when we think of it as an operator on $L^2(\OO,\nu)$
via $(\T_{\rho} f)(x)
= \Expect{{\bf y}\sim \T_{\rho} x}{f({\bf y})}$.

Recall that in Theorem \ref{thm:SSE} we want to bound
$\mb{P}(x \in {\cal F}, y \in {\cal F})$ for some
${\cal F} \sub \Omega$, when $x$ and $y$ are
consecutive states of the stationary chain
for some product chain $U$ on $\OO$.
The analytic form is to bound $\inner{f}{Uf}$
where $f$ is the characteristic function of ${\cal F}$.
We will soon see that this can be bounded by an
analogous expression in terms of the noise operator,
i.e.\ ${\sf Stab}_{\rho}(f) := \inner{f}{\T_{\rho} f}$,
which is called the {\em noise stability} of $f$.
For future reference we note the following estimate
showing that a bound on the noise stability
for any given $\rho>0$ implies one for all $\rho<1$.

\begin{lemma} \label{lem:changenoise}
${\sf Stab}_{\rho}(f)
\le \norm{f}_2^{2(1-1/t)} {\sf Stab}_{\rho^t}(f)^{1/t}$
whenever $t=2^d$ with $d \in \mb{N}$.
\end{lemma}

\begin{proof}
By  Cauchy-Schwarz we have
\[ {\sf Stab}_{\rho}(f) = \inner{f}{T_{\rho} f}
  \leq \norm{f}_2 \norm{T_{\rho} f}_2
= \norm{f}_2 \sqrt{{\sf Stab}_{\rho^2}(f)}.\]
The lemma follows by iterating this estimate.
\end{proof}

We also need the following well-known formulae
for the noise operator and stability.

\begin{fact}\label{fact:noise+stab_formula}
  $\T_{\rho} f(x)
  = \sum\limits_{S\subset[n]}{\rho^{\card{S}} f^{=S}(x)}$ and
  ${\sf Stab}_{\rho}(f)
  = \sum\limits_{S\subset[n]}{\rho^{\card{S}}\norm{f^{=S}}_2^2}$.
\end{fact}

The following lemma reduces showing small set expansion
of a general chain $U$ to that of the noise operator,
provided that we have a uniform lower bound on the
absolute spectral gap in each coordinate.

\begin{lemma}\label{lem:op_to_stab}
Let $U = \prod_{i=1}^n U_i$ be a product chain
on $\OO = \prod_{i=1}^n \OO_i$
with each $\lL_*(U_i) \geq \lL$.
Then for all $f:\Omega\to\mathbb{R}$ we have
$\inner{f}{U f}\leq {\sf Stab}_{1-\lL}(f)$.
\end{lemma}

\begin{proof}
We use the orthogonal decomposition
$f = \sum_{S \sub [n]} f^{=S}$.
We note that $\inner{f^{=S}}{Uf^{=T}}$
can only be non-zero if $S=T$,
as $U f^{=T}$ only depends on coordinates in $T$.
Thus \[  \inner{f}{U f}
  =\sum\limits_{S\sub [n]}{\inner{f^{=S}}{U f^{=S}}}
  \leq \sum\limits_{S\sub[n]}{\norm{f^{=S}}_2\norm{U f^{=S}}_2}.  \]
Applying Claim \ref{claim:op_contract}
and Fact~\ref{fact:noise+stab_formula} completes the proof.
\end{proof}

By Lemma \ref{lem:op_to_stab}, to prove
Theorem \ref{thm:SSE} it remains to prove
the following corresponding global small set expansion
theorem for the noise operator.

\begin{thm}\label{thm:global_stab2}
For every $\rho < 1$  there is $c>0$ such that if
$f:\OO\to\power{}$ is $(\log(1/\mu), \mu^{1-c})$-global
with $\mu\in (0,1/16)$ then
${\sf Stab}_{\rho}(f)\leq \mu^{1+c}$.
\end{thm}

A key ingredient in the proof is the following
lemma proved in the next subsection
via global hypercontractivity.
First we introduce some notation.
Given an orthogonal decomposition
$f = \sum_{S \sub [n]} f^{=S}$ and $r \ge 0$
we write $f^{\le r} = \sum_{|S| \le r} f^{=S}$
and $f^{>r}=f-f^{\le r}$. We say $f$ has degree
(at most) $r$ if $f=f^{\le r}$.

\begin{lemma}\label{corr:hypercontract_global}
For any $\rho\leq 1/80$,
if $f\colon\Omega\to\mathbb{R}$
is $(r,\beta)$-global of degree $r$ then
\[  \norm{\T_{\rho} f}_4 \leq \beta^{1/4} \norm{f}_2^{1/2}. \]
\end{lemma}

\begin{proof}[Proof of Theorem \ref{thm:global_stab2}]
We start by showing that there exist $\rho',c'>0$
such that the statement of the theorem holds
with $(\rho',c')$ in place of $(\rho,c)$.
We take $\rho'=2^{-200}$ and $c'=1/100$.
First we note that by globalness
(applied with no restriction) we have
$\mu(f)=\mb{E}[f^2] \le \mu^{.99}$.
Let $d = \round{c' \log(1/\mu)}$. We have
  \[
  {\sf Stab}_{\rho}(f)
  = \sum\limits_{S\subset[n]}{\rho^{\card{S}} \norm{f^{=S}}_2^2}
  \leq \sum\limits_{\card{S}\leq d}{\rho^{\card{S}}
  \norm{f^{=S}}_2^2}   +\rho^{d+1}\norm{f^{>d}}_2^2
  =\inner{f}{T_{\rho}f^{\leq d}}
  + \rho^{d+1}\norm{f^{>d}}_2^2.  \]
Clearly $\rho^{d+1}\norm{f^{>d}}_2^2
\le 2^{-2\log(1/\mu)} = \mu^2$.
By Holder's inequality
\[ \inner{f}{T_{\rho}f^{\leq d}}\leq
  \norm{f}_{4/3} \norm{T_{\rho} f^{\leq d}}_{4}
\leq \mu^{.99} \norm{f}_2^{1/2}
\le (\mu^{.99})^{3/2}, \]
using Lemma \ref{corr:hypercontract_global}
and $\norm{f}_{4/3}=\mu(f)^{3/4}$ (as $f$ is Boolean),
so ${\sf Stab}_{\rho}(f)
\leq (\mu^{.99})^{3/2} + \mu^2\leq \mu^{1.01}$.

Now we will deduce the full version
of Theorem \ref{thm:global_stab2},
i.e.\ for any $\rho<1$ there is $c>0$
such that the statement holds.
We let $d = \ceil{\log(\rho/\rho')}$,
$t=2^d$ and $c = c'/4t$.
By Lemma \ref{lem:changenoise}
we have ${\sf Stab}_{\rho}(f) \leq
\norm{f}_2^{2(1-1/t)} {\sf Stab}_{\rho^t}(f)^{1/t}$.
We have
${\sf Stab}_{\rho^t}(f) \le {\sf Stab}_{\rho'}(f)$
by monotonicity of $\rho \mapsto {\sf Stab}_{\rho}(f)$
and $\rho^t \le \rho'$. By globalness
$\mu(f)=\mb{E}[f^2] \le \mu^{1-c}$, so
${\sf Stab}_{\rho}(f)\leq
\mu^{ (1-c)(1-1/t) + (1+c')/t  }
= \mu^{1 - c + (c'+c)/t} \leq \mu^{1+c}$.
\end{proof}

\subsection{Noise stability via global hypercontractivity}

As mentioned in the previous subsection,
in this subsection we will prove the noise stability
estimate Lemma \ref{corr:hypercontract_global}.
We start with some definitions required to state our
global hypercontractivity inequality.
As before, we consider a product measure
$\nu = \prod_{i=1}^n \nu _i$ on $\OO = \prod_{i=1}^n \OO_i$.
For $S \sub [n]$ we let $\nu_S$ denote the product measure
$\prod_{i \in S} \nu_i$ on $\OO_S = \prod_{i \in S} \OO_i$.

Given $f \in L^2(\OO,\nu)$ with orthogonal decomposition
$f = \sum_{S \sub [n]} f^{=S}$ and $T \sub [n]$,
the {\it Laplacian} of $f$ according to $T$
is the function $\L_T f\colon[m]^n\to\mathbb{R}$ defined by
\[(\L_T f)(x) = \sum\limits_{S\supseteq T}{f^{=S}(x)}.\]
If $T$ is a singleton $\set{i}$,
we denote the Laplacian by $\L_i$.
We also require the following alternative,
more combinatorial, definition of the Laplacian.
We let $L_\es$ be the identity operator.
For $i \in [n]$, it is easily noted that
\[ (\L_i f)(x) = f(x) - \Expect{{\bf a}_i\sim \nu_i}{f(x_1,\ldots,x_{i-1},{\bf a}_i,x_{i+1},\ldots,x_n)}.\]
Then, for $T = \set{i_1,\ldots,i_d}$ with $d\geq 2$,
one can show that $L_T$ may be defined alternatively by composition, i.e.\
$\L_T f = \L_{i_d}(\L_{i_{d-1}}(\ldots(\L_{i_1} f)\ldots))$.
It is not hard to check that this definition
does not depend on order in which the Laplacians are taken
and is equivalent to the definition via orthogonal decompositions.

In the next subsection we will prove the following
refined version of the global hypercontractive inequality
on product spaces from \cite{KLLM}.
For simplicity we only consider the version required
for our purposes, where we bound the $4$-norm
after applying noise by a function of the
$2$-norms of the Laplacians.

\begin{thm}\label{thm:correct_hypercontract}
  Let $(\Omega,\nu)$ be a finite product space.
  Then for every $f\colon \Omega\to\mathbb{R}$
  and $\rho\leq 1/160$ we have
  \[  \norm{\T_{\rho} f}_4^4
  \leq \sum\limits_{S\subset [n]} \E_{{\bf y}\sim \nu_S}
  \big[\norm{(\L_S f)_{S\ra {\bf y}}}_2^4\big].  \]
\end{thm}

Along with Theorem \ref{thm:correct_hypercontract},
the proof of Lemma \ref{corr:hypercontract_global}
also requires the following consequence of globalness
for norms of Laplacians.

\begin{claim}\label{claim:global_laplacian_bd}
Let $f\colon\Omega\to\mathbb{R}$ be $(r,\eps)$-global,
$T\sub[n]$ with $|T| \leq r$ and  $y\in [m]^T$. Then
$\norm{(\L_T f)_{T\ra y}}_2\leq 2^{\card{T}}\sqrt{\eps}$.
\end{claim}

The proof requires the following
alternative formula for Laplacians.

\begin{claim}\label{claim:laplacian_combinatorial}
For any $f\colon\Omega\to\mathbb{R}$, $T\subset[n]$ we
have $(\L_T f)(z) = \sum\limits_{S\subset T}(-1)^{\card{S}} \Expect{{\bf a}\sim\nu_S}{f(x_S = {\bf a}, x_{\ov{S}} = z_{\ov{S}})}$.
\end{claim}
\begin{proof}
We argue by induction on $|T|$.
The claim is immediate from the definition
for $\card{T} = 0, 1$. Let $\card{T} = d+1\geq 2$,
and write $T = T'\cup \set{i}$ with $|T'|=d$.
Then by definition and the induction hypothesis
  \[  (\L_T f)(z) = \L_i (\L_{T'} f)(z)
  = \L_i \sum\limits_{S\subset T'}(-1)^{\card{S}} \Expect{{\bf a}\sim\nu_{S}}{f(x_S = {\bf a}, x_{\ov{S}} = z_{\ov{S}})}. \]
By linearity and the definition of $\L_i$ we deduce
  \[  (\L_T f)(x)
  =\sum\limits_{S\subset T'}(-1)^{\card{S}} \Expect{{\bf a}\sim\nu_{S}}{f(x_S = {\bf a}, x_{\ov{S}} = z_{\ov{S}}) -
  \Expect{{\bf b}\sim\nu_i}{f(x_{S} = {\bf a}, x_i = {\bf b}, x_{\ov{S\cup\set{i}}} = z_{\ov{S\cup\set{i}}})}}.  \]
The claim follows as
$({\bf a},{\bf b})$ is distributed
according to $\nu_{S\cup\set{i}}$.
\end{proof}

\begin{proof}[Proof of Claim \ref{claim:global_laplacian_bd}]
By Claim~\ref{claim:laplacian_combinatorial} and globalness
we have
  \begin{align*}
 (\L_T f)_{T\ra y}(z)
&=\sum\limits_{S\subset T}(-1)^{\card{S}}
\Expect{{\bf a}\sim\nu_S}{
f(x_S = {\bf a}, x_{T\sm S} = y_{T\sm S}, x_{\ov{T}} = z_{\ov{T}})}\\
  &=\sum\limits_{S\subset T}(-1)^{\card{S}} \Expect{{\bf a}\sim\nu_S}{f_{T\ra ({\bf a}, y_{T\sm S})}(z_{\ov{T}})},
  \end{align*}
  and taking norm over $z$ and using the triangle inequality yields
  \[
  \norm{(\L_T f)_{T\ra y}}_2
    \leq  \sum\limits_{S\subset T}
  \Expect{{\bf a}\sim\nu_S}
  {\norm{f_{T\ra ({\bf a}, y_{T\sm S})}}_2}
 \leq 2^{\card{T}}\sqrt{\eps}.\qedhere
  \]
\end{proof}

We conclude this subsection with our estimate for
noise stability of global functions.

\begin{proof}[Proof of Lemma \ref{corr:hypercontract_global}]
Suppose $f\colon\Omega\to\mathbb{R}$ is $(r,\beta)$-global
of degree $r$. Let $\rho=1/80$.
By Theorem~\ref{thm:correct_hypercontract}
  \[  \norm{\T_{\rho/2} f}_4^4
  \leq \sum\limits_{S\subset [n]}
  \E_{{\bf y}\sim \nu_S}\big[\norm{(\L_S \T_{1/2} f)_{S\ra {\bf y}}}_2^4\big].  \]
By assumption on $f$ we only need to consider $|S| \le r$,
and for such $S$  by Claim~\ref{claim:global_laplacian_bd}
we have  $\norm{(\L_S f)_{S\ra y}}_2\leq 2^{\card{S}} \sqrt{\beta}$
for all $y\in\Omega_S$. As
$\norm{(\L_S \T_{1/2} f)_{S\ra y}}_2^2
\leq 4^{-\card{S}}\norm{(\L_S f)_{S\ra y}}_2^2$ we deduce
  \[  \norm{\T_{\rho/2} f}_4^4
  \leq\beta \sum\limits_{S\subset [n]}
  \E_{{\bf y}\sim \nu_S}\big[\norm{(\L_S \T_{1/2}f)_{S\ra {\bf y}}}_2^2\big].  \]
We estimate each summand using Parseval as
\[  \E_{{\bf y}\sim \nu_S}
\big[\norm{(\L_S \T_{1/2} f)_{S\ra {\bf y}}}_2^2\big]
=\sum\limits_{T\supseteq S}{4^{-\card{T}} \norm{f^{=T}}_2^2}
\leq \sum\limits_{T\supseteq S}{2^{-\card{T}}
\norm{f^{=T}}_2^2}, \]
  \[  \text{ so } \
  \bB^{-1} \norm{\T_{\rho/2} f}_4^4 \leq
  \sum\limits_{S\subset[n]}
  \sum\limits_{T\supseteq S}{2^{-\card{T}} \norm{f^{=T}}_2^2}
  =\sum\limits_{T\subset[n]}{\norm{f^{=T}}_2^2}
  = \norm{f}_2^2.  \qedhere  \]
\end{proof}

\subsection{Global hypercontractivity} \label{sec:pf_hyp}

We conclude this section by proving
Theorem \ref{thm:correct_hypercontract},
via our global hypercontractivity
inequality from \cite{KLLM}.
We start by stating this inequality,
for which we require some notation.
Let ${\bf Z}_1,\ldots,{\bf Z}_n$
be independent random variables,
each with mean $0$, variance $1$ and
$\Expect{}{|{\bf Z}_i|^4}\leq \sigma_i^{-2}$.
For $S\subset [n]$, we let
${\bf Z}_S = \prod\limits_{i\in S}{{\bf Z}_i}$
and $\sigma_S = \prod\limits_{i\in S}{\sigma_i}$.
We consider multilinear functions
$g({\bf Z}_1,\ldots,{\bf Z}_n)
= \sum\limits_{S\subset[n]}{a_S {\bf Z}_S}$
with all $a_S\in \mathbb{R}$.
For $S \sub [n]$ the discrete derivative of $g$ at $S$
is $\pl_S g({\bf Z}) = \frac{1}{\sigma_S}
\sum\limits_{T\supseteq S}{a_T {\bf Z}_{T\sm S}}$.
\begin{thm}[Theorem 7.1, \cite{KLLM}]\label{thm:p_biased_hyp}
   In the above set up, for $\rho \in [0,1/16]$ we have
  $\norm{\T_{\rho} g}_4^4
  \leq\sum\limits_{S\subset [n]}\sigma_{S}^2
  \norm{\pl_S g}_2^4$.
\end{thm}

We will reduce Theorem \ref{thm:correct_hypercontract}
to Theorem \ref{thm:p_biased_hyp} as follows.
Suppose $(\OO,\nu)$ is a product space
with $\Omega = [m]^n$ and $f \in L^2(\OO,\nu)$.
We will simulate $f$ via a function $g:\{0,1\}^{nm} \to {\mathbb R}$
which takes $nm$ biased random bits
$\set{{\bf z}_{i,j}}_{i\in[n],j\in[m]}$,
where the bias of ${\bf z}_{i,j}$ is $p_{i,j} = \nu_i(j)/4$.
Let $\sigma_{i,j} = \sqrt{p_{i,j}(1-p_{i,j})}$
and $\chi_{i,j}(z_{i,j}) = ({z}_{i,j} - p)/\sigma_{i,j}$.
We note that ${\chi }_{i,j}$ satisfy the conditions
in the above setup, i.e.\ $\mb{E}{\chi }_{i,j}=0$,
$\mb{E}{\chi }_{i,j}^2=1$,
$\mb{E}{\chi }_{i,j}^4 \le \sS_{i,j}^{-2}$.
For any $S\subset[n]$ and $x\in\Omega_S$
we define the corresponding character
$\chi_{S,x}: \{0,1\}^{nm} \to {\mathbb R}$ for $z = (z_{i,j}: i\in [n], j\in [m])$ by
setting ${\chi }_{S,x} ( z ) = \prod\limits_{i\in S}{\chi_{i,x_i}}(z_{i,x_i})$;
we also write
$\sigma_{S,x} = \prod\limits_{i\in S}{\sigma_{i,x_i}}$.
We then define $g\colon\power{nm}\to\mathbb{R}$ by setting
\[ g(z) = \sum_{S \sub [n]} \sum_{x \in \OO_S}
 \sigma_{S,x} \card{f^{=S}(x)} \chi_{S,x}(z) .\]

\begin{claim}
  $\norm{\T_{\rho} f}_4^4 \leq \norm{\T_{4\rho} g}_4^4$.
\end{claim}
\begin{proof}
Let ${\cal S}$ be the set of $(S_1,S_2,S_3,S_4)$
where each $S_\aA \sub [n]$ and
$|\{\aA: i \in S_\aA\}| \ne 1$ for all $i \in [n]$.
Expanding the definition of the left hand side,
we can write
\[ \norm{\T_{\rho} f}_4^4  =
  \Expect{{\bf x}\sim \nu}
  {\sum\limits_{(S_1,S_2,S_3,S_4)\in \mathcal{S}}
  {\rho^{\card{S_1} + \ldots + \card{S_4}}
  f^{=S_1}({\bf x})\cdots f^{=S_4}({\bf x})}}.\]
Also, if $S=(S_1,\dots,S_4) \in {\cal S}$
and $x \in \OO_{\bigcup S}$ then
$\Expect{}{\prod_{\aA=1}^4 \sS_{S_\aA,x_{S_{\aA}}}
\chi_{S_\aA,x_{S_{\aA}}}}
\ge \prod_{i \in \bigcup S} (p_{i,x_i}/4)
= 16^{-|\bigcup S|} \nu_{\bigcup S}(x)$,
using $\mb{E}[ (\sS_{i,j} \chi_{i,j})^q] \ge p_{i,j}/4$
when $q \in \{2,3,4\}$, so expanding the right hand side
\[ \norm{\T_{4\rho} g}_4^4
\geq  \sum\limits_{S=(S_1,S_2,S_3,S_4)\in \mathcal{S}}
\sum_{x \in \OO_S}
(4\rho)^{\card{S_1} + \ldots + \card{S_4}}
\card{f^{=S_1}(x)}\cdots \card{f^{=S_4}(x)}
16^{-|\bigcup S|} \nu_{\bigcup S}(x).\]
As $|\bigcup S| \le (\card{S_1}+\ldots+\card{S_4})/2$
the claim follows.
\end{proof}

To bound $\norm{\T_{4\rho} g}_4^4$
we apply $(4\rho )$-biased hypercontractivity
(Theorem \ref{thm:p_biased_hyp}),
which is valid if $4\rho \le 1/16$.
As $\sigma_{S,x}^2\leq \nu_S(x)$ we get
$\norm{\T_{4\rho} g}_4^4
\leq\sum\limits_{S\subset[n],x\in\Omega_{S}}
{\nu_S(x) \norm{\pl_{(S,x)} g}_2^4}$.
For any $S\subset [n]$ and $x\in\Omega_{S}$ we have
\[
\norm{\pl_{(S,x)} g}_2^2
=\frac{1}{\sigma_{S,x}^{2}}
\sum_{T \supseteq S} \sum_{y\in\Omega_{T\sm S}}
\sigma_{(T,x\circ y)}^{2} f^{=T}(x,y)^2
\leq\sum\limits_{T\supseteq S}
\Expect{{\bf y}\sim\nu_{T\sm S}}{f^{=T}(x,{\bf y})^2},
\]
as $\sigma_{S,x}^{-2} \sigma_{(T,x\circ y)}^{2} =
\sigma_{T\sm S, y}^2\leq \nu_{T\sm S}(y)$.
By Fact \ref{fact:restrictOD} and Parseval
we get $ \norm{\pl_{(S,x)} g}_2^2
\leq \norm{(\L_S f)_{S\ra x}}_2^2$, so
\[ \norm{\T_{\rho} f}_4^4 \leq
\norm{\T_{4\rho} g}_4^4
\leq \sum\limits_{S\subset[n], x\in \Omega_S}
{\nu_S(x) \norm{(\L_S f)_{S\ra x}}_2^4}
=\sum\limits_{S\subset[n]}{\Expect{{\bf x}\sim\nu_S}{\norm{(\L_S f)_{S\ra {\bf x}}}_2^4}}.\]
This proves Theorem \ref{thm:correct_hypercontract}.

\section{Moderate alphabets}\label{sec:moderate}

This section contains the proof of our main result
Theorem \ref{thm:main} in the case of moderate alphabets,
i.e.\ $m>m_0(t)$ is large, but not huge (exponential in $n$).
As discussed previously, the strategy is inspired by that
for small $m$, but we must settle for a regularity lemma
(Lemma \ref{lem:regularity_large_m}) that only provides parts
which are uncapturable, so the proof of
the junta approximation theorem becomes considerably harder.

As in the case of small $m$, we want to show that
the restrictions defining the regularity decomposition
form a $t$-intersecting family, so we need to find
cross-agreements of any fixed size between
two pieces of the decomposition.
Again we can reduce to finding cross disagreements
by taking restrictions, but this reduction is not immediate
as with the stronger pseudorandomness condition
in the first part, as uncapturability is not preserved by
arbitrary restrictions. We therefore start in the first
subsection by proving a `fairness proposition' showing that
random restrictions are unlikely to significantly reduce
the measure of a code if it is non-negligible
(for which the threshold is such that
this is only useful when $m$ is not huge).
In the second subsection we then complete
the proof of the main theorem for moderate $m$
assuming the junta approximation theorem,
and of the junta approximation theorem
assuming the existence of fixed cross-agreements
between non-negligible uncapturable codes.

The idea for finding cross disagreements
is to apply the global small set expansion theorem
from the previous section to show that
for any code of small measure we can
substantially increase its measure by a combination
of taking restrictions and applying a gluing operation,
in which we pass to a smaller alphabet
by randomly identifying symbols in each coordinate.
Here we note that any cross disagreement after gluing
must come from a cross disagreement before gluing
(this is why we will reduce to disagreements,
as we do not have any corresponding statement for
finding cross-agreements of some fixed non-zero size).
By applying Hoffman's bound to the glued codes rather
than the original codes we thus obtain a much stronger
bound on the original measures.

We develop the theory of gluings in the third subsection,
which we use for measure boosting in the fourth subsection.
We then prove the existence of fixed cross-agreements
in the final subsection.
The concept of globalness is fundamental throughout,
as it is needed for measure boosting,
and also to maintain some pseudorandomness condition
throughout the repeated restrictions needed
for measure boosting. Indeed, as uncapturability
is not preserved by arbitrary restrictions,
we need a careful combination of taking restrictions
and upgrading uncapturability to globalness.
We must also take care to remove extraneous agreements
that may be introduced by these restrictions,
which is possible as globalness implies uncapturability,
and the definition of uncapturability
is designed for this argument.

\subsection{The fairness proposition}

Here we prove the following `fairness proposition',
analogous to that proved for hypergraphs
by Keller and Lifshitz \cite{KellerLifshitz}.
The proofs are quite similar, but we include
the details for the convenience of the reader.

\begin{proposition}\label{prop:fairness}
For any $\delta>0$ and $s\in\mathbb{N}$
there is $C>0$ such that for any
${\cal F} \sub [m]^n$ with $\mu({\cal F}) \ge e^{-n/C}$,
for uniformly random ${\bf S} \in {[n] \choose s}$
and ${\bf x} \in [m]^{\bf S}$ we have
$\mb{P}[ \mu(\mathcal{F}_{{\bf S}\ra {\bf x}})
\ge (1-\delta)\mu(\mathcal{F}) ] \ge 1-\dD$.
\end{proposition}

\begin{proof}
First we consider $s=1$.
For each $i\in[n]$, let $V_i = \sett{a\in [m]}
{\mu(\mathcal{F}_{x_i\ra a}) < (1-\delta)\mu(\mathcal{F})}$.

We suppose for contradiction that the probability
of the complementary event is too large, i.e.\ that
\[ \Prob{{\bf i}\in [n], {\bf a}\in[m]}{{\bf a}\in V_{{\bf i}}}
= \frac{1}{nm}\sum\limits_{i=1}^{n}{\card{V_i}} > \dD. \]

Let $I = \sett{i\in [n]}{\card{V_i}\geq \frac{\delta}{2} m}$.
We note that
$\frac{1}{nm}\sum\limits_{i\in I}{\card{V_i}}\geq \delta/2$.
We consider uniformly random ${\bf x}\in [m]^n$ and
let ${Z} = Z({\bf x}) = |\{i: {\bf x}_i \in V_i\}|$.
Then ${Z}({\bf x}) = \sum_{i \in I} 1_{{\bf x}_i\in V_i}$
is a sum of independent indicator variables with mean
\[ \mb{E}{Z} = \sum_{i \in I} |V_i|/m \ge \dD n/2.\]
Let ${\cal F}'$ be the set of $x \in {\cal F}$
such that $|\{i: x_i \in V_i\}| \ge (1-\dD/2)\mb{E}Z$.
By the Chernoff bound, $\mu({\cal F}') \ge
\mu({\cal F}) - e^{\OO_\dD(n)} \ge (1-\dD/2)\mu({\cal F})$,
provided $C=C(\dD,s)$ is sufficiently large.

Now we estimate
$E := \mb{E}[ Z({\bf x}) 1_{{\bf x} \in {\cal F}} ]$
in two ways. By definition of $V_i$ we have
\[ E = m^{-n} \sum_{x \in {\cal F}} \sum_{i \in I} 1_{x_i \in V_i}
= m^{-n} \sum_{i \in I} \sum_{a \in V_i} |{\cal F}_{x_i \to a}|
\le m^{-1}  \sum_{i \in I} |V_i| (1-\dD)\mu({\cal F})
= (1-\dD)\mu({\cal F}) \mb{E}Z.\]
On the other hand, by definition of ${\cal F}'$ we have
\[ E \ge m^{-n} \sum_{x \in {\cal F}'} (1-\dD/2)\mb{E}Z
= (1-\dD/2) \mu({\cal F}') \mb{E}Z
\ge (1-\dD/2)^2 \mu({\cal F}) \mb{E}Z.\]
These bounds are contradictory,
so the proof for $s=1$ is complete.

For $s \ge 2$ we proceed by induction.
We suppose that the statement holds
for any $\dD'>0$ and $s'<s$ with $C=C(\dD',s')$.
We let $\dD'=\dD/2$ and $s'=s-1$
and consider uniformly random
${\bf S'} \in {[n] \choose s'}$
and ${\bf x'} \in [m]^{{\bf S'}}$.
By the induction hypothesis, which can be applied
if we choose $C(\dD,s) > C(\dD',s')$, we have
$\mb{P}[ E_1({\bf S',x'}) ] \ge 1-\dD'$.
where $E_1({\bf S',x'})$ is the event that
$\mu(\mathcal{F}_{{\bf S'}\ra {\bf x'}})
\ge (1-\dD')\mu(\mathcal{F})$.

For each $S',x'$ such that $E_1(S',x')$ holds
we consider ${\bf S} = S' \cup \{\bf i\}$
and ${\bf x} = (x',{\bf a}) \in [m]^{\bf S}$
for uniformly random ${\bf i} \in [n] \sm S'$
and ${\bf a} \in [m]$.
We have $\mu(\mathcal{F}_{S'\ra x'})
\ge (1-\dD')\mu(\mathcal{F})
> e^{-(n-s+1)/C(\dD',1)}$ for large $C(\dD,s)$.
Applying the base case to ${\cal F}_{S' \ra x'}$
we have $\mb{P}[ E_2({\bf S,x}) ] \ge 1-\dD'$.
where $E_2({\bf S,x})$ is the event that
$\mu(\mathcal{F}_{{\bf S}\ra {\bf x}})
\ge (1-\dD')\mu(\mathcal{F}_{S' \ra x'})$.
With probability at least $(1-\dD')^2 \ge 1-\dD$
both $E_1$ and $E_2$ hold, and we then have
$\mu(\mathcal{F}_{S \ra  x})
\ge (1-\dD')^2 \mu(\mathcal{F})
\ge (1-\dD)\mu(\mathcal{F})$, as required.
\end{proof}

\subsection{Proof summary}
\label{subsect:proof_summary}

In this  subsection we complete
the proof of the main theorem for moderate $m$
assuming the junta approximation theorem,
and of the junta approximation theorem
assuming the existence of fixed cross-agreements
between non-negligible uncapturable codes.
As $m$ is large, the largest ball
is a subcube of co-dimension $t$,
so we can restate our main result
for moderate $m$ as follows.

\begin{thm}\label{thm:exact_moderate_m}
For any $t\in\mathbb{N}$ there are $m_0,N\in \mathbb{N}$
such that if $m\geq m_0$, $n\geq N\log m$
and ${\cal F} \sub [m]^n$ is $(t-1)$-avoiding
then $|{\cal F}| \le m^{n-t}$, with equality
only when ${\cal F}$ is a subcube of co-dimension $t$.
\end{thm}

We will prove Theorem \ref{thm:exact_moderate_m}
assuming the following junta approximation theorem.

\begin{thm}\label{thm:approx_moderate_m}
For every $t, k\in\mathbb{N}$
there exist $C,m_0,N\in\mathbb{N}$
such that if $\mathcal{F}\sub[m]^n$ is $(t-1)$-avoiding
with $m\geq m_0$ and $n\geq N\log m$ then there
is a $t$-intersecting collection ${\cal D}$
of at most $C$ subcubes of co-dimension at most $k$
such that $\mu({\cal F}\sm \bigcup {\cal D}) \leq C m^{-k}$.
\end{thm}

\begin{proof}[Proof of Theorem \ref{thm:exact_moderate_m}]
Suppose  $\mathcal{F}\subset [m]^n$
is $(t-1)$-avoiding with $\mu(\mathcal{F}) \geq m^{-t}$.
By Theorem~\ref{thm:approx_moderate_m} there is
a $t$-intersecting collection $\mathcal{D}$ of
$O_t(1)$ subcubes of co-dimension at most $t+1$ such that
$\mu(\mathcal{F}\sm \bigcup \mathcal{D})\leq O_{t}(1) m^{-(t+1)}$.
As $\mathcal{D}$ is $t$-intersecting,
its subcubes all have co-dimension at least $t$.
Let $\mathcal{D}'$ consist of the subcubes
in $\mathcal{D}$ that have co-dimension $t$.
Then $\mu(\bigcup\mathcal{D}\sm\bigcup\mathcal{D}')
\le O_{t}(1) m^{-(t+1)}$. As
$O_t(1) m^{-(t+1)} < m^{-t} \le \mu({\cal F})$
for large $m$ we must have ${\cal D}' \ne \es$.
Thus ${\cal D}'$ consists of exactly one
subcube of co-dimension $t$, say
$\mathcal{S} = \sett{x\in[m]^n}{ x_1=1,\ldots, x_t =1}$.

Write $\mu({\cal F}_{[t] \to 1})=1-\eps$,
where $0 \le \eps = m^t \mu({\cal S} \sm {\cal F})
\le m^t \mu({\cal F} \sm {\cal S}) \le O_t(m^{-1})$.
Suppose for contradiction $\eps>0$.
We claim that $\eps > e^{-2n/m}$.
To see this, fix any $a \in \mathcal{F}\sm\mathcal{S}$
(using $\eps>0$).
Write $|\{i \in [t]:a_i=1\}|=t-1-s$ with $s \ge 0$,
fix any $S \sub [n] \sm [t]$ with $|S|=s$,
and let $R = [n] \sm ([t] \cup S)$.
For $b = a_S$ and $c \in [m]^R$ with
${\sf agr}(c,a_R)=0$, we have $(1^t,b,c) \notin {\cal F}$
(since $(1^t,b,c)$ and $a$ agree on $t-1$ coordinates),
giving $|{\cal S} \sm {\cal F}| \ge (m-1)^{n-t - s}$,
so $\eps \ge (1-1/m)^{n-t-s} > e^{-2n/m}$, as claimed.

As $\mu({\cal F} \sm {\cal S}) \ge m^{-t} \eps$,
by averaging, we can fix $1^t \ne x \in [m]^t$
with $\mu({\cal F}_{[t] \to x}) \ge m^{-t} \eps$.
Write $|\{i \in [t]:a_i=1\}|=t-1-s$ with $s \ge 0$.
Consider uniformly random ${\bf S} \in \tbinom{[n] \sm [t]}{s}$
and ${\bf y} \in [m]^{\bf S}$.
Let ${\cal G} = {\cal F}_{[t] \to 1,{\bf S} \to {\bf y}}$
and ${\cal H} = {\cal F}_{[t] \to x,{\bf S} \to {\bf y}}$.
By Markov's inequality,
$\mb{P}[\mu({\cal G}) \ge 1-2\eps] \ge 1/2$.
By Proposition \ref{prop:fairness},
$\mb{P}[\mu({\cal H}) \ge .9 m^{-t}\eps] \ge .9$.
Thus we can fix $(S,x)$ so that
$\mu({\cal G}) \ge 1-2\eps$ and
$\mu({\cal H}) \ge .9 m^{-t}\eps $.
However, ${\cal G}$ and ${\cal H}$ are cross intersecting,
so this contradicts Theorem \ref{thm:boot}.
Thus $\eps=0$, as required.
\end{proof}

We conclude this subsection by proving
Theorem \ref{thm:approx_moderate_m}
assuming the following result on cross-agreements
between uncapturable codes, the proof of which will
be the goal of the remainder of this section.

\begin{thm} \label{thm:uncapagree}
For any $s,k \in \mb{N}$ there are $r,m_0,N \in \mb{N}$
such that if $m \ge m_0$, $n \ge N\log m$
and ${\cal A}_j \sub [m]^{[n] \sm R_j}$
are $(r,m^{-k})$-uncapturable with $|R_j| \le k$ for $j=1,2$
then there are $x^j \in {\cal A}_j$ for $j=1,2$ with
$|\{i \in [n] \sm (R_1 \cup R_2): x^1_i=x^2_i\}|=s$.
\end{thm}

\begin{proof}[Proof of Theorem \ref{thm:approx_moderate_m}]
Suppose $r,m,N \gg t,k$ and $\mathcal{F}\sub[m]^n$
with $n\geq N\log m$ is $(t-1)$-avoiding.
By Lemma~\ref{lem:regularity_large_m} with $\eps=1$
there is a collection $\mathcal{D}$ of
at most $r^k$ subcubes of co-dimension at most $k$
such that ${\cal F}_{R \to \aA}$ is
$(r,m^{-k})$-uncapturable
for each $D = D_{R \to \aA} \in {\cal D}$ and
$\mu({\cal F} \sm \bigcup {\cal D})
\le 3 r^{k+1} \eps m^{-k}$.
Suppose for a contradiction that
$\mathcal{D}$ is not $t$-intersecting.
Then there are $D_{R^j \to \aA^j} \in\mathcal{D}$
for $j=1,2$ (not necessarily different) that agree
on $t-1-s$ coordinates for some $s \ge 0$.
Let ${\cal A}_1 = {\cal F}_{R^1 \to \aA^1}
\sm \bigcup_{i \in R^2 \sm R^1} D_{i \to \aA^2_i}
\sub [m]^{[n] \sm R^1}$ and define ${\cal A}_2$ similarly.
Then ${\cal A}_1$, ${\cal A}_2$
are $(r-k,m^{-k})$-uncapturable,
so by Theorem \ref{thm:uncapagree}
there are $x^j \in {\cal A}_j$
with $|\{i \in [n] \sm (R^1 \cup R^2):
x^1_i=x^2_i\}|=s$. But then
${\sf agr}((\aA^1,x^1),(\aA^2,x^2))=t-1$,
which is a contradiction.
\end{proof}

\subsection{Gluings and expansion}

In this subsection we introduce our gluing operation
and establish a small set expansion property
for global codes under random gluings.

\begin{definition} \label{def:glue}
Let $k<m\in\mathbb{N}$ and $b \ge 1$.
A $b$-balanced gluing from $[m]$ to $[k]$
is a function $\pi\colon [m]\to[k]$ such that
$\card{\pi^{-1}(i)} \le bm/k $ for all $i\in [k]$.
We let $\Pi_{m,k,b}$ denote the set of all such gluings.
If $b=1$ (which is only possible when $k \mid m$)
we may omit it from our notation.

A $b$-balanced gluing of $[m]^n$ to $[k]^n$
is a mapping $\pi\colon [m]^n \to [k]^n$ of the form
$\pi(x_1,\ldots,x_n) = (\pi_1(x_1),\ldots,\pi_n(x_n))$
with $\pi_1,\ldots,\pi_n\in \Pi_{m,k,b}$.
We let $\Pi_{m,k,b}^{\otimes n}$ denote the set
of all such gluings; we may omit
the superscript if $n$ is clear from context.
For $\mathcal{F}\subset [m]^n$
and $\pi\in\Pi_{m,k,b}^{\otimes n}$ we write
${\cal F}^\pi = \pi({\cal F}) \sub [k]^n$.
\end{definition}

\begin{example}
Consider the gluing $\pi\colon [3]^n\to[2]^n$
where for each $i \in [n]$ we have
$\pi_i(1)=\pi_i(2)=1$ and $\pi_i(3)=2$.
Let $\mathcal{F} = \sett{x\in[3]^n}
{\card{\sett{i}{x_i = 1 \lor x_i = 2}}\geq \frac{2}{3} n}$.
Then ${\cal F}$ has constant measure in $[3]^n$, but
${\cal F}^\pi$ has exponentially small measure in $[2]^n$.
\end{example}

This example indicates that we should
make a careful choice of measure in $[k]^n$
for gluing to be useful.

\begin{definition} \label{def:gluebias}
Given a measure $\nu$ on $[m]$ and $\pi:[m] \to [k]$,
we define a measure $\nu^\pi$ on $[k]$ by
$\nu^{\pi}(x) = \sum_{y\in\pi^{-1}(x)}{\nu(y)}$.
Given a product measure
$\nu = \prod_{i=1}^n \nu_i$ on $[m]^n$
and $\pi=(\pi_1,\dots,\pi_n)$
with each $\pi_i:[m] \to [k]$ we define a product measure
$\nu^\pi = \prod_{i=1}^n \nu^\pi_i$ on $[k]^n$
by $(\nu^\pi)_i = (\nu_i)^{\pi_i}
= \sum_{y\in\pi_i^{-1}(x)}{\nu(y)}$ for each $i$.
We say $\nu$ is $b$-balanced if $\nu_i(x) \le b/m$
for all $i \in [n]$ and $x \in [m]$.
\end{definition}

\begin{claim}\label{claim:glue_basic_biasing}
With notation as in Definition \ref{def:gluebias},
for any $\mathcal{F}\subset[m]^n$ we have
$\nu^\pi(\mathcal{F}^{\pi})\geq \nu(\mathcal{F})$.
\end{claim}
\begin{proof}
For any $y\in[k]^n$ we have
\[ \nu^\pi(y) = \prod_{i=1}^n \nu^\pi_i(y_i)
= \prod_{i=1}^n \sum_{x_i \in \pi_i^{-1}(y_i)} \nu_i(x_i)
= \sum_{x\in \pi^{-1}(y)} \prod_{i=1}^n \nu_i(x_i)
= \sum_{x\in \pi^{-1}(y)} \nu(x), ~~\text{ so } \]
 \[  \nu^{\pi}(\mathcal{F}^\pi)
  =\sum\limits_{y\in [k]^n}
  {\nu_{\pi}(y) 1_{y\in \mathcal{F}^{\pi}}}
  = \sum\limits_{y\in [k]^n}
  {\sum\limits_{x\in \pi^{-1}(y)}{\nu(x)1_{y\in \mathcal{F}^{\pi}}}}
  \geq \sum\limits_{y\in [k]^n}{\sum\limits_{x\in \pi^{-1}(y)}{\nu(x)1_{x\in \mathcal{F}}}} =\nu(\mathcal{F}). \qedhere \]
\end{proof}

Now we establish global small set expansion
for random balanced gluings.

\begin{lemma}\label{lem:bias_glue_stab}
With notation as in Definitions
\ref{def:glue} and \ref{def:gluebias},
there is $c>0$ such that the following holds.
Let $s,k,m\in\mathbb{N}$ be such that
$k=m/s$ and $s\geq 4$, let $\nu$ be an $s$-balanced
product measure on $[m]^n$, and suppose ${\cal F} \sub [m]^n$
is $(\log(1/\mu), \mu^{1-c})$-global
with $\mu\in (0,1/16)$.
Then $\Expect{\bm{\pi}\in \Pi_{m,k}^{\otimes n}}{\nu^{\bm{\pi}}(\mathcal{F}^{\bm{\pi}})} \geq \nu(\mathcal{F})^{1-c}$.
\end{lemma}

\begin{proof}
The plan for the proof is to show
$\mb{E}_{\bm{\pi}}[\nu^{\bm{\pi}}({\cal F}^{\bm{\pi}})]
\ge \nu({\cal F})^2/\inner{f}{Tf}$,
where $f$ is the characteristic function of ${\cal F}$
and $T = \prod_{i=1}^n T_i$ is some product Markov chain
on $[m]^n$ with each $\lL_*(T_i) \ge 1/6$.
By Theorem \ref{thm:SSE}
this will suffice to establish the lemma.

To construct $T$, we first consider for each $\pi$
the operator $T_\pi^\ua: L^2([m]^n,\nu) \to L^2([k]^n,\nu^\pi)$
defined by $T_{\pi}^{\ua} f(y)
= \cExpect{{\bf x}\sim\nu}{\pi({\bf x}) = y}{f({\bf x})}$
for any $y\in [k]^n$. Note that
$\nu({\cal F})=\nu(f)=\nu^\pi(T_\pi^\ua f)$,
as if $y \sim \nu^\pi$ and $x \sim \nu \mid \pi(x)=y$
then $x \sim \nu$. Writing $f^\pi$ for
the characteristic function of $\mathcal{F}^{\pi}$,
by Cauchy-Schwarz we can bound $\nu({\cal F})^2
= \Expect{\bm{\pi}}{\nu^{\bm{\pi}}(T_{\bm{\pi}}^{\ua} f)}^2$
as
\[ \Expect{\bm{\pi}}{\nu^{\bm{\pi}}(T_{\bm{\pi}}^{\ua} f)}^2
  = \Expect{\bm{\pi}}{\inner{T_{\bm{\pi}}^{\ua} f}
  {f^{\bm{\pi}}}_{\nu^{\bm{\pi}}}}^2
  \leq \Expect{\bm{\pi}}{\norm{T_{\bm{\pi}}^{\ua} f}_{2,\nu^{\bm{\pi}}}\norm{f^{\bm{\pi}}}_{2,\nu_{\bm{\pi}}}}^2
  \leq \Expect{\bm{\pi}}{\norm{T_{\bm{\pi}}^{\ua} f}_{2,\nu^{\bm{\pi}}}^2}\Expect{\bm{\pi}}{\norm{f^{\bm{\pi}}}_{2,\nu^{\bm{\pi}}}^2}. \]
We note that
$\Expect{\bm{\pi}}{\norm{f^{\bm{\pi}}}_{2,\nu^{\bm{\pi}}}^2}
= \Expect{\bm{\pi}}{\nu^{\bm{\pi}}(\mathcal{F}^{\bm{\pi}})}$
is the expression that we wish to bound. We write
 \[  \Expect{\bm{\pi}}
 {\norm{T_{\bm{\pi}}^{\ua} f}_{2,\nu_{\bm{\pi}}}^2}
  =\cExpect{\substack{\bm{\pi}\\ {\bf y}
  \sim\nu_{\bm{\pi}}\\ {\bf x},{\bf x'}\sim \nu}}{\bm{\pi}({\bf x})
  = \bm{\pi}({\bf x'}) = \bm{y}}{f({\bf x})f({\bf x'})}
  = \inner{f}{Tf}, \]
where $T$ is the reversible Markov chain on $[m]^n$
characterised by the property that two consecutive states
${\bf x,x'}$ of its stationary chain are distributed
as independent samples from $\nu$ conditioned on
$\bm{\pi}({\bf x}) = \bm{\pi}({\bf x'}) = \bm{y}$,
where $\pi \sim \Pi_{m,k}^{\otimes n}$
and $y \sim \nu^{\bm{\pi}}$.
We note that each of ${\bf x,x'}$
then has marginal distribution $\nu$,
which is therefore the stationary distribution.
As coordinates are independent, we can write
$T = \prod_{i=1}^n T_i$ as a product chain.
To complete the proof,
it remains to show each $\lL_*(T_i) \ge 1/6$.
By Lemma \ref{lem:gap} it suffices
to prove the following claim.

\begin{claim}
For any $i \in [n]$ and $a,b\in [m]$ we have
$p_i(a,b) := \mb{P}({\bf x}_i=a,{\bf x}'_i=b)
\geq \frac{1}{6}\nu_i(a) \nu_i(b)$.
\end{claim}

To see this, we expand out the definition to write
 \[ p_i(a,b) = \Expect{\bm{\pi}}{ \sum\limits_{j\in [k]}
    \nu_i^{\bm{\pi}}(j) 1_{\bm{\pi}(a) = \bm{\pi}(b) = j} \frac{\nu_i(a)}{\nu_i^{\bm{\pi}}(j)}
\frac{\nu_{i}(b)}{\nu_i^{\bm{\pi}}(j)}}
  = \nu_i(a)\nu_i(b)  \sum\limits_{j\in [k]}
    \Expect{\bm{\pi}}{1_{\bm{\pi}(a) = \bm{\pi}(b) = j}
    \frac{1}{\nu_i^{\bm{\pi}}(j)}}. \]
Each $\mb{P}(\bm{\pi}(a) = \bm{\pi}(b) = j)
= \frac{1}{k} \frac{s-1}{m - 1}\geq \frac{1}{2k^2}$,
so by Jensen's inequality
    \[    p_i(a,b)  \geq
    \frac{\nu_i(a)\nu_i(b)}{2k^2}
    \sum\limits_{j\in [k]}
    \cExpect{\bm{\pi}}{\substack{\bm{\pi}(a) = j,
    \\\\ \bm{\pi}(b) = j}}{\frac{1}{\nu_i^{\bm{\pi}}(j)}}
    \geq \frac{\nu_i(a)\nu_i(b)}{2k^2}
    \sum\limits_{j\in [k]}\frac{1}{
    \cExpect{\bm{\pi}}{\bm{\pi}(a)
    = \bm{\pi}(b) = j}{\nu_i^{\bm{\pi}}(j)}}.    \]
As $\bm{\pi}^{-1}(j)$ consists of $a,b$ and $s-2$ uniformly random elements from $[m]\sm\set{a,b}$ we have
    \[    \cExpect{\bm{\pi}}{\bm{\pi}(a) = \bm{\pi}(b) = j }
   {\nu_i^{\bm{\pi}}(j)}
    = \nu_i(a) + \nu_i(b) + \frac{s-2}{m-2}
    \sum\limits_{x\neq a,b}{\nu_i(x)}
    \leq \nu_i(a) + \nu_i(b) + \frac{s}{m} \leq \frac{3}{k}, \]
as each $\nu_i(y) \le s/m = 1/k$. Thus
    $  p_i(a,b)\geq \frac{1}{2k^2}\nu_i(a)\nu_i(b)
    \sum_{j\in [k]} \frac{k}{3}
    = \frac{1}{6} \nu_i(a)\nu_i(b)$.
    This completes the proof of the claim, and so of the lemma.
\end{proof}

\subsection{Boosting measure}

In this subsection we apply the small set expansion
properties of random gluings
established in the previous subsection
to prove the following result,
which shows that the measure of any small code
can be substantially increased
via restrictions and gluings.

\begin{lemma}\label{lem:sharp_thresh_gen}
For every $\eps > 0$ there is $C>0$ such that for any
$b$-balanced product measure $\nu$ on $[m]^n$
with $4 \le b \in \mb{N}$ and $m>b^{3C}$,
if ${\cal F} \sub [m]^n$
with $\nu({\cal F}) = \mu < 16^{-1/\eps}$ then there are
$\pi \in \Pi_{m,m',b}$ with $m'>m/b^{2C+1}$ and
 $\aA \in [m']^R$, where
$R \sub [n]$ with $|R| < C \log (\mu^{-1})$,  such that
$\nu^{\pi}(({\cal F}^{\pi})_{R\ra \alpha})\geq \mu^{\eps}$.
\end{lemma}

\begin{proof}
We start by applying an arbitrary $b$-balanced gluing
$\pi_0 \in \Pi_{m,m_0,b}^{\otimes n}$, where $m_0$ is
the largest power of $b$ that is at most $m$.
Clearly $\nu_0:=\nu^{\pi_0}$ is $b^2$-balanced.
We let ${\cal F}_0 = \mathcal{F}^{\pi_0} \sub [m_0]^{S_0}$,
where $S_0=[n]$. By Claim~\ref{claim:glue_basic_biasing}
we have $\mu_0 := \nu_0({\cal F}_0)\geq \mu$.

Now we apply the following iterative procedure for $i \ge 0$.
Given ${\cal F}_i \sub [m_i]^{S_i}$, where $S_0=[n]$,
with $\nu_i({\cal F}_i)=\mu_i \ge \mu$ and $\nu_i$ is a $b^2$-balanced product measure,
\begin{enumerate}
\item if $\mu_i \geq \mu^\eps$ we stop, otherwise,
\item if ${\cal F}_i$ is not
$(\log(1/\mu_i),\mu_i^{1-c})$-global according to $\nu_i$,
where $c>0$ is as in Lemma \ref{lem:bias_glue_stab},
then by definition we can choose ${\cal F}_{i+1}
= ({\cal F}_i)_{R_i \to \aA_i} \sub [m_{i+1}]^{S_{i+1}}$
with $\mu_{i+1}=\nu_{i+1}({\cal F}_{i+1}) \ge \mu_i^{1-c}$,
where $m_{i+1}=m_i$, $\nu_{i+1}=\nu_i$
and $S_{i+1} = S_i \sm R_i$ for some $R_i$
with $|R_i| \le \log(1/\mu_i)$ and $\aA_i \in [m_i]^{R_i}$,
\item otherwise, as $\mu_i<\mu^\eps \le 1/16$,
by Lemma~\ref{lem:bias_glue_stab} we can choose
${\cal F}_{i+1} = ({\cal F}_i)^{\pi_i}
\sub [m_{i+1}]^{S_{i+1}}$ with $m_{i+1}=m_i/b^2$,
$S_{i+1}=S_i$, $\pi \in \Pi_{m_i,m_{i+1}}$,
and $\mu_{i+1}=\nu_{i+1}({\cal F}_{i+1}) \ge \mu_i^{1-c}$,
where $\nu_{i+1} = \nu_i^{\pi_i}$.
\end{enumerate}
If $C > C_0(\eps,c)$ is large then this process
terminates in at most $C$ steps,
with some ${\cal F}_r \sub [m_r]^{S_r}$,
where $m_r \ge m/(b^{2C+1})$ and $S_r = [n] \sm R$,
where $R$ is the union  of all sets $R_i$
in the process, so $|R| \le C\log(1/\mu)$.
For $i \ge 0$ we let $\pi_{i \to r} \in \Pi_{m_i,m_r}$
be obtained by composing all $\pi_j$ with $i < j \le r$.
We define $\aA \in [m_r]^R$ by
$\aA_x = \pi_{i \to r}((\aA_i)_x)$ for $x \in R_i$.
We let $\pi = \pi_{0 \to r}$
and note that $\nu^\pi = \nu_r$ and
${\cal F}_r \sub ({\cal F}^\pi)_{R \to \aA}$,
so $\nu^\pi(({\cal F}^\pi)_{R \to \aA})
\ge \nu_r({\cal F}_r) \ge \mu^\eps$.
\end{proof}

\subsection{Uncapturable codes agree}

In this subsection we prove our cross-agreement
result for uncapturable codes, Theorem \ref{thm:uncapagree}. As
demonstrated in Subsection \ref{subsect:proof_summary}, this will complete the proof of our main theorem for moderate alphabets.
We start with an outline of the proof.
We are given two uncapturable codes
${\cal A}_1$ and ${\cal A}_2$
and need to find a cross-agreement
of some fixed size $s$.
Moreover, the coordinate sets
may be slightly different: we have
${\cal A}_j \sub [m]^{[n] \sm R_j}$
with $|R_j| \le k$ for $j=1,2$.

{\bf Step 1: Globalness.}
We would like to restrict to a common coordinate set,
but we cannot do so immediately,
as uncapturability is not closed under restrictions.
We therefore start by upgrading to globalness,
while avoiding unwanted agreements.
We find a global code ${\cal A}'_1$
obtained from ${\cal A}_1$ by a small restriction.
We obtain ${\cal A}'_2$ from ${\cal A}_2$
by removing any agreements with this restriction,
using uncapturability to see that
${\cal A}'_2$ is not negligible,
and find a global code ${\cal B}_2$
obtained from ${\cal A}'_2$ by a small restriction.
Then we obtain a global code
${\cal B}_1$ from ${\cal A}'_1$
by removing any agreements with this restriction.

{\bf Step 2: Fairness.}
By the fairness proposition, we find a common
restriction of size $s$
by which we obtain non-negligible global codes
${\cal C}_1,{\cal C}_2$ from ${\cal B}_1,{\cal B}_2$.
It remains to show that ${\cal C}_1,{\cal C}_2$
cannot be cross-agreeing.

{\bf Step 3: Expansion.}
We apply measure boosting to find
a gluing and restriction so that
${\cal C}_2$ becomes some ${\cal C}'_2$
with dramatically larger measure.
We obtain ${\cal C}'_1$ from ${\cal C}_1$
by removing any extra agreements
created by the gluing and restriction,
and then ${\cal C}''_1$ with non-negligible
measure by applying the gluing that were found for
$\mathcal{C}_2$.
We now find a gluing and restrictions for $\mathcal{C}''_1$
to get from it a family ${\cal C}'''_1$ with dramatically
larger measure than ${\cal C}''_1$. We then remove these
restrictions as well as apply this gluing on $\mathcal{C}'_2$
to get ${\cal C}'''_2$ whose measure not much
smaller than that of ${\cal C}'_2$.
By averaging we can apply further restrictions
without reducing measures to obtain
${\cal G}_1,{\cal G}_2$ on a common
set of coordinates.

{\bf Step 4: Hoffman bound.}
The measures of ${\cal G}_1,{\cal G}_2$
are so large that they cannot be cross-agreeing,
so we find a cross disagreement,
which corresponds to an agreement of size $s$
in the original codes.

\skipi
We proceed to the formal proof of Theorem \ref{thm:uncapagree}.
\begin{proof}[Proof of Theorem \ref{thm:uncapagree}]
We are given $(r,m^{-k})$-uncapturable
${\cal A}_j \sub [m]^{[n] \sm R_j}$
with $|R_j| \le k$ for $j=1,2$, and
we need to find $x^j \in {\cal A}_j$ with
$|\{i \in [n] \sm (R_1 \cup R_2): x^1_i=x^2_i\}|=s$,
where $n \ge N\log m$ and $r,m,N \gg s,k$.

{\bf Step 1: Globalness.}
By uncapturability $\mu({\cal A}_1) \ge m^{-k}$,
so by Lemma~\ref{lem:make_global}
with $\gamma = m^{-1/10}$ and $r/100k$ in place of $r$
we obtain $\mathcal{A}_1' = (\mathcal{A}_1)_{R_1'\ra \alpha_1'}$
that is $(r/100k,\mu(\mathcal{A}'_1)/\gamma)$-global
with $\mu({\cal A}'_1) \geq \mu({\cal A}_1)$, where
$|R'_1| \le \log_{1/\gamma}(1/\mu({\cal A}_1)) r/100k \le r/10$.
We note that ${\cal A}'_2 := {\cal A}_2 \sm
\bigcup_{i\in R_1'}D_{i\ra \alpha_1'(i)}$
is $(0.9r,m^{-k})$-uncapturable,
so $\mu({\cal A}'_2) \ge m^{-k}$.
From Lemma~\ref{lem:make_global}
we obtain ${\cal B}_2 = ({\cal A}'_2)_{R_2'\ra \alpha_2'}$
that is $(r/100k,\mu({\cal B}_2)/\gamma)$-global
with $\mu({\cal B}_2) \geq \mu({\cal A}'_2)$,
where $|R'_2| \le r/10$. In particular, ${\cal B}_2 \ne \es$,
so $R_2' \ra \alpha_2'$ has no agreement with
$R_1 \to \alpha_1$ or $R_1' \to \alpha_1'$.
We let ${\cal B}_1 = {\cal A}'_1 \sm
\bigcup_{i\in R_2'}{D_{i\ra \alpha_2'(i)}}$.
By  Claim~\ref{claim:global_to_uncap}, ${\cal A}'_1$
is $(\gamma m/4, \mu({\cal A}'_1)/2)$-uncapturable,
so $\mu(\mathcal{B}_1)\geq \half\mu(\mathcal{A}_1')$,
which implies that $\mathcal{B}_1$
is $(r/100k,2\mu(\mathcal{B}_1)/\gamma)$-global.

{\bf Step 2: Fairness.}
As $n \ge N\log m$ and $N$ is large, we have
$\mu({\cal B}_1), \mu({\cal B}_2)
\ge \half m^{-k} \ge e^{-n/C}$,
where $C=C(s,0.1)$ is as in Proposition \ref{prop:fairness}.
Consider uniformly random
${\bf S}\subset [n]\sm (R_1\cup R_1'\cup R_2\cup R_2')$
of size $s$ and ${\bf z} \in [m]^{\bf S}$.
For large $n$ the distribution of ${\bf S}$ has
total variation distance $o(1)$ from the uniform
distribution on $\tbinom{[n]\sm (R_1\cup R_1')}{s}$.
Thus by Proposition~\ref{prop:fairness} we have
$\mb{P}[ \mu((\mathcal{B}_1)_{{\bf S}\ra {\bf z}})
\geq .9 \mu(\mathcal{B}_1) ] \ge .9-o(1)$,
and similarly for ${\cal B}_2$.
Thus we can fix $S$ and $z$ so that both
${\cal C}_j = ({\cal B}_j)_{S \to z}$
have $\mu({\cal C}_j) \ge \half \mu({\cal B}_j)$,
so are $(r/100k,4\mu(\mathcal{C}_j)/\gamma)$-global.

{\bf Step 3: Expansion.}
By Lemma~\ref{lem:sharp_thresh_gen} applied
to ${\cal C}_2$ with $\eps=1/3k$ and $b=4$,
there are $\pi_2 \in \Pi_{m,m_2,4}$ with $m_2 = \OO_k(m)$,
$\aA_2'' \in [m_2]^{R_2''}$, where
$R_2'' \sub [n]\sm(R_2\cup R_2'\cup S)$ with
$|R_2''| < O_k(\log m) \ll n$, such that ${\cal C}_2' :=
({\cal C}_2^{\pi_2})_{R_2'' \to \aA_2''}$
has $\mu^{\pi_2}({\cal C}_2') \ge 1/\sqrt{m}$.
Let \[ {\cal C}_1' = {\cal C}_1 \sm
\bigcup \{ D_{i \to a}: i \in R_2'',
(\pi_2)_i(a)=(\aA_2'')_i \}.\]
By Claim~\ref{claim:global_to_uncap}, ${\cal C}_1$
is $(\gamma m/16, \mu({\cal C}_1)/2)$-uncapturable,
so $\mu({\cal C}'_1)\geq \half\mu({\cal C}_1)$.
Let ${\cal C}_1'' = ({\cal C}_1')^{\pi_2}$.
By Claim~\ref{claim:glue_basic_biasing}
we have $\mu^{\pi_2}(\mathcal{C}_1'')
\geq \mu({\cal C}_1') \ge \tfrac{1}{8}m^{-k}$.

By Lemma~\ref{lem:sharp_thresh_gen} applied
to ${\cal C}_1''$ under the $4$-balanced measure
$\mu^{\pi_2}$ with $\eps=1/3k$ and $b=4$,
there are $\pi_1 \in \Pi_{m_2,m_1,4}$ with $m_1 = \OO_k(m)$,
$\aA_1'' \in [m_1]^{R_1''}$, where
$R_1'' \sub [n]\sm(R_1\cup R_1'\cup S)$ with
$|R_1''| < O_k(\log m) \ll n$, such that ${\cal C}_1'''
:= (({\cal C}_1'')^{\pi_1})_{R_1'' \to \aA_1''}$ has
$\mu^{\pi_1 \circ \pi_2}({\cal C}_1''') \ge 1/\sqrt{m}$.
Let \[{\cal C}_2'' = {\cal C}_2' \sm
\bigcup \{ D_{i \to a}: i \in R_1'',
(\pi_1)_i(a)=(\aA_1'')_i \}.\]
Then $\mu^{\pi_2}({\cal C}_2'') \ge
\mu({\cal C}_2') - O_k(m^{-1}\log m) \ge 1/2\sqrt{m}$.
Let ${\cal C}_2''' = ({\cal C}_2'')^{\pi_1}$.
By Claim~\ref{claim:glue_basic_biasing}
we have $\mu^{\pi_1 \circ \pi_2}(\mathcal{C}_2''')
\geq \mu({\cal C}_2'') \ge 1/2\sqrt{m}$.

{\bf Step 4: Hoffman bound.}
By averaging, we can choose restrictions
${\cal G}_j \sub [m_1]^{[n]\setminus R}$ of ${\cal C}_j'''$ for $j=1,2$
where $R  = R_1\cup R_1'\cup R_1''
\cup R_2\cup R_2'\cup R_2''\cup S$ such that both
$\nu({\cal G}_j) \geq \nu({\cal C}_j''') \geq 1/2\sqrt{m}$,
where $\nu=\mu^{\pi_1 \circ \pi_2}$ is $16$-balanced.
By construction, the elements of ${\cal G}_j$ for $j=1,2$
are of the form $\pi_1\pi_2(x^j_{[n] \sm R})$
where $x^j \in {\cal A}_j$ with
$|\{i \in R \sm (R_1 \cup R_2): x^1_i=x^2_i\}|=s$.
By Lemma \ref{lem:hoffman} applied with $\lL=O_k(1/m)$
we can find a cross disagreement,
which corresponds to $x^j \in {\cal A}_j$ with
$|\{i \in [n] \sm (R_1 \cup R_2): x^1_i=x^2_i\}|=s$.
\end{proof}

\section{Huge alphabets} \label{sec:huge}

This section contains the proof of our main result
Theorem \ref{thm:main} in the case of huge alphabets,
i.e.\ when $n \le N(t)\log m$,
with $N(t)$ as in Theorem~\ref{thm:exact_moderate_m}.
As previously discussed, there are examples showing
that a proof strategy based on cross agreements
between uncapturable codes cannot work in this setting,
so instead we adopt a more combinatorial argument to
obtain expansion in measure from a `shadow' operation,
which is analogous
(but quite different in various details)
to an argument in the hypergraph setting
due to Keller and Lifshitz \cite{KellerLifshitz}.
This operation requires us to consider
more general agreement configurations
(which are anyway of interest)
even if we only want to find
pairwise agreements as in our main result.
We introduce these configurations and their
interpretation in terms of expanded hypergraphs
in the first subsection, and prove an
extremal result for configurations.
In the second subsection
we define our shadow operation and establish
two key properties , namely that
a code with a given forbidden configuration
(a) has an average shadow with much larger measure, and
(b) there is some shadow with much stronger uncapturability.
We extend these properties in the third subsection
to iterated shadows when there is some forbidden
configuration with a `kernel', i.e.\ some common
intersection of all restrictions in the configuration.
We apply this theory to prove the junta approximation
theorem in the fourth subsection. Then in the final
subsection we complete our proof via a bootstrapping
argument based on Shearer's entropy inequality.

\subsection{Hypergraphs}

When $m$ is huge, it is natural
to view a code ${\cal F} \sub [m]^n$
as an $n$-graph ($n$-uniform hypergraph)
which is $n$-partite (each edge has one vertex
in each part) with parts $V_1,\dots,V_n$,
where each $V_i = \sett{(i,a)}{a\in [m]}$,
identifying any $x \in [m]^n$
with $\{ (i,a): x_i = a \}$.
This setting is most convenient
for introducing general agreement configurations
in the following definition, as these are
a natural partite variation on the well-studied
topic of expanded hypergraphs
(see the survey \cite{mubayi2016survey}).

\begin{definition}
An $\ell$-configuration is a pair $({\cal H},{\cal P})$
where ${\cal H}$ is a multi-$\ell$-graph and ${\cal P} = (U_1,\dots,U_\ell)$
is a partition of $V(\mathcal{H})$ such that
each edge has one vertex in each part.
We identify any ${\cal H}$ with its multiset of edges
$\{e_1,\dots,e_h\}$, so its size $h=|{\cal H}|$ is its
number of edges. We often omit ${\cal P}$ from our notation.
The density of ${\cal H}$  (with respect to ${\cal P}$) is
$\mu({\cal H}) = |{\cal H}| \prod_{i \in \ell} |U_i|^{-1}$.
The kernel of ${\cal H}$ is
$K({\cal H}) = \bigcap_{i=1}^h e_i$.

The $n$-expansion ${\cal H}^+(n)$ of ${\cal H}$
is the $n$-configuration
obtained by adding disjoint sets $S_j$ of $n-\ell$
new vertices to each $e_j$, forming new parts
$U_{\ell+1},\dots,U_n$ so that each $S_j$
has one vertex in each new part.
We say that ${\cal F}_1,\dots,{\cal F}_h \sub [m]^n$
cross contain ${\cal H}$ if they do so
for ${\cal H}^+(n)$ when viewed as $n$-graphs,
i.e.\ there are $x^j \in {\cal F}_j$ for $j \in [h]$
and an injection $\Phi:[\ell] \to [n]$,
so that for any $j,j' \in [h]$ and $i \in [n]$
we have $x^j_i=x^{j'}_i$
exactly when $i=\Phi(k)$ for some $k \in [\ell]$
and $e_j \cap e_{j'} \cap U_k \ne \es$.
We say that $x_1,\dots,x_h$ realise ${\cal H}$
in ${\cal F}_1,\dots,{\cal F}_h$.
If ${\cal F}_i={\cal F}$ for all $i$
we say that ${\cal F}$ contains ${\cal H}$,
otherwise we say ${\cal F}$ is ${\cal H}$-free.
\end{definition}

\begin{example}
A code ${\cal F} \sub [m]^n$ is $(t-1)$-avoiding
if when viewed as an $n$-partite $n$-graph it does not
contain two edges $e,e'$ with $|e \cap e'|=t-1$;
equivalently, ${\cal F}$ is ${\cal H}$-free
where ${\cal H}$ is the multi-$(t-1)$-graph
with two identical edges.
\end{example}

The main result of this subsection is
the following extremal result
for cross containment at constant densities
(this suffices for our purposes, so we do not
investigate the optimal bound).

\begin{lemma}\label{corr:cros_contain_hyp}
For any $\ell, h\in\mathbb{N}$ there is $C>0$
so that if ${\cal H}$ is an $\ell$-configuration
of size $h$ and ${\cal F}_1,\dots,{\cal F}_h \sub [m]^n$
with each $\mu({\cal F}_i) > \eps$,
where $n>C \log(\eps^{-1})$ and $m>2hn/\eps$,
then ${\cal F}_1,\dots,{\cal F}_h$ cross contain ${\cal H}$.
\end{lemma}

The proof will reduce to the case when ${\cal H}$
is a matching, as in the following claim.

\begin{claim}\label{claim:cross_contain_match}
If ${\cal F}_1,\dots,{\cal F}_h \sub [m]^n$
with $m > hn/\eps$ and each $\mu({\cal F}_i) > \eps$
then ${\cal F}_1,\dots,{\cal F}_h$ cross contain a matching.
\end{claim}
\begin{proof}
We choose disjoint edges $e_i \in {\cal F}_i$ for $i \ge 1$
according to a greedy algorithm. Each choice reduces the
density of any ${\cal F}_i$ by at most $n/m < \eps/h$,
so the algorithm can be completed.
\end{proof}

\begin{proof}[Proof of Lemma \ref{corr:cros_contain_hyp}]
Write ${\cal H} = \{e_1,\dots,e_h\}$ and let
$(U_1,\dots,U_\ell)$ be the fixed partition of ${\cal H}$.
We identify each ${\cal F}_i$ with an $n$-partite $n$-graph
with parts $V_i = \{ (i,a): a \in [m] \}$.
We consider uniformly random injections $\Phi:[\ell]\to[n]$
and $\phi_j\colon U_j\to V_{\Phi(j)}$ for each $j\in [\ell]$.
Each edge $e_i$ then defines a restriction
${\cal G}_i = ({\cal F}_i)_{\Phi([\ell]) \to \aA^i}$, where
$\aA^i_{\Phi(j)} = \phi_j(e_i \cap U_j)$ for $j \in [\ell]$.

We let $C=C(\ell,1/2h)$ be as in
Proposition~\ref{prop:fairness},
which is then applicable as
$\mu({\cal F}_i)\geq \eps \geq e^{-n/C}$, giving
$\mb{P}[ \mu({\cal G}_i) \ge (1-1/2h) \mu({\cal F}_i) ]
\ge 1 - 1/2h$. By a union bound we can fix $\Phi$
and $\phi_1,\dots,\phi_j$ so that all $\mu({\cal G}_i)>\eps/2$.
Then ${\cal G}_1,\dots,{\cal G}_h$ cross contain a matching,
so ${\cal F}_1,\dots,{\cal F}_h$ cross contain ${\cal H}$.
\end{proof}

\subsection{Shadows}

In this subsection we define our shadow (projection) operation
and establish its two key properties mentioned above
(boosting measure and strengthening uncapturability).

\begin{definition}
For ${\cal F}\sub [m]^n$ and $i\in[n]$,
the $i$-shadow of ${\cal F}$ is $\pl_i({\cal F})
= \bigcup_{a \in [m]} \pl_{i\ra a}({\cal F})$, where
$\pl_{i\ra a}({\cal F}) = {\cal F}_{i\ra a}\sub [m]^{n-1}$.
For $I \sub [n]$ we let $\pl_I$ be the composition
(in any order) of $(\pl_i: i \in I)$.
\end{definition}

The next lemma, analogous to a
lemma for hypergraphs in  \cite{kostochka2015turan},
shows that shadows have significantly larger measure
on average if we forbid a configuration with the
following `flatness' property.

\begin{definition}
The centre of a configuration is the set
of vertices contained in more than one edge.

We say that a configuration is flat if each part
has at most one vertex in the centre.
\end{definition}

\begin{lemma}\label{lem:shadow_lower}
Suppose ${\cal H}$ is a flat $\ell$-configuration of size $h$
and ${\cal F} \sub [m]^n$ is ${\cal H}$-free, with $n \ge h\ell$.
Then $|{\cal F}| \le h \sum_{i=1}^n |\pl_i({\cal F})|$.
\end{lemma}

\begin{proof}
Let ${\cal F}'$ be obtained from ${\cal F}$ by the following
iterative deletion procedure starting from ${\cal F}'={\cal F}$:
if there is any $i \in [n]$ and $y \in \pl_i({\cal F}')$
such that at most $h$ choices of $x \in {\cal F}'$
with $x_{[n] \sm i} = y$ then we delete all such $x$.
Any $y \in \pl_i({\cal F})$ is considered at most once
in this procedure before it is removed from the shadow.
Thus the number of deleted sets is at most
$h \sum_{i=1}^n |\pl_i({\cal F})|$,
so it suffices to show ${\cal F}'=\es$.

Suppose for contradiction ${\cal F}' \ne \es$.
We will show that ${\cal F}$ contains ${\cal H}$.
We write ${\cal H} = \{e_1,\dots,e_h\}$,
denote the parts of ${\cal H}$ by $U_1,\dots,U_\ell$,
and fix $u_j \in U_j$ for each $j \in [\ell]$
so that each vertex of $U_j$ other than $u_j$
is contained in at most one edge.
Fix any $x \in {\cal F}'$. We will construct
$x^1,\dots,x^h \in {\cal F}'$ realising ${\cal H}$
according to injections $\phi_j:U_j \to [m]$
so that $x^i_j = \phi_j(e_i \cap U_j)$
and $x_j = \phi_j(u_j)$
for all $i \in [h]$ and $j \in [\ell]$.
As ${\cal H}$ is flat this can be achieved greedily.
Indeed, to construct $x^i$ we can start from $x^i=x$
and one by one for each $j$ such that
$e_i \cap U_j \ne \{u_j\}$ replace $x^i_j$
by some new value not yet used in coordinate $j$,
which is possible as there are at least $h+1$ choices
for $x^i_j$ for any given $x^i_{[n] \sm \{j\}}$.
However, ${\cal F}$ is ${\cal H}$-free,
so we have the required contradiction.
\end{proof}

We conclude this subsection by showing under the same
conditions as the previous lemma, that if a code
is uncapturable, then it has some shadow
is significantly more uncapturable.
The key point is that the uncaptured measure
is increased by a factor $\OO(m/n)$,
albeit at the expense of only considering
restrictions that are $n$ times smaller.

\begin{lemma}\label{lem:shadow_uncap}
Suppose ${\cal H}$ is a flat $\ell$-configuration of size $h$
and ${\cal F} \sub [m]^n$ is ${\cal H}$-free, with $n \ge h\ell$.
If ${\cal F}$ is $(r,\eps)$-uncapturable then $\pl_i({\cal F})$
is $(r/n, \eps m/nh)$-uncapturable for some $i \in [n]$.
\end{lemma}

\begin{proof}
We suppose each $\pl_i(\mathcal{F})$ is $(r/n,\dD)$-capturable
and show that $\dD\geq \eps m/nh$. By definition,
for each $i \in [n]$ there is a collection ${\cal D}_i$
of at most $r/n$ dictators in $[m]^{[n] \sm \{i\}}$ such that
$\mu(\pl_i({\cal F})\sm\bigcup {\cal D}_i)\leq \dD$.
We let ${\cal D} = \bigcup_{i=1}^n {\cal D}_i$
where now we consider each dictator in $[m]^n$.
Then $\mu({\cal F} \sm \bigcup {\cal D}) \ge \eps$
by uncapturability. Applying Lemma~\ref{lem:shadow_lower}
to ${\cal F} \sm \bigcup {\cal D}$, noting that each
$\pl_i({\cal F} \sm \bigcup {\cal D})
\sub \pl_i({\cal F}) \sm \bigcup {\cal D}_i$, we have
\[ |{\cal F} \sm \bigcup {\cal D}| \leq
s\sum_{i=1}^n |\pl_i({\cal F}) \sm \bigcup {\cal D}_i|
\leq hn \cdot \dD m^{n-1},\]
so $\eps \leq \mu({\cal F} \sm \bigcup {\cal D})
\leq \dD hn/m$, i.e.\ $\dD\geq \eps m/nh$.
\end{proof}

\subsection{Kernels and iterated shadows}

In this subsection we consider configurations
with a non-trivial kernel (intersection of all edges),
for which we show that they remain free of some
configuration under iterated shadows
(as many as the size of the kernel),
so the results of the previous subsection on
single shadows become correspondingly
stronger in this setting.
First we introduce some convenient notation.

\begin{definition}
Given an $\ell$-configuration ${\cal H}$,
we write ${\cal H} \oplus [t]$ for
the $(\ell+t)$-configuration with $t$ additional
parts of size $1$ where each edge of ${\cal H}$
is extended to also include the $t$ new vertices.

Given an $\ell$-configuration ${\cal H}$ on $v$ vertices,
we let ${\sf flat}({\cal H})$ be the $v$-configuration
obtained by taking a copy of ${\cal H}$ with one vertex
in each part and adding to each edge $e$
disjoint sets $S_e$ of $v-\ell$ new vertices.
\end{definition}

\begin{remark} \label{rem+flat} $ $
\begin{enumerate}
\item
Any (flat) configuration
with a kernel of size $t$ may be expressed
as ${\cal H} \oplus [t]$ for some (flat)
configuration ${\cal H}$ with no kernel.
\item
If ${\cal H}$ is flat and contained in ${\cal H}'$
then ${\cal H}$ is contained in ${\sf flat}({\cal H}')$.
\end{enumerate}
\end{remark}

\begin{lemma}\label{lem:kernel_down}
For any (flat) configuration ${\cal H}$
there exists a (flat) configuration ${\cal H}'$
such that for all $t\in\mathbb{N}$
there exist $m_0,n_0 \in \mb{N}$
such that if ${\cal F} \sub [m]^n$
with $n \ge n_0$, $m \ge m_0$ is ${\cal H}\oplus[t]$-free
then $\pl_i({\cal F})$ is ${\cal H}'\oplus[t-1]$-free
for all $i \in [n]$.
\end{lemma}

\begin{proof}
Consider any $\ell$-configuration ${\cal H}$ of size $h$.
We let $C=C(\ell+t,h)$
be as in Lemma \ref{corr:cros_contain_hyp},
$\eps=1/2h$, $n_1 = 2C \log(\eps^{-1})$, $m_1 = 3hn/\eps$,
and then prove the statement for ${\cal H}' = [m_1]^{n_1}$,
the complete $n_1$-partite $n_1$-graph with parts of size $m_1$.
We note that if $\mathcal{H}$ was flat initially,
then one can take ${\sf flat}(\mathcal{H}')$ instead of $\mathcal{H}'$
to preserve flatness, and the correctness follows from the analysis below
and by Remark~\ref{rem+flat}.

We show the contrapositive statement,
i.e.\ that if $\pl_{i^*}({\cal F})$ contains
${\cal H}'\oplus[t-1]$ for some $i^* \in [n_1]$
then ${\cal F}$ contains ${\cal H}\oplus[t]$.
The version for flat configurations will then
follow by Remark \ref{rem+flat}.2.

By relabelling we can assume that we have
$X = \{ x(y) : y \in [m_1]^{n_1} \} \sub {\cal F}$
where each $x(y)_{[n_1]}=y$,
there is some $T \in \tbinom{[n] \sm [n_1]}{t-1}$
such that $x(y)_i = 1$ for all $i \in T$,
and $x(y)_i \ne x(y')_i$ whenever $y \ne y'$ and
$i \notin [n] \sm (T \cup [n_1] \cup \{i^*\})$.
We $m$-colour $[m_1]^{n_1}$
as ${\cal C}_1,\dots,{\cal C}_m$, where
each ${\cal C}_j = \{ y: x(y)_{i^*}=j \}$.

Note that if some $\mu({\cal C}_j) \geq \eps$
then applying Lemma~\ref{corr:cros_contain_hyp}
with ${\cal F}_1 = \dots = {\cal F}_h = {\cal C}_j$
we find a copy of ${\cal H}$ in ${\cal C}_j$.
The corresponding $x(y) \in {\cal F}$ for each $y$
in this copy agree outside $[n_1]$ in coordinates
$T \cup \{i^*\}$ and no others, so we obtain
a copy of ${\cal H}\oplus [t]$ in ${\cal F}$.

Thus we may assume that each $\mu({\cal C}_j) < \eps$.
By repeated merging we can form `meta-colours'
${\cal C}'_1,\dots,{\cal C}_{m'}$,
each of which is a union of some of the ${\cal C}_j$'s,
such that each $\mu({\cal C}'_j) \in (\eps,2\eps)$,
so $m' \ge 1/2\eps = h$.
By Lemma~\ref{corr:cros_contain_hyp},
${\cal C}'_1,\dots,{\cal C}_{h}$
cross contain $\mathcal{H}\oplus[1]$.
The corresponding $x(y) \in {\cal F}$ for each $y$
in this copy agree outside $[n_1]$ in coordinates
$[T]$ and no others, so again we obtain
a copy of ${\cal H}\oplus [t]$ in ${\cal F}$.
\end{proof}

The following corollary is immediate
by iterating Lemma~\ref{lem:kernel_down}.

\begin{corollary}\label{corr:kernel_down}
For any (flat) configuration ${\cal H}$ and $t \in \mb{N}$
there exist $m_0,n_0 \in \mb{N}$ and
a (flat) configuration ${\cal H}'$
such that if ${\cal F} \sub [m]^n$
with $n \ge n_0$, $m \ge m_0$ is ${\cal H}\oplus[t]$-free
then $\pl_I({\cal F})$ is ${\cal H}'$-free
for all $I \in \tbinom{[n]}{t}$.
\end{corollary}

We also have the following corollary giving
improved estimates on measures and uncapturability
of iterated shadows.

\begin{corollary}\label{corr:find_suitable_shadow}
For any flat configuration ${\cal H}$ and $t \in \mb{N}$
there exist $m_0,n_0 \in \mb{N}$ and $C>0$
such that for any ${\cal H}\oplus[t-1]$-free
${\cal F} \sub [m]^n$ with $n \ge n_0$, $m \ge m_0$,
\begin{enumerate}
\item $|{\cal F}| \leq
C \sum_{I \in \tbinom{[n]}{t}} |\pl_I({\cal F})| $,
\item if ${\cal F}$ is $(r,\eps)$-uncapturable then
$\pl_I({\cal F})$ is $(r/n^t, (m/n)^t \eps/C)$-uncapturable
for some $I \in \tbinom{[n]}{t}$.
\end{enumerate}
\end{corollary}

\begin{proof}
We argue by induction on $t$. The base case $t=1$
is given by Lemmas \ref{lem:shadow_lower}
and \ref{lem:shadow_uncap}. Now suppose $t \ge 2$.
By Lemma~\ref{lem:kernel_down} there is a configuration
${\cal H}'$ depending only on ${\cal H}$ such that
each $\pl_i({\cal F})$ is ${\cal H}\oplus[t-2]$-free.
The induction hypothesis of (1) gives $C'=C({\cal H}',t-1)$
such that each $|\pl_i({\cal F})| \le C' \sum
\{ |\pl_{I \cup \{i\}}({\cal F})| : I \in \tbinom{[n]}{t-1} \}$,
which proves (1). For (2), if ${\cal F}$
is $(r,\eps)$-uncapturable then by Lemma~\ref{lem:shadow_uncap} the family $\pl_i({\cal F})$
is $(r/n, \eps m/n|{\cal H}|)$-uncapturable for some $i \in [n]$.
By the induction hypothesis of (2),
$\pl_{I \cup \{i\}}({\cal F})$ is $(r/n^t, (m/n)^t \eps/C)$-uncapturable for some $I \in \tbinom{[n] \sm \{i\}}{t-1}$,
which proves (2).
\end{proof}

\subsection{Junta approximation}

In this subsection we prove
Theorem~\ref{thm:junta_approx}
in the case that $m$ is huge
(at least exponential in $n$).

\begin{thm}\label{thm:approx_huge_m}
For any $t,N\in\mb{N}$ there are $K,n_0\in\mb{N}$
such that if $\mathcal{F}\sub[m]^n$ is $(t-1)$-avoiding
with $n\geq n_0$ and $m\geq 2^{n/N}$ then
there exists a subcube $D$ of co-dimension $t$
such that $\mu(\mathcal{F}\sm D)\leq 2^{-n/K} m^{-t}$.
\end{thm}

\begin{proof}
We apply Lemma~\ref{lem:regularity_large_m}
with $r=n^t$, $k=t$ and $\eps = 2^{-2n/K} \ge 1/m$,
where $K,n_0 \gg t,N$, obtaining
a collection $\mathcal{D}$ of
at most $r^k = n^{t^2}$ subcubes of co-dimension
at most $t$ such that ${\cal F}_{R \to \aA}$ is
$(r,\eps \mu(D)^{-1} m^{-t})$-uncapturable
for each $D = D_{R \to \aA} \in {\cal D}$ and
$\mu({\cal F} \sm \bigcup {\cal D})
\le n^{2t^2} \eps m^{-t}$.
We let ${\cal D}_d$ be the set of
subcubes in ${\cal D}$ of co-dimension $d$.
To prove the theorem, it suffices to show that
(a) ${\cal D}_d = \es$ for $d<t$, and
(b) $|{\cal D}_t| \le 1$.

To see (a), suppose for a contradiction that
$D_{R \to \alpha} \in \mathcal{D}_{t-1-s}$ with $s \ge 0$.
As $\mathcal{F}$ is $(t-1)$-avoiding,
${\cal F}_{R \to \aA}$ is $s$-avoiding,
i.e.\ is ${\cal H} \oplus [s]$-free,
where ${\cal H}$ is the flat $0$-configuration
consisting of two copies of the empty set.
By Corollary \ref{corr:kernel_down}, there is some
flat configuration ${\cal H}'$ such that
$\pl_I({\cal F}_{R \to \aA})$ is ${\cal H}'$-free
for any $I \in \tbinom{[n] \sm R}{s}$.
By Corollary \ref{corr:find_suitable_shadow}, as
${\cal F}_{R \to \aA}$ is $(n^t,\eps m^{-s-1})$-uncapturable,
there is some $I \in \tbinom{[n] \sm R}{s}$ such that
${\cal G}:=\pl_I({\cal F}_{R \to \aA})$
is $(n^{t-s},\eps/O_t(mn^s))$-uncapturable.

To obtain the required contradiction
we will show that ${\cal G}$ contains ${\cal H}'$.
Write $|{\cal H}'|=h'=O_t(1)$. Let ${\cal J}$ be
the set of all dictators $D_{i \to a}$ such that
$\mu({\cal G}_{i \to a}) > \eps^2/n^2$.
We claim that $|{\cal J}|<h'$. To see this,
suppose on the contrary that ${\cal J}$ contains
$D_{i^1 \to a^1},\dots,D_{i^{h'} \to a^{h'}}$.
Let $I' = \{i^1,\dots,i^{h'}\}$.
By averaging, we can fix
$x^j \in [m]^{I'}$ for $j \in [h']$
such that $x^j_{i^j}=a^j$,
$x^j_{i^{j'}} \ne a^{j'}$ for all $j' \ne j$,
so that ${\cal F}_j = {\cal G}_{I' \to x^j}$
has $\mu({\cal F}_j) > \eps^2/2n^2$.
However, then ${\cal H}'$ is cross contained in
${\cal F}_1,\dots,{\cal F}_{h'}$
by Lemma \ref{corr:cros_contain_hyp},
applied with $\eps^2/2n^2$ in place of $\eps$
(using $n > C\log(2n^2/\eps^2)$
and $m > 2hn \cdot 2n^2/\eps^2$ for large $K$).
Thus $|{\cal J}|<h'$, as claimed.

By uncapturability of ${\cal G}$, writing
${\cal G}' = {\cal G} \sm \bigcup {\cal J}$ we have
$\mu({\cal G}') \ge \eps/O_t(mn^s) > \eps^2/m$.
By Lemma \ref{lem:shadow_lower} we can fix
$i^* \in [n] \sm (R \cup I)$ with
$|{\cal G}'|/h'n \le |\pl_{i^*}({\cal G}')|$.
We fix any partition $({\cal F}'_a: a \in [m])$
of $\pl_{i^*}({\cal G}')$ such that
each ${\cal F}'_a \sub \pl_{i^* \to a}({\cal G}')$.
Then $\sum_a \mu({\cal F}'_a) = \mu(\pl_{i^*} {\cal G}')
\ge \mu({\cal G}') m/h'n > \eps^2 /h'n$.
Also, by definition of ${\cal J}$ each
$\mu({\cal F}'_a) < \eps^2/n^2$.
By repeated merging we can form a partition
${\cal P}$ of $[m]$ such that each $S \in {\cal P}$
has $\sum_{a \in S} \mu({\cal F}'_a)
\in (\eps^2/n^2,2\eps^2/n^2)$.
Then $|{\cal P}| \ge h'$, so we can choose
$S_1,\dots,S_{h'}$ in ${\cal P}$ and
apply Lemma \ref{corr:cros_contain_hyp}
to see that ${\cal F}_1,\dots,{\cal F}_{h'}$
cross contain ${\cal H}'$, where each
${\cal F}_i = \bigcup_{a \in S_i} {\cal F}'_a$.
This completes the proof of (a).

To see (b), suppose for contradiction that we have
distinct subcubes $D_{R_j \to \aA_j}$ for $j=1,2$
of co-dimension $t$. Suppose they agree on
$t-1-s$ coordinates, for some $s \ge 0$. Consider
$\mathcal{G}_1 = \mathcal{F}_{R_1\ra \alpha_1}
\sm \mathcal{F}_{R_2\ra \alpha_2}$ and
$\mathcal{G}_2 = \mathcal{F}_{R_2\ra \alpha_2}
\sm \mathcal{F}_{R_1\ra \alpha_1}$. By uncapturability,
both $\mu({\cal G}_j)\geq \eps > e^{-n/C}$, where
$C=C(s,0.1)$ is as in Proposition \ref{prop:fairness},
as $K$ is large. Consider uniformly random
${\bf S}\sim {[n]\sm (R_1\cup R_2) \choose s}$
and ${\bf x}\in[m]^{\bf S}$.

By Proposition~\ref{prop:fairness} both
$\mb{P}[ \mu((\mathcal{G}_j)_{{\bf S}\ra {\bf x}})
\geq .9 \mu(\mathcal{G}_j) ] \ge .9-o(1)$,
as ${\bf S}$ is total variation distance $o(1)$
from uniform on $\tbinom{[n] \sm R_j}{s}$.
Thus we can fix $S,x$ so that both
${\cal G}'_j = ({\cal G}_j)_{S \to x}$
have $\mu({\cal G}'_j) > .9 \eps$.
By averaging, we can fix some
${\cal G}''_1 = ({\cal G}'_1)_{R_2 \sm R_1 \to a^1}$
with $\mu({\cal G}''_1) \ge \mu({\cal G}'_1) > .9\eps$,
and similarly some ${\cal G}''_2$.
Then ${\cal G}''_1$, ${\cal G}''_2$ are defined by
restrictions of ${\cal F}$ to $R_1 \cup R_2 \cup S$
with agreement exactly $t-1$, so must be cross intersecting.
However, $m\geq 2^{n/N} \gg \eps^{-1}$ for large $K$,
so this contradicts Lemma~\ref{lem:hoffman1}.
\end{proof}

\subsection{Bootstrapping}

We conclude this part with the bootstrapping step that
completes the proof of our main theorem for huge alphabets,
which we restate as follows.

\begin{thm}\label{thm:exact_huge_m}
For any $t,N\in\mathbb{N}$ there is $n_0 \in \mathbb{N}$
such that if $n\geq n_0$, $m\geq 2^{n/N}$
and ${\cal F} \sub [m]^n$ is $(t-1)$-avoiding
then $|{\cal F}| \le m^{n-t}$, with equality
only when ${\cal F}$ is a subcube of co-dimension $t$.
\end{thm}

We require Shearer's entropy lemma  \cite{Shearer},
as applied to the projection operators
$\Pi_S = \pl_{[n] \sm S}$ on $[m]^n$.
\begin{lemma}\label{lem:shearer}
For ${\cal F} \sub [m]^n$ and $k \in [n]$ we have
$\card{\mathcal{F}}^{{n-1\choose k-1}}\leq
\prod_{\card{S} = k}{\card{\Pi_S(\mathcal{F})}}$.
\end{lemma}

\begin{proof}[Proof of Theorem \ref{thm:exact_huge_m}]
Suppose ${\cal F} \sub [m]^n$ is $(t-1)$-avoiding
with $|{\cal F}| \geq m^{n-t}$.
By Theorem~\ref{thm:approx_huge_m} there is a subcube $D$
of co-dimension $t$ such that ${\cal G}:={\cal F} \sm D$
has $\eps := \mu({\cal G}) m^t \leq 2^{-n/K}$,
for some $K=K(N,t)$.
We may assume $D=\sett{x\in[m]^n}{x_1=\ldots=x_t=1}$.
Suppose for contradiction that $\eps>0$.
For each $T \subn [t]$ let ${\cal G}_T$
be the set of all $x_{[n] \sm [t]}$ where
$x \in {\cal G}$ with $T = \{i \in [t]: x_i=1\}$.
We have $\eps = \mu({\cal G}) m^t
\leq \sum_T m^{t-|T|} \mu({\cal G}_T)$,
so for a contradiction it suffices to show that
each $\mu({\cal G}_T) < m^{|T|-t} \eps/n$.

As ${\cal F}$ is $(t-1)$-avoiding,
each ${\cal G}_T$ is $(t-1-|T|)$-avoiding.
In particular, if $|T|=t-1$ then ${\cal G}_T$
is intersecting, so by Lemma \ref{lem:hoffman1}
we have the required bound
$\mu({\cal G}_T) < 2\eps/m^2 < m^{-1} \eps/n$.

Now fix any $T \sub [t]$
where $|T|=t-1-d$ with $d \ge 1$.
As ${\cal G}_T$ is $d$-avoiding, it is free
of a configuration with kernel size $d$,
so by Corollary \ref{corr:find_suitable_shadow}
we have $|{\cal G}_T| \leq O_t(1)
\sum_{I \in \tbinom{[n]}{d+1}} |\pl_I({\cal G}_T)|$.
Fix any $I \in \tbinom{[n]}{d+1}$.
To complete the proof
it suffices to establish the following claim,
as this will imply $\mu({\cal G}_T)
< O_t(1) (n/m)^{d+1} \eps^2 <  m^{|T|-t} \eps/n$.

\begin{claim}
$\mu(\pl_I {\cal G}_T) < \eps^2$.
\end{claim}

We will prove this claim using Shearer's inequality
with $k=d$, so we now analyse the projections
$\Pi_S \pl_I {\cal G}_T = \Pi_S {\cal G}_T$
for $S \in \tbinom{[n] \sm I}{d}$.
For such $S$ with $S \cap [t] \ne \es$
we use the trivial bound
$|\Pi_S {\cal G}_T| \le m^d$.
Now fix $S$ with $S \cap [t] = \es$.
We will show that $\mu(\Pi_S {\cal G}_T) < 2\eps$.

To see this, we first show for any
$x \in \Pi_S {\cal G}_T$ that
${\cal F}'_x := {\cal F}_{[t] \to {\bf 1}, S \to x}$
has $\mu({\cal F}'_x) \le n/m$. Suppose not,
and fix $y \in {\cal G}$ with  $\pi_S(y)=x$
and $T = \{i \in [t]: y_i=1\}$.  By a union bound
$\mu({\cal F}'_x \sm \bigcup_{i \in [n] \sm ([t] \cup S)}
D_{i \to y_i}) > 0$, so we can choose $z \in {\cal F}'_x$
that disagrees with $y$ on $[n] \sm ([t] \cup S)$.
However, extending $z$ with $x \in [m]^S$
and ${\bf 1} \in [m]^t$ gives $z^+ \in {\cal F}$
with ${\sf agr}(z^+,y)=t-1$, which is impossible,
so indeed $\mu({\cal F}'_x) \le n/m$.

As $|{\cal F}| \ge \card{D_{[t] \to {\bf 1}}}$ this implies
$|{\cal G}| \ge |D_{[t] \to {\bf 1}} \sm {\cal F}|
\ge |\Pi_S {\cal G}_T| \cdot (1-n/m) m^{n-t-d}$,
so $\eps = \mu({\cal G}) m^t \ge
(1-n/m) \mu(\Pi_S {\cal G}_T)$,
giving $\mu(\Pi_S {\cal G}_T) < 2\eps$.
Finally, writing $n'=|[n] \sm I|=n-(d+1)$,
Lemma \ref{lem:shearer} gives
\[ \card{\pl_I {\cal G}_T}^{n'-1\choose d-1}\leq
\prod_S \card{\Pi_S \pl_I {\cal G}_T)} \leq
(m^d)^{{n' \choose d}-{n'-t\choose d}}
(2\eps m^d)^{n'-t\choose d}
=(2\eps)^{n'-t\choose d} (m^d)^{{n'\choose d}},\]
so $|\pl_I {\cal G}_T|
\le (2\eps)^{n'/2d}  (m^d)^{n'/d} < \eps^2 m^{n'}$.
This completes the proof of the claim,
and so of the theorem.
\end{proof}

\section{Configurations}

In this section  we briefly consider
generalisations to excluded configurations
(as in the previous section).
Our aim is not a systematic study,
but just to illustrate the further potential
applications of our methods.
We start with a general junta approximation
result for small alphabets.

\begin{thm}\label{thm:junta+small}
For every $\eta>0$, configuration ${\cal H}$
and $m\in\mathbb{N}$ with $m > |{\cal H}|$
there are $J,n_0 \in \mathbb{N}$ such that
if $\mathcal{F} \sub [m]^n$ is ${\cal H}$-free
with $n \ge n_0$ then there is an ${\cal H}$-free $J$-junta
$\mathcal{J}\subset[m]^n$ such that
$\mu({\cal F} \sm {\cal J}) \le \eta$.
\end{thm}

The proof requires the following
generalisation of Theorem \ref{thm:MOS}.

\begin{thm}\label{thm:MOS+}
For every $h,m \in \mb{N}$ with $m>h$ and $\mu>0$
there are $\eps,c>0$ and $r \in \mb{N}$ such that
if ${\cal F}_1,\dots,{\cal F}_h\sub [m]^n$
are $(r,\eps)$-pseudorandom with
each $\mu({\cal F}_j) > \mu$
and $(x_1,\dots,x_h) \in ([m]^n)^h$
is uniformly random subject to
${\sf agr}(x_j,x_{j'}) = 0$ whenever $j \ne j'$ then
$\mb{P}(x_1 \in \mc{F}_1, \dots,x_h \in \mc{F}_h)>c$.
\end{thm}

The proof of Theorem \ref{thm:MOS+}
is the same as that of Theorem \ref{thm:MOS},
except that the absolute spectral gap condition
must be replaced by a more general condition on
`correlated spaces', which specialises to our
situation as follows. Given $f,g:[m]^h \to \mb{R}$
with $\mb{E}f=\mb{E}g=0$ and $\mb{E}f^2=\mb{E}g^2=1$,
such that $f$ depends only on the first coordinate
and $g$ does not depend on the first coordinate,
and  ${\bf a} \in [m]^h$ with distinct coordinates
chosen uniformly at random, we need to show that
$\mb{E} f({\bf a})g({\bf a}) < 1$. By considering
the equality conditions for Cauchy-Schwarz,
it is not hard to see that this holds when $m>h$.

\begin{proof}[Proof of Theorem \ref{thm:junta+small}]
The proof is the same as that of
Theorem \ref{thm:junta_approx_small_m},
except that instead of showing that
${\cal J}$ is $t$-intersecting we need to show
that ${\cal J}$ is ${\cal H}$-free.
To see this, we suppose for a contradiction
that ${\cal J}$ contains ${\cal H}$ and
show that ${\cal F}$ contains ${\cal H}$.
We suppose ${\cal H} = \{e_1,\dots,e_h\}$
is an $\ell$-configuration with
parts $(U_1,\dots,U_\ell)$ realised by
$x^1,\dots,x^h \in {\cal J}$.
By relabelling we can assume that ${\cal H}$
is realised on coordinate set $[\ell]$,
i.e.\ for any $j,j' \in [h]$ and $i \in [n]$
we have $x^j_i=x^{j'}_i$ exactly when $i \in [\ell]$
and $e_j \cap e_{j'} \cap U_i \ne \es$.
For $j \in [h]$ we let ${\cal G}_j =
{\cal F}_{J \cup [\ell] \to x^j_{J \cup [\ell]}}$.
Then each ${\cal G}_j$ is $(r-\ell,\eps)$-pseudorandom
with density at least $\eta/3$,
so by Theorem \ref{thm:MOS+}
we find $w_j \in {\cal G}_j$ for $j \in [h]$
with ${\sf agr}(w_j,w_{j'})=0$ whenever $j \ne j'$.
However, $((x^j_{J \cup [\ell]},w_j): j \in [h])$
realise ${\cal H}$ in ${\cal F}$, contradiction.
\end{proof}

Next we will turn to large alphabets,
for which we require the following
generalised Hoffman bound.

\begin{lemma}\label{lem:hoffhyp}
Let $m > hb$ and suppose that $\nu$ is a $b$-balanced product measure on $[m]^n$
and ${\cal F}_1,\dots,{\cal F}_h \sub [m]^n$
with $\prod_{j=1}^h \nu({\cal F}_j)
> 2^h b/(m-hb) > 0$.
Then ${\cal F}_1,\dots,{\cal F}_h$
cross contain an $h$-matching.
\end{lemma}

\begin{proof}
We show the following statement by induction on $h$:
if $(x^1,\dots,x^h) \in ([m]^n)^h$ is distributed
as $\nu^h$ conditioned on ${\sf agr}(x^j,x^{j'}) = 0$
whenever $j \ne j'$ then
$\mb{P}(x^1 \in {\cal F}_1,\dots, x^h \in {\cal F}_h)
\ge \nu({\cal F}_1) \dots \nu({\cal F}_h) - 2^h b/(m-hb)$.
The case $h=1$ is trivial.

For the induction step, as in the proof of
Lemma \ref{lem:hoffman} we consider the
product Markov chain $T$ on $[m]^n$
where each $T_i$ is the Markov chain on $[m]$
with transition probabilities $(T_i)_{xx}=0$
and $(T_i)_{xy} = \nu_i(y)/(1-\nu_i(x))$ for $y \ne x$.
We also consider $y^1,\dots,y^{h-1},x^h$,
where $x^h$ is chosen according to $\nu$
and each $y^j$ is chosen independently according
to $\nu$ conditioned on ${\sf agr}(y^j,x^h) = 0$.
We write $f_j$ for the characteristic functions
of ${\cal F}_j$ for $j \in [h]$. We have
\[\mb{E}[f_1(y^1) \dots f_{h-1}(y^{h-1})f_h(x^h)]
= \mb{E}_x [ Tf_1(x) \dots Tf_{h-1}(x) f_h(x) ]
=  \mb{E}f_1 \dots \mb{E}f_h
+ \sum_{\es \ne S \sub [h-1]} \mb{E} g_S,\]
where $g_S(x) = \prod_{i \in S} (Tf_i-\mb{E}f_i)(x)
\prod_{i \in [h-1] \sm S} f_i(x)$.
For each such $S$ we fix some $s \in S$
and write $g_S(x) = (Tf_s-\mb{E}f_s)(x)h_S(x)$.
As $Tf_s-\mb{E}f_s = T(f_s-\mb{E}f_s)$
and $\mb{E}(f_s-\mb{E}f_s)=0$,
as in the proof of Lemma \ref{lem:hoffman}
we have the spectral bound
\[ \|Tf_s-\mb{E}f_s\|_2 \le \frac{b/m}{1-b/m}
= \frac{b}{m-b}. \]
Then $|\mb{E}g_S(x)| \le b/(m-b)$ by Cauchy-Schwarz,
so $\mb{E}[f_1(y^1) \dots f_{h-1}(y^{h-1})f_h(x^h)]
\ge \mb{E}f_1 \dots \mb{E}f_h - 2^{h-1}b/(m-b)$.

Now we write
$\mb{P}(x^1 \in {\cal F}_1,\dots, x^h \in {\cal F}_h)
= \mb{E} \prod_{j=1}^h f_j(x^j) = \mb{E}_x f_h(x)
\mb{E}[ \prod_{j=1}^{h-1} f_j(x^j) \mid x^h=x]$.
For each $x$ we apply the induction hypothesis
to $f_1,\dots,f_{h-1}$ on
$\{x \in [m]^n: {\sf agr}(x,x^h)=0\}$,
which is isomorphic to $[m-1]^n$,
according to the product measure $\nu[x^h]$ with each
$\nu[x^h]_i(a) = \nu_i(a)/(1-\nu_i(x^h_i))
\le \frac{b/m}{1-b/m} = \frac{b}{m-b}$,
so $\nu[x^h]$ is $b'$-balanced,
where $b' = b(m-1)/(m-b)$.
As $b'/(m-1-(h-1)b') = b/(m-hb)$,
by induction hypothesis
$\mb{E}[ \prod_{j=1}^{h-1} f_j(x^j) \mid x^h=x]
\ge \mb{E}[ \prod_{j=1}^{h-1} f_j(y^j) \mid x^h=x]
- 2^{h-1} b/(m-hb)$, so
$\mb{E} \prod_{j=1}^h f_j(x^j) \ge \mb{E}_x f_h(x) \big [
\mb{E}[ \prod_{j=1}^{h-1} f_j(y^j) \mid x^h=x]
- 2^{h-1} b/(m-hb) \big ]
\ge \mb{E}f_1 \dots \mb{E}f_h - 2^h b/(m-hb)$.
\end{proof}

For moderate alphabets, we have the following
generalised form of our earlier lemma on
fixed agreements between uncapturable families:
we show that uncapturable families
cross contain any configuration.

\begin{thm} \label{thm:uncapagree+}
For any configuration ${\cal H}$ of size $h$ and
$s,k \in \mb{N}$ there are $r,m_0,N \in \mb{N}$
such that if $m \ge m_0$, $n \ge N\log m$
and ${\cal A}_j \sub [m]^{[n] \sm R_j}$
are $(r,m^{-k})$-uncapturable with $|R_j| \le k$ for $j \in [h]$
then there is a realisation $y^1,\dots,y^h$ of ${\cal H}$
with $y^j = x^j_{[n] \sm T}$ for some
$x^j \in {\cal A}_j$ for $j \in [h]$,
where $T = \bigcup_j R_j$.
\end{thm}

\begin{proof}
We follow the proof of Theorem \ref{thm:uncapagree}.

{\bf Step 1: Globalness.}
We define ${\cal A}^t_j$ for $j \in [h]$,
$0 \le t \le h$ as follows.
Initially all ${\cal A}^0_j={\cal A}_j$.
At step $t \in [h]$ we apply
Lemma~\ref{lem:make_global} to ${\cal A}^{t-1}_t$,
which will have $\mu({\cal A}^{t-1}_t) \ge m^{-k}$,
with $\gamma = m^{-1/10}$ and $r/100kh$ in place of $r$
we obtain ${\cal A}^t_t = ({\cal A}^{t-1}_t)_{R_t'\ra \aA_t'}$
that is $(r/100kh,\mu({\cal A}^{t-1}_t)/\gamma)$-global
with $\mu({\cal A}^t_t) \geq \mu({\cal A}^{t-1}_t)$,
where $|R'_1| \le
 (r/100kh) \log_{1/\gamma}(1/\mu({\cal A}^{t-1}_t)) \le r/10h$.
For each $j \in [h] \sm \{t\}$ we let
${\cal A}^t_j =  {\cal A}^{t-1}_j \sm
\bigcup_{i\in R_t'}D_{i\ra \aA_t'(i)}$.
Then uncapturability implies the above assumption
$\mu({\cal A}^{t-1}_t) \ge m^{-k}$.
By  Claim~\ref{claim:global_to_uncap}, each ${\cal A}^t_t$
is $(\gamma m/4, \mu({\cal A}^t_t)/2)$-uncapturable,
so $\mu({\cal A}^h_t)\geq \half\mu(\mathcal{A}^t_t)$,
which implies that $\mathcal{A}^h_t$
is $(r/100kh,2\mu(\mathcal{A}^h_t)/\gamma)$-global.

{\bf Step 2: Fairness.}
As $n \ge N\log m$ and $N$ is large, each
$\mu({\cal A}^h_j) \ge \half m^{-k} \ge e^{-n/C}$,
where $C=C(s,1/2h)$ is as in Proposition \ref{prop:fairness}.
Consider uniformly random
${\bf L} \in \tbinom{[n]\sm \bigcup_j(R_j\cup R_j')}{\ell}$
and let ${\bf z}_1,\dots,{\bf z}_\ell \in [m]^{\bf L}$
be a uniformly random copy of ${\cal H}$.
By Proposition~\ref{prop:fairness} each
$\mb{P}[ \mu((\mathcal{A}^h_j)_{{\bf L}\ra {\bf z}_j})
\geq \half \mu(\mathcal{A}^h_j) ] \ge 1-1/2h-o(1)$.
Thus we can fix $L$ and $z_1,\dots,z_\ell$ so that all
${\cal C}_j = ({\cal A}^h_j)_{L \to z_j}$
have $\mu({\cal C}_j) \ge \half \mu({\cal A}^h_j)$,
so are $(r/100kh,4\mu(\mathcal{C}_j)/\gamma)$-global.

{\bf Step 3: Expansion.}
We define ${\cal C}^t_j \sub [m_t]^n$
for $j \in [h]$, $0 \le t \le h$ as follows.
Initially all ${\cal C}^0_j={\cal C}_j$ and $m_0=m$.
At step $t \in [h]$ we apply
Lemma~\ref{lem:sharp_thresh_gen}
with $\eps=1/4kh$ and $b=b_t=4^{2^t}$,
to ${\cal C}^{t-1}_t$,
which will have $\mu^{\pi_{t-1}}(\mathcal{C}^{t-1}_t)
\geq \tfrac{1}{8}m^{-k}$,
obtaining $\pi_t \in \Pi_{m_{t-1},m_t,b_t}$
with $m_t = \OO_k(m_{t-1})$,
$\aA_t'' \in [m_t]^{R_t''}$, where
$R_t'' \sub [n]\sm(R_t\cup R_t'\cup L)$ with
$|R_t''| < O_k(\log m) \ll n$, such that ${\cal C}^t_t :=
({\cal C}^{t-1}_t)^{\pi_t}_{R_t'' \to \aA_t''}$
has $\mu^{\pi_t}({\cal C}^t_t) \ge 2m^{-1/2h}$.
For each $j \in [h] \sm \{t\}$ we let
${\cal C}^t_j = \pi_t({\cal C}^{t-1}_j)
\sm \bigcup_{i \in R_t''} D_{i \to \aA_t''}$.

For $j>t$ we can write
${\cal C}^t_j = \pi_{\circ t}({\cal X})$,
where $\pi_{\circ t} = \pi_t \circ \dots \circ \pi_1$
and ${\cal X} = {\cal C}_j \sm \bigcup_{t' \le t}
\bigcup \{ D_{i \to a}: i \in R_{t'}'',
(\pi^{\circ t'}_i(a)=(\aA_{t'}'')_i \} $.
As $\mu({\cal C}_j)$ is
$(r/100kh,4\mu(\mathcal{C}_j)/\gamma)$-global,
it is $(\gamma m/8, \mu({\cal C}_j)/2)$-uncapturable,
so $\mu({\cal X}) \ge \mu({\cal C}_j)/2
\geq \tfrac{1}{8}m^{-k}$.
By Claim~\ref{claim:glue_basic_biasing}
this implies the above assumption
$\mu^{\pi_{t-1}}(\mathcal{C}^{t-1}_t)
\geq \tfrac{1}{8}m^{-k}$.
At the end of the process, each
$\mu^{\pi_{\circ h}}(\mathcal{C}^h_j)
\geq \mu^{\pi_{\circ h}}({\cal C}^j_j)
- O_{k,h}(m^{-1}\log m) \ge m^{-1/2h}$.

{\bf Step 4: Generalised Hoffman bound.}
By averaging, we can choose restrictions
${\cal G}_j \sub [m_1]^R$ of ${\cal C}^h_j$ for $j \in [h]$
where $R  = L \cup \bigcup_j (R_j\cup R_j'\cup R_j'')$
such that all $\nu({\cal G}_j) \geq m^{-1/2h}$,
where $\nu=\mu^{\pi_{\circ h}}$ is $b_h$-balanced.
By construction, the elements of ${\cal G}_j$
are of the form $\pi_{\circ h}(x^j_{[n] \sm R})$
where $x^j \in {\cal A}_j$ form a copy of ${\cal H}$
on $L$ and have no other agreements in
$R \sm \bigcup_j (R_j\cup R_j'\cup R_j'')$.
As $\prod_j \nu({\cal G}_j) \ge m^{-1/2}
> 2^h b_h/(m_h-hb_h) > 0$,
by Lemma \ref{lem:hoffhyp}
we can find a cross matching
in ${\cal G}_1,\dots,{\cal G}_h$,
which corresponds to $x^j \in {\cal A}_j$
such that $y^j = x^j_{[n] \sm T}$ realise ${\cal H}$.
\end{proof}

We conclude this section with a junta approximation
result for configurations over large alphabets,
where for simplicity we restrict attention
to flat configurations with no kernel.
For this case we obtain a result
that is analogous to our junta approximation
result in terms of `crosscuts'
of expanded hypergraphs in \cite{KLLM}.

First we give the appropriate definition
of the crosscut for configurations.
Let ${\cal H}$ be an $\ell$-configuration of size $h$.
The crosscut $\sS({\cal H})$ is the minimum number $s$
such that there is a collection $\bigcup{\cal D}$ of $s$
co-dimension $1$ subcubes such that ${\cal H}\subseteq \bigcup{\cal D}$,
among all collection ${\cal D}$ of $s$ co-dimension $1$
and each edge $e\in\mathcal{H}$ is contained in exactly
one subcube in $\mathcal{D}$.
Note that $\sS({\cal H})>1$ if and only if
${\cal H}$ has no kernel, i.e.\ $K({\cal H})=\es$.

\begin{thm} \label{thm:K0}
For every $\eta>0$
and flat configuration ${\cal H}$ with no kernel,
there is $C$ such that if $m,n>C$ and
$\mathcal{F} \sub [m]^n$ is ${\cal H}$-free,
then there is a collection ${\cal D}$
of fewer than $\sS({\cal H})$
subcubes of co-dimension $1$ such that
$\mu({\cal F}\sm \bigcup {\cal D}) \leq \eta/m$.
\end{thm}

\begin{proof}
Let $\mathcal{F}\sub[m]^n$ be ${\cal H}$-free,
where ${\cal H}=\{e_1,\dots,e_h\}$
is an $\ell$-configuration
with parts $(U_1,\dots,U_\ell)$
and $K({\cal H})=\es$.

First we consider moderate alphabet sizes,
i.e.\ $n\geq N\log m$, with $m,N \gg h,\ell$.
We can assume that ${\cal F}$
is $(r,m^{-2})$-capturable,
with $h,\ell \ll r \ll m$,
otherwise we find ${\cal H}$ by Theorem \ref{thm:uncapagree+},
applied with all ${\cal A}_j={\cal F}$.
Thus we find a collection ${\cal J}$ of at most $r$
subcubes of co-dimension $1$ such that
$\mu({\cal F} \sm \bigcup {\cal J}) \le m^{-2}$.
Let ${\cal D}$ be the set of $D_{i \to a} \in {\cal J}$
such that $\mu({\cal F}_{i \to a}) \ge \eta/2r$.
Then $\mu(\bigcup {\cal J} \sm \bigcup {\cal D})
< \eta/2m$, so it suffices to show $|{\cal D}|<\sS({\cal H})$.

Suppose for a contradiction
that $|{\cal D}| \ge \sS({\cal H})$.
Then by definition ${\cal D}$ contains a copy of ${\cal H}$,
without loss of generality realised
on coordinate set $[\ell]$ by injections
$\phi_i:U_i \to V_i = \{(i,a): a \in [m]\}$,
such that for each $j \in [h]$ there is
$D_{i^j \to a^j} \in {\cal D}$ such that
$\phi_{i^{j'}}(e_j \cap U_{i^{j'}})= (i_{j'}, a^{j'})$ iff $j=j'$.

Let $C$ be the set of $i \in [\ell]$ such that
$U_i$ contains a vertex in the centre of ${\cal H}$.
For $i \in C$ let $c_i$ be the vertex of $U_i$ in the
centre of ${\cal H}$ (which is unique by flatness).
We may assume for any $i \in C$ that ${\cal D}$
either contains $D_{i \to \phi_i(c_i)}$
or does not contain any $D_{i \to a}$;
indeed, if ${\cal D}$ does not contain
$D_{i \to \phi_i(c_i)}$ then each $D_{i \to a}$
is $D_{i^j \to a^j}$ for at most one $j \in [h]$,
so we can obtain an alternative realisation
replacing $\phi_i$ by $\phi'_i:U_i \to V_{i'}$
for some new $i' \in [m]$, where
$\phi'_i(v)=(i',a)$ whenever $\phi_i(v)=(i,a)$.

Let $I$ be the set of all $i \in [n]$
such that ${\cal D}$ contains some $D_{i \to a}$.
We claim that we can fix $y^j \in [m]^I$ for $j \in [h]$
such that (a) $\mu({\cal F}_{I \to y^j}) \ge \eta/3r$ for all $j\in[h]$,
(b) $y^j_{i^j} = a^j$ for all $j\in [h]$, and
(c) for all $j\in [h]$, $j'\neq j$, if $i^{j'} =  i^j$
then $y^j_i \ne y^{j'}_i$ for all $i\neq i^j$, and otherwise
$y^j_i \ne y^{j'}_i$ for all $i\in I$ (in words, $y^j$ and $y^{j'}$ may only
agree on $i^j$ if the $i$'s corresponding to $j,j'$ coincide, and must disagree on any
other coordinate).
To see this we apply a greedy algorithm.
To define $y^j$, we consider
\[
{\cal G}_j = {\cal F}_{i^j \to a^j} \sm
 \bigcup_{\substack{j'<j, i\in I\\ (i, y^{j'}_i)\neq (i^j, a^j) } }
 \left(\sett{x\in [m]^{n}}{x_i = y_i^{j'}}\right)_{i^j\rightarrow a^j},
\]
which has $\mu({\cal G}_j) \geq \mu({\cal F}_{i^j \to a^j}) - \frac{h\card{I}}{m}
\geq \frac{\eta}{2r} - \frac{h\ell}{m} > \frac{\eta}{3r}$.
By averaging we can fix a restriction ${\cal F}_{I \to y^j}$
of ${\cal G}_j$ with at least this measure,
so the claim holds.

It remains to show that ${\cal G}_1,\dots,{\cal G}_h$
cross contain the configuration ${\cal H}'$
obtained from ${\cal H}$ by deleting the parts
corresponding to $I$. As in the proof of
Lemma \ref{corr:cros_contain_hyp},
by Proposition \ref{prop:fairness}
we can reduce to the case that
${\cal H}'$ is a matching,
which holds by Lemma \ref{lem:hoffhyp}.
Thus the theorem holds for
moderate alphabet sizes,

Now we consider huge alphabets, i.e.\
$n\geq n_0$ and $m\geq 2^{n/N}$,
where $K,n_0 \gg N$.
We let ${\cal D}$ be
the set of all dictators $D_{i \to a}$ such that
$\mu({\cal F}_{i \to a}) > \eta^2/n^2$.
Similarly to the proof of (a)
in Theorem \ref{thm:approx_huge_m}
we have $|{\cal D}|<\sS({\cal H})$.
Let ${\cal F}' = {\cal F} \sm \bigcup {\cal D}$.
It suffices to show $\mu({\cal F}')<\eta/m$.

Suppose $\mu({\cal F}') \ge \eta/m$.
Similarly to the proof of (a)
in Theorem \ref{thm:approx_huge_m}
we fix $i^* \in [n]$ with
$|{\cal F}'|/hn \le |\pl_{i^*}({\cal F}')|$
and partition $\pl_{i^*}({\cal G}')$
into $({\cal F}'_a: a \in [m])$ so that
$\sum_a \mu({\cal F}'_a) = \mu(\pl_{i^*} {\cal G}')
\ge \mu({\cal G}') m/hn \ge \eta/hn$.
By definition of ${\cal J}$ and repeated merging
we can form ${\cal F}_1,\dots,{\cal F}_h$ of the form
${\cal F}_i = \cup_{a \in S_i} {\cal F}'_a$ with each
$\mu({\cal F}_i) \in (\eta^2/n^2,2\eta^2/n^2)$.
However, these cross contain ${\cal H}$
by Lemma \ref{corr:cros_contain_hyp}.
\end{proof}

\section{Concluding remarks}

An open problem is to decide whether our
main theorem holds for the binary alphabet $m=2$.
Here our junta approximation method cannot work,
as the (conjectural) extremal examples are not juntas:
they are balls depending on all coordinates.
Despite this, it is still plausible that a result
can be obtained by a stability method,
by adapting the methods of \cite{KeevashLongVIP}
in proving stability for Katona's intersection theorem.

Another natural open problem is to obtain an infinitary version
of our main theorem. Say $A \sub \mb{R}^n$ is $(t-1)$-avoiding
if it contains no pair $x,y$ with $|\{i: x_i=y_i\}|=t-1$.
What is the maximum possible Hausdorff dimension of $A$?
At first one might think that the answer is $n-t$,
and that this would follow from our theorems for large finite alphabets
via a standard limiting argument if one assumes that $A$ is closed.
One must make some assumption on $A$ for
any non-trivial result, as there are pathological examples
of $A \sub \mb{R}^n$ of Hausdorff dimension $n$ in which
any distinct $x$, $y$ have $x_i \ne y_i$ for all $i \in [n]$.
However, even when $A$ is closed there are some surprises.
For example, although it is not hard to see that a $1$-avoiding set
in $[m]^3$ has size $O(m)$, there is a closed $1$-avoiding set
$A \sub \mb{R}^3$ with Hausdorff dimension $2$:\
this can be achieved by $A = \{(x,f(x),f(x)): x \in [0,1]\}$
for a suitably pathological continuous function $f$.

\subsection*{Acknowledgments}
We thank Ben Green for helpful remarks regarding
the infinitary forbidden intersection problem.

\bibliographystyle{plain}
\bibliography{ref}

\end{document}